\documentclass[10pt,a4paper, reqno]{article}
\usepackage{graphicx}
\usepackage{amsmath, color, verbatim}
\usepackage[colorlinks=true]{hyperref}

\usepackage[numbers,sort&compress]{natbib}
\usepackage{amsfonts, mathtools, enumitem} 
\usepackage{enumitem}
\usepackage{titlesec}
\usepackage[hyperref, dvipsnames]{xcolor}
\hypersetup{
  colorlinks=true,
}

\usepackage{amssymb}
\usepackage{amsthm}
\usepackage{thmtools, thm-restate}
\usepackage[left=2cm,right=2cm,top=2cm,bottom=2cm]{geometry}
\DeclareMathOperator*{\esssup}{ess\,sup}

\newcommand{\norm}[2]{\left\lVert #1\right\rVert_{#2}}

\newcommand{\md}{\partial^\bullet}

\newtheorem{theorem}{Theorem}[section]
\newtheorem{defn}[theorem]{Definition}
\newtheorem{lem}[theorem]{Lemma}
\newtheorem{prop}[theorem]{Proposition}

\newtheorem{ass}[theorem]{Assumption}

\newtheorem{remark}[theorem]{Remark}
\newtheorem{eg}[theorem]{Example}
\newcommand{\R}{\mathbb{R}}

\renewcommand{\bar}{\overline}

\newcommand{\grad}{\nabla}
\newcommand{\sgrad}{\nabla_g}

\newcommand{\cts}{\hookrightarrow}
\newcommand{\ctsDense}{\xhookrightarrow{d}}
\newcommand{\ctsCompact}{\xhookrightarrow{c}}

\makeatletter
\newcommand*{\T}{%
  {\mathpalette\@T{}}%
}
\newcommand*{\@T}[2]{%
  \raisebox{\depth}{$\m@th#1\intercal$}%
}
\makeatother

\titleformat{\part}
  {\centering\Large\scshape}{}{1em}{Part \thepart: }

\newcommand{\W}{\mathbb W} 
\newcommand{\N}{\mathbb N}
\theoremstyle{definition}
\newtheorem{definition}[theorem]{Definition}

\allowdisplaybreaks

  \numberwithin{figure}{section}
\numberwithin{table}{section}

\numberwithin{equation}{section}
\numberwithin{theorem}{section}

\begin{document}
\hypersetup{
  urlcolor     = blue, 
  linkcolor    = Bittersweet, 
  citecolor   = Cerulean
}
\title{Function spaces, time derivatives and compactness for  evolving families of Banach spaces with applications to PDEs}
\author{Amal Alphonse\thanks{Weierstrass Institute, Mohrenstrasse 39, 10117 Berlin, Germany ({\tt alphonse@wias-berlin.de})} 
\and Diogo Caetano\thanks{Mathematics Institute, University of Warwick, Coventry CV4 7AL, United Kingdom ({\tt Diogo.Caetano@warwick.ac.uk})}  \and Ana Djurdjevac\thanks{Institut f\"ur Mathematik, Freie Universit\"at Berlin, Arnimallee 6, 14195 Berlin, Germany ({\tt adjurdjevac@zedat.fu-berlin.de})}  \and Charles M. Elliott\thanks{Mathematics Institute, University of Warwick, Coventry CV4 7AL, United Kingdom ({\tt C.M.Elliott@warwick.ac.uk})} }
\maketitle

\begin{abstract}
    We develop a functional framework suitable for the treatment of partial differential equations and variational problems on evolving families of Banach spaces. We propose a definition for the weak time derivative that does not rely on the availability of a Hilbertian structure and explore conditions under which spaces of weakly differentiable functions (with values in an evolving Banach space) relate to classical Sobolev--Bochner spaces. An Aubin--Lions compactness result is proved. We analyse concrete examples of function spaces over time-evolving spatial domains and hypersurfaces for which we explicitly provide the definition of the time derivative and verify isomorphism properties with the aforementioned Sobolev--Bochner spaces. We conclude with the proof of well posedness for a class of nonlinear monotone problems on an abstract evolving space (generalising the evolutionary $p$-Laplace equation on a moving domain or surface) and identify some additional problems that can be formulated with the setting developed in this work. 
\end{abstract}
\tableofcontents

\section{Introduction}
In this paper, we provide a theory and analysis of time-dependent function spaces suitable for posing and solving evolutionary variational problems on families of time-evolving Banach spaces. We further demonstrate our theory via examples and applications of partial differential equations on moving domains and surfaces.


By way of illustration, for each  $t\ge 0$, let $H(t)$ be a Hilbert space and $X(t)$ be a Banach space with dual $X^*(t)$ such that \[X(t) \subset H(t) \subset X^*(t)\]
is  a Gelfand triple. We say that $H(t)$ is the \emph{pivot space}. 
Let   $A(t)\colon X(t) \to X^*(t)$ be an elliptic operator and $\dot u$ an appropriate time derivative (to be defined later) of $u$. With this, we can consider the abstract problem
\begin{equation}\label{eq:abstractEquation}
\begin{aligned}
\dot u(t) + A(t) u(t) &= f(t) &&\text{in $X^*(t)$},\\
u(0) &= u_0 &&\text{in $H(0)$}.
\end{aligned}
\end{equation}
One possible weak formulation concept for this problem would ask for the solution to satisfy
\[\int_0^T \langle \dot u(t), \eta(t) \rangle_{X^*(t),X(t)} + \langle A(t)u(t), \eta(t) \rangle_{X^*(t),X(t)} = \int_0^T \langle f(t), \eta(t) \rangle_{X^*(t), X(t)}\]
for every appropriate test function $\eta$, as well as a given initial condition. To make this precise, one needs to specify 
\begin{enumerate}[label=(\roman*)]
    \item the exact function spaces that the solutions lie in, \vspace{-0.2cm}
    \item how to define the time derivative in an abstract evolving Banach  space setting,
    \vspace{-0.2cm}
    \item the properties of the above-mentioned spaces and objects that allow for analysis (e.g. existence of solutions) to be performed.
\end{enumerate}
Our motivation comes from the study of partial differential equations on moving or evolving domains and manifolds. Such equations  have received considerable attention in part due to their wide applicability in the biological and physical sciences. We mention applications in biomembranes \cite{venkataraman2011modeling}, cell interactions \cite{AlpEllTer17}, cardiovascular biomechanics \cite{liu2020fluid}, fluid mechanics \cite{vcanic2020moving}, chemotaxis \cite{EllStiVen12}, to name but a few. In addition to modelling aspects, the analysis \cite{AlpEllSti15a,  AlpEllFractional, AlpEllTer17, AlpEll15, sauer1970existence, Cortez, MR2409519, KnobKrech, djurdjevac2017advection, crauel2010stochastic} and numerics and simulation  \cite{EllRan12, EllRan15, MR2877965, MR1699243, olshanskii2014error, zbMATH07063744,elliott2017unified, long2012fluid} of such problems is challenging and an active area of research.

In the case that $X(t)$ is a Hilbert space, such issues have been considered. In particular, in \cite{AlpEllSti15a} an abstract framework for the formulation and well posedness of solutions of equations of the form \eqref{eq:abstractEquation} was provided for linear parabolic problems in the Hilbert triple setting; for this, Lions-type solution spaces $\W^{p,q}(X,X^*)$ (referring to the set of $p$-integrable functions that have values in $X(t)$ with $q$-integrable weak time derivatives with values in $X^*(t)$) were defined and rigorously justified to have certain properties that are necessary for the existence theory. See also \cite{AlpEllSti15b} for several concrete examples of applications of this theory.

In this work, our setup involves not necessarily Gelfand triples but in fact more general families of Banach spaces \[X(t) \subset Y(t)\] with no intermediate inner product structure available. As there is no pivot space to work with, the formulation and properties of the weak time derivative and evolving function spaces become more complicated. It is the aim of this paper to provide the theoretical background for constructing these spaces in the fully Banach space setting, to study their properties, and to provide examples that will cover most cases of interest to practitioners working with evolutionary variational problems on moving domains and surfaces. We will also provide an Aubin--Lions type compactness result (a tool widely used in the study of nonlinear problems) for these spaces. A crucial point in achieving the Aubin--Lions result (as well as other results and properties) is an intermediary result in which we give conditions under which the space
    $\W^{p,q}(X,Y)$ 
is isomorphic to the standard Sobolev--Bochner space (or Lions space)
\begin{equation*}
    \mathcal{W}^{p,q}(X_0,Y_0) := \{ u \in L^p(0,T;X_0) : u' \in L^q(0,T;Y_0)\},
\end{equation*}
where $X_0:=X(0)$ and $Y_0 := Y(0)$. Expending effort in achieving this isomorphism property is worthwhile since it has the advantage of allowing for a simple transferral of the properties of $\mathcal{W}^{p,q}(X_0,Y_0)$ onto the time-evolving version $\W^{p,q}(X,Y)$. In particular, it leads to a relatively straightforward proof for the extension of the standard Aubin--Lions result to the evolving setting.

\medskip

\noindent In summary, the novelty of the work is the following:
\begin{enumerate}
    \vspace{-.2cm}
    \item[(1)] we consider and define weak time derivatives in a fully Banach space setting (separability and reflexivity are not assumed); with no inner product or Gelfand triple structure to aid us, the formulation of such a time derivative is non-trivial and requires care and justification;
    \vspace{-.2cm}
    \item[(2)] we provide conditions that can be checked to ensure the isomorphism with equivalence of norms between the standard Sobolev--Bochner space $\mathcal{W}^{p,q}(X_0,Y_0)$ and the evolving Sobolev--Bochner spaces $\W^{p,q}(X,Y)$ under consideration in this paper;
    \vspace{-.2cm}
    \item[(3)] we provide an Aubin--Lions result also in this generality (with no restriction needed for the evolving spaces to be related to domains or manifolds);
    \vspace{-.2cm}
    \item[(4)] we study a number of concrete examples involving function spaces of moving domains and surfaces that fit our abstract framework;
    \vspace{-.2cm}
    \item[(5)]  we prove existence and uniqueness of solutions to a monotone first-order evolution equation (of the form \eqref{eq:abstractEquation}) in a Gelfand triple setting using the theory developed in this paper.
\end{enumerate}
Under the assumptions on the evolution of the spaces in this paper, it is always possible to pull back equations such as \eqref{eq:abstractEquation} onto a reference space $X_0$ and apply standard theory on fixed spaces once the relevant assumptions have been verified. However, our approach --- which enables the problem to be treated directly in its natural formulation --- offers a certain elegance and simplicity and is also of use in numerical and finite element analysis \cite{elliott2017unified} on moving domains/surfaces (in addition to being an interesting mathematical problem in its own right). Furthermore, pulling back onto a reference domain nonetheless requires the checking of regularity of the resulting coefficients in order to apply standard theory and the analogue of this is performed for some rather general cases in \S \ref{sec:examples}, which we believe has a wide appeal for a variety of problems on moving domains and surfaces.

\paragraph{Organisation of the paper.} The paper is split into two parts. Part I focuses on the abstract theory and Part II contains applications of the theory. Beginning in \S \ref{sec:timedep}, we define and study properties of the evolving Bochner spaces $L^p_X$ and their dual spaces. We move onto defining a weak time derivative in \S \ref{sec:generalCase}, as well as defining spaces of functions with weak time derivatives and their relation to the standard Sobolev--Bochner spaces. We study the conditions under which the two spaces are isomorphic. Proceeding in \S \ref{sec:gelfandtriple}, we specialise the above theory to the setting where we have a Gelfand triple, which leads to a simplification in the statement of the assumptions that are required. We generalise the Aubin--Lions result to our setting in \S \ref{sec:aubinlions}, concluding Part I. Part 2 is devoted to examples and applications. In \S \ref{sec:examples}, we study several concrete examples of the abstract theory. Finally, in \S \ref{sec:application}, we provide an application to a nonlinear parabolic equation.

\paragraph{Notation and conventions}
\begin{itemize}
    \item We will always work with \emph{real} Banach spaces.
       \vspace{-.2cm}
    \item
       The action of the linear map $x^*\in X^*$ on $x\in X$ is denoted by
    $$\langle x^*,x\rangle_{X^*,X}=\langle x,x^*\rangle_{X,X^*}.$$
     \vspace{-.8cm}
    \item Continuous, dense and compact embeddings of spaces will be denoted by $\cts,$ $\ctsDense$ and $\ctsCompact$ respectively.
    \vspace{-.2cm}
    \item We will usually leave out the differential in integrals, i.e., we write $\int_0^T f(t)$ rather than $\int_0^T f(t)\;\mathrm{d}t.$
    \vspace{-.2cm}
    \item For a function $f\colon [0,T]\to X$ onto a Banach space, we denote the difference quotient 
        \begin{align*}
            \delta_h f(t) := \dfrac{f(t+h)-f(t)}{h}. 
        \end{align*}
    \vspace{-.6cm}
    \item Given $a,b\in\R$, $a\land b := \min(a,b)$. 
    \vspace{-.2cm}
    \item The letters $p$ and $q$ will typically be used for (not necessarily conjugate) integrability exponents in $L^p$-type spaces; the conjugate of $p$ will always be denoted by $p':= p\slash (p-1)$.
    \item We write $\mathcal D(\Omega) \equiv C_c^\infty(\Omega)$ to refer to the set of infinitely differentiable functions with compact support in the open set $\Omega \subset \mathbb{R}^n$. Likewise, for a Banach space $X$,  $\mathcal{D}((0,T);X) \equiv C_c^\infty((0,T);X)$ denotes the space of smooth, compactly supported functions on $(0,T)$ with values in $X$. The dual space of $\mathcal{D}(\Omega)$ will be denoted $\mathcal{D}^*(\Omega)$, which is the space of continuous linear functionals on $\mathcal{D}(\Omega)$ (i.e., the space of distributions) endowed with the strong dual topology. The space $\mathcal{D}^*((0,T);X)$ will stand for the space of continuous linear mappings from $\mathcal{D}(0,T)$ into $X$, i.e.,
    \[\mathcal{D}^*((0,T);X) = \mathcal{L}(\mathcal{D}(0,T),X),\]
see \cite{MR2961861} for further details. 

\end{itemize}

\part{Theory}

This part is devoted to establishing the abstract theory necessary for the analysis of function spaces and the treatment of partial differential equations on evolving surfaces or bulk domains. We will assume familiarity with the classical theory of standard Bochner spaces $L^p(0,T;X)$; useful texts on this topic are \cite{Zeidler, Rou05, Boy13, DautrayLions, MR2961861}.

\section{Time-evolving Bochner spaces $L^p_X$}\label{sec:timedep}
The aim in this section is to define a generalisation of the Bochner spaces $L^p(0,T;X)$ to describe integrable (in time) functions with values in a Banach space that itself depends on time. In \cite{AlpEll15, AlpEllSti15a}, two of the present authors defined and studied properties of spaces $L^p_{X}$ given a sufficiently smooth parametrised family of Banach spaces $\{X(t)\}_{t \in [0,T]}$. These spaces were generalisations to the abstract Banach space setting of spaces introduced by Vierling in \cite{Vie11} in the context of Sobolev spaces on evolving surfaces. We now recall (and in some places, refine) the theory in \cite{AlpEll15} so that the presentation is essentially self-contained.

For each $t \in [0,T]$, let $X(t)$ be a real Banach space with $X_0 := X(0)$ and let 
\[\phi_t\colon X_0 \to X(t)\]
be a bounded, linear, invertible map with inverse 
\[\phi_{-t}\colon X(t) \to X_0.\]
It follows that the inverse is also bounded. These maps `link' the time-dependent spaces and we call $\phi_t$ the \textit{pushforward} map and $\phi_{-t}$ the \textit{pullback} map. We assume these satisfy the following properties.

\begin{ass}[Compatibility]\label{ass:def_compatibility}
Suppose that
\begin{enumerate}[label=({\arabic*})]
\item $\phi_0$ is the identity,
\vspace{-0.2cm}\item there exists a constant $C_X$ independent of $t \in [0,T]$ such that
\begin{equation*}
\begin{aligned}
\norm{\phi_t u}{X(t)} &\leq C_X\norm{u}{X_0} &&\forall u \in X_0,\\
\norm{\phi_{-t} u}{X_0} &\leq C_X\norm{u}{X(t)} &&\forall u \in X(t),
\end{aligned}
\end{equation*}
\vspace{-0.4cm}
\item for all $u \in X_0$, the map $t \mapsto \norm{\phi_t u}{X(t)}$ is measurable.
\end{enumerate}
We say that the pair $(X(t), \phi_t)_t$ is \emph{compatible}.
\end{ass}
In what follows, we always assume that $(X(t), \phi_t)_t$ satisfies Assumption \ref{ass:def_compatibility} and we (formally) identify the family $\{X(t)\}_{t \in [0,T]}$ with the symbol $X$. 
\begin{remark}
Under this compatibility assumption, note that for $s,t\in [0,T]$, the map $U(t,s):=\phi_s  \phi_{-t} \colon X(t)\to X(s)$ defines a 2-parameter semigroup in the sense of \cite[Definition 1.1.1]{Pot10}.
\end{remark}
Let us define the disjoint union
\[\mathcal{X}_T := \bigcup_{t \in [0,T]}\!\!\!\! X(t) \times \{t\}.\]
\begin{defn}[The space $L^p_X$]
For $p \in [1,\infty]$, define the space
\begin{align*}
L^p_X &:= \left\{u:[0,T] \to \mathcal{X}_T, \;\; t \mapsto (\hat u(t), t)  \;\; \mid \;\; \phi_{-(\cdot)} \hat u(\cdot) \in L^p(0,T;X_0 )\right\}.
\end{align*}
Identifying $u(t) = (\hat u(t), t)$ with $\hat u(t)$, endow the space with the norm
\begin{equation*}
\begin{aligned}
\norm{u}{L^p_X} &:=
\begin{cases}
\left({\int_0^T \norm{u(t)}{X(t)}^p}\right)^{\frac 1 p} &\text{for $p \in [1,\infty)$},\\ 
\esssup_{t \in [0,T]}\norm{u(t)}{X(t)} &\text{for $p=\infty$}.
\end{cases}
\end{aligned}
\end{equation*}
\end{defn}
\begin{theorem}\label{thm:LpX}
Under Assumption \ref{ass:def_compatibility},  $L^p_X$ is a Banach space. If $X$ is a family of Hilbert spaces, then $L^2_X$ is a Hilbert space with the canonical inner product
\[(u,v)_{L^2_X} := \int_0^T (u(t),v(t))_{X(t)}.\]
Furthermore, $L^p(0,T;X_0)$ and $L^p_X$ are isomorphic via $\phi_{(\cdot)}$ with an equivalence of norms:
\begin{equation}
\begin{aligned}
\frac{1}{C_X}\norm{u}{L^p_X} &\leq \norm{\phi_{-(\cdot)}u(\cdot)}{L^p(0,T;X_0)} \leq C_X\norm{u}{L^p_X}&&\text{for all $u \in L^p_X$}.\label{eq:equivalenceBetweenLpXAndStdBochner}
\end{aligned}
\end{equation}
\end{theorem}
\begin{proof}
For the first two claims, see \cite[Theorem 2.8]{AlpEllSti15a} for the Hilbertian case and  the paragraph after Definition 2.1 in \cite{AlpEll15} for the general Banach setting. The equivalence of norms is proved in \cite[Lemma 2.3]{AlpEll15}.
\end{proof}



\paragraph{Spaces of smooth functions.}The following $C^k$-type spaces will also be of use later. We start by defining,   
for $k\in \mathbb{N} \cup \{0\}$, the spaces $C^k_X$ of $k$-times continuously differentiable functions (on the closed interval $[0,T]$)
\begin{align*}
C^k_X &= \left\{\eta\colon [0,T] \to \mathcal{X}_T,\;\; t \mapsto (\eta(t), t)\;\; \mid \;\; \phi_{-(\cdot)}\eta(\cdot) \in C^k([0,T];X_0)\right\}.
\end{align*}
We will also need the space $\mathcal{D}_X$ of smooth, compactly supported functions (but now on the \emph{open} interval $(0,T)$)
\begin{align*}
\mathcal{D}_X &= \left\{\eta\colon [0,T] \to \mathcal{X}_T,\;\; t \mapsto ( \eta(t), t)\;\; \mid \;\; \phi_{-(\cdot)}\eta(\cdot) \in \mathcal D((0,T);X_0)\right\}.
\end{align*}
\subsection{Dual spaces}\label{sec:dualSpaces}
In this section, we study the dual space of $L^p_X$ for appropriate $p$. First, we shall see that given a compatible pair $(X(t), \phi_t)_{t \in [0,T]}$, we can also define the space $L^p_{X^*}$ associated to the family $\{X^*(t)\}$ by using dual maps. Indeed, denote by 
\[\phi_{-t}^*\colon X_0^*\to X^*(t)\]
the dual operator of $\phi_{-t}\colon X(t)\to X_0$. Under the condition
\begin{equation}\label{ass:measurabilityOfDual}
    t \mapsto \norm{\phi_{-t}^*f}{X^*(t)} \text{ is measurable for all $f \in X_0^*$},
\end{equation}
it is not difficult to verify that the pair $(X^*(t), \phi_{-t}^*)_{t\in [0,T]}$ is also compatible in the sense of the definition above (see \cite[Remark 2.4]{AlpEll15}). This justifies the next definition.

\begin{defn}[The space $L^p_{X^*}$]\label{defn:dual} Given a compatible pair $(X(t), \phi_t)_{t\in [0,T]}$, under \eqref{ass:measurabilityOfDual}, we define the space $L^p_{X^*}$ using the dual spaces $\{X^*(t)\}_{t\in [0,T]}$ and the dual maps $\{\phi_{-(\cdot)}^*\colon X_0^* \to X^*(t)\}$. 
\end{defn}

\begin{remark}
Note that if $X_0$ is separable, then so is $X(t)$ for every $t\in [0,T]$  and the condition \eqref{ass:measurabilityOfDual} follows from Assumption \ref{ass:def_compatibility}. 
\end{remark}

Regarding the relationship between the dual of a Bochner space and the Bochner space of the dual, recall that if $Z$ is a reflexive Banach space, then $Z^*$ is also reflexive and hence it possesses the Radon--Nikodym property, which is key to identifying the dual of $L^p(0,T;Z)$ as $L^{p'}(0,T;Z^*)$ whenever $p\neq \infty$.
\begin{theorem}[Identification of the dual of $L^p_X$ with $L^{p'}_{X^*}$]\label{thm:dualSpaceIdentification} Suppose that   the family of reflexive Banach spaces $\{X(t)\}_{t \in [0,T]}$ satisfies Assumption \ref{ass:def_compatibility}, \eqref{ass:measurabilityOfDual} holds and let $p \in [1,\infty)$. 
The dual space $(L^p_X)^*$ is isometrically isomorphic to $L^{p'}_{X^*}$ (taken as in Definition \ref{defn:dual}) with duality pairing 
\[\langle f, u \rangle_{L^{p'}_{X^*},L^p_{X}} = \int_0^T \langle f(t), u(t) \rangle_{X^*(t), X(t)}.\]
Furthermore, if $p \in (1,\infty)$, then $L^p_X$ is reflexive.
\end{theorem}
\begin{proof}
The proof follows the classical proof for the corresponding result for Bochner spaces \cite[\S IV]{MR0453964} with modifications, see Theorem 2.5 in \cite{AlpEll15}.
\end{proof}

We now establish a version of the fundamental theorem of calculus of variations for the evolving space setting. The proof is simple but we provide it to illustrate the kind of argument required when working with these kinds of spaces.
\begin{lem}\label{lem:evolvingfundamental}
If $u\in L^1_X$ is such that
\begin{align*}
    \int_0^T \langle u(t), \eta(t)\rangle_{X(t), X^*(t)} = 0 \quad \forall \eta\in\mathcal{D}_{X^*},
\end{align*}
then $u\equiv 0$.
\end{lem}
\begin{proof}
Given $\eta\in\mathcal D_{X^*}$, by writing $ \big\langle u(t), \eta(t)\big\rangle_{X(t),  X^*(t)} = \big\langle \phi_{-t} u(t), \phi_t^* \eta(t)\big\rangle_{X_0,  X_0^*}$ and setting  $\varphi :=\phi_{(\cdot)}^*\eta(\cdot) \in \mathcal D((0,T); X_0^*)$ it follows by the arbitrariness of $\eta \in \mathcal{D}_{X^*}$ that 
\begin{align*}
    \int_0^T \big\langle \phi_{-t} u(t), \varphi(t) \rangle_{X_0, X_0^*} = 0 \quad \forall \varphi\in \mathcal D((0,T); X_0^*),
\end{align*}
from where $\phi_{-(\cdot)} u(\cdot) \equiv 0$, hence $u\equiv 0$.
\end{proof}
\begin{remark}[Relation between the Riesz maps in the Hilbert space case]
Suppose that $\{H(t)\}_{t\in [0,T]}$ is a family of Hilbert spaces compatible with a family of maps $\{\phi_t\}_{t\in [0,T]}$ as above. Let us discuss the relationship between the various Riesz isomorphisms that are present in this situation.

\medskip
 Let $\mathcal R\colon L^2_H \to L^2_{H^*}$ and $\mathcal S_t\colon H(t) \to H^*(t)$ be the associated Riesz maps. If $u,v \in L^2_H$, by definition
$t \mapsto \langle \mathcal S_tu(t), v(t)\rangle_{H^*(t), H(t)} = (u(t), v(t))_{H(t)}$ is measurable and we have
\begin{align*}
(u,v)_{L^2_H} = \int_0^T (u(t),v(t))_{H(t)} = \int_0^T \langle \mathcal S_tu(t), v(t) \rangle_{H^*(t), H(t)}
\end{align*}
but on the other hand, by Theorem \ref{thm:dualSpaceIdentification},
\begin{align*}
(u,v)_{L^2_H} = \langle \mathcal Ru, v \rangle_{L^2_{H^*}, L^2_H} = \int_0^T \langle (\mathcal Ru)(t), v(t) \rangle_{H^*(t), H(t)}.
\end{align*}
This implies that
\[\mathcal Ru = \mathcal S_{(\cdot)}u(\cdot) \quad \text{in $L^2_{H^*}$}\]
and thus $(\mathcal Ru)(t) = \mathcal S_tu(t) \in H^*(t)$ for almost all $t$.  This suggests that identifying $H(t)$ with $H(t)^*$ forces $L^2_H$ to be identified with $L^2_{H^*}$ and vice versa.
\end{remark}

\section{Time derivatives in evolving spaces}\label{sec:generalCase}
Having defined Bochner-type spaces to deal with evolving families of Banach spaces, we focus in this section on defining a notion of a weak time derivative for functions in such spaces. We recall the definition of a weak time derivative on a fixed setting: given $X\hookrightarrow Y$, a function $v'\in L^1(0,T; Y)$ is the weak time derivative of $v\in L^1(0,T;X)$ if 
\begin{align}\label{eq:classweakder}
    \int_0^T v' \varphi = - \int_0^T v \varphi' \quad \forall\varphi\in\mathcal D_{Y^*}.
\end{align}
Firstly, since the pullbacks of functions in $C^1_X$ (recall the definition in \S \ref{sec:timedep}) are differentiable, we are able to define a time derivative for such functions with a simple and natural formula.
\begin{defn}
A function $u\in C_X^1$ has a \emph{strong time derivative} $\dot u\in C_X^0$ defined by
\begin{align}\label{eq:defnStrongMaterialDerivative}
    \dot u(t) = \phi_t^X \left( (\phi^X_{-t} u)' \right),
    \end{align}
where $(\phi^X_{-t} u)'$ denotes the classical weak derivative of $\phi^X_{-t} u$ as in \eqref{eq:classweakder}.
\end{defn}
Evidently, this time derivative depends on the maps $\{\phi_t^X\}$. We will sometimes also use the notation $\md u$ in place of $\dot u$. A similar definition could be stated for higher order derivatives but we will not need it in this text. 
\begin{remark}
This definition implies the following simple \emph{transport property}: if $u\in C^1_X$ is of the form $u = \phi_t^X \eta$ for some $\eta\in X_0$, then $\dot u = 0$.
\end{remark}
The aim now is to look for a weaker notion of time derivative than the strong time derivative. Motivated by the integration by parts formula \eqref{eq:classweakder}, we expect the definition of the weak time derivative to be similar to the non-moving setting but in view of the fact that the spaces here are evolving, we expect an additional term in its definition. Such a weak time derivative was defined in the setting of Hilbert triples $X(t) \subset H(t) \subset X^*(t)$ (with each space a Hilbert space) in \cite{AlpEllSti15a}. Here, we aim to drop the assumption of an existing pivot Hilbert space and define the weak time derivative in the full generality of the classical Banach space setting. 

For the rest of this section, we work under the following assumptions: 

\begin{ass}\label{ass:sect4_assumption}
We fix families
\begin{align*}
    \big(X(t), \phi_t^X\colon X_0\to X(t)\big)_{t\in [0,T]} \qquad \text{ and } \qquad \big(Y(t), \phi_t^Y\colon Y_0\to Y(t)\big)_{t\in [0,T]}, \quad\quad
\end{align*}
where $X_0:= X(0)$ and $Y_0:=Y(0)$, satisfying Assumption \ref{ass:def_compatibility} and such that the Banach spaces $X(t)\cts Y(t)$ continuously for all $t\in [0,T]$. 
\end{ass}

\begin{remark}
We do \emph{not} assume that $\phi_t^X = \phi_t^Y|_{X_0}$! Doing so 
would lead to a simplified setting in what follows, see Remark \ref{rem:afterWeakDerDefn} (ii) for more details.
\end{remark}

\subsection{Definition and properties of the weak time derivative}
For a function $u\in L^p_X$, we wish to define an appropriate concept of a weak time derivative $\dot u\in L^q_Y$ motivated by the usual so-called transport formula in the non-moving setting. Taking a test function $\eta\in\mathcal D_{Y^*}$, we expect
\begin{align}
    \dfrac{d}{dt} \left\langle u(t), \eta(t)\right\rangle_{X(t), \, X^*(t)} = \left\langle \dot u(t), \eta(t)\right\rangle_{Y(t), \, Y^*(t)} + \left\langle u(t), \dot \eta(t)\right\rangle_{X(t), \, X^*(t)} + \text{ extra term,}\label{eq:somethingToShow}
\end{align}
where the extra term accounts for the time-dependence of the duality pairing. Integrating over $[0,T]$, and using the fact that $\eta$ is compactly supported, this would lead to a weak derivative formula of the integration by parts type, with an extra term which we now must identify. 
To isolate the effect of time-dependency that the evolution of the spaces induces in the associated duality product, we make the following assumption.

\begin{ass}\label{ass:generalcaseChanges}
We assume that
\begin{itemize}
    \item[(i)] the map 
    \[t\mapsto \big\langle \phi_t^X u_0, (\phi_{-t}^Y)^* v_0 \big\rangle_{X(t),  X^*(t)}\] is continuously differentiable for each fixed $u_0\in X_0$, $v_0\in Y_0^*$;
    \item[(ii)] for all $t\in [0,T]$, the map 
     \[ X_0 \times Y_0^* \ni(u_0, v_0)\mapsto \frac{\partial}{\partial t}\big\langle \phi_t^X u_0, (\phi_{-t}^Y)^* v_0 \big\rangle_{X(t),  X^*(t)}\]
    is continuous; 
    \item[(iii)] there exists $C>0$ such that, for almost all $t\in [0,T]$ and all $u_0\in X_0$, $v_0\in Y_0^*$, 
    \begin{equation*}
    \left|\frac{\partial}{\partial t}\big\langle \phi_t^X u_0, (\phi_{-t}^Y)^* v_0 \big\rangle_{X(t),  X^*(t)}\right| \leq C \, \|u_0\|_{X_0} \, \|v_0\|_{Y_0^*}.
    \end{equation*}
\end{itemize}
\end{ass}
Here, we have used  the fact that $Y^*(t)\cts X^*(t)$ continuously. It is convenient to define the bilinear form $\lambda(t;\cdot,\cdot)\colon X(t) \times Y^*(t) \to \mathbb{R}$ by
\begin{equation}
\lambda(t;u,v) := \frac{\partial}{\partial t}\big\langle \phi_t^X u_0, (\phi_{-t}^Y)^* v_0 \big\rangle_{X(t),  X^*(t)}\bigg\rvert_{(u_0,v_0) = (\phi_{-t}^X u, (\phi_t^Y)^* v)}.\label{eq:definitionOfLambdaNew}    
\end{equation}
This leads us to the following generalization of the weak time derivative for functions that take values in evolving Banach spaces. 
\begin{definition}[Weak time derivative]
We say $u\in L^1_X$ is \textit{weakly differentiable} with \textit{weak time derivative} $v\in L^1_Y$ if
\begin{align}\label{eq:genweakderiv}
\int_0^T  \langle u(t), \dot \eta(t) \rangle_{X(t),  X^*(t)}  &= -\int_0^T  \langle v(t), \eta(t) \rangle_{Y(t), Y^*(t)}  - \int_0^T \lambda(t; u(t), \eta(t)) \quad \forall \eta \in \mathcal{D}_{Y^*}.
\end{align}
\end{definition}In \S \ref{sec:L2pivotspace}, we will see that this definition recovers the well-established definition of the weak material derivative in the Gelfand triple setting where the pivot space is an $L^2$-type space on an evolving domain or surface. 

We note that this generalises to the fully Banach space case the definition in the work \cite{AlpEllSti15a} co-authored by the first and final authors where all spaces were assumed to be Hilbert spaces in a Gelfand triple setting.
\begin{remark}\label{rem:afterWeakDerDefn}
\begin{enumerate}[label=({\roman*})]
\item The first two parts of Assumption \ref{ass:generalcaseChanges} imply that $\lambda$ is a Carath\'eodory function, thus for $u\in L^1_X$ and $v\in L^1_{Y^*}$, the superposition map
    $t\mapsto \lambda (t; u(t), v(t)) \text{ is measurable.}$ 
       \item The expression for $\lambda$ suggests that our definition could lead to problems in the case $Y(t):=X(t)$ with the same maps $\phi_t^Y \equiv \phi_t^X$, in which case $\lambda \equiv 0$ and the extra term in the definition of a weak time derivative would vanish. But this is indeed the case for smooth functions $u\in C^1_X$. To wit, omitting the exponent in $\phi_t = \phi^X_t$, we have, for any $\eta\in\mathcal{D}_{X^*}(0,T)$,
        \begin{align*}
            \int_0^T \langle \dot u(t), \eta(t) \rangle_{X(t),  X^*(t)} &= \int_0^T \langle \phi_t (\phi_{-t} u(t))', \eta(t)\rangle_{X(t),  X^*(t)} = -\int_0^T \langle u(t), \dot \eta(t) \rangle_{X(t),  X^*(t)}.
        \end{align*}
    Hence, our setting includes the case $Y(t)\equiv X(t)$ and the calculation above shows that $u\in L^p_X$ is weakly differentiable (in the sense of \eqref{eq:genweakderiv}) if and only if $\phi_{-t}^X u$ is weakly differentiable in the classical sense, and it holds that $\dot u = \phi_t^X (\phi_{-t}^X u)'$.
    \item  Note that the above is different to the case where there is a Hilbert triple framework in place and the derivative has sufficient smoothness to lie in $L^q_X$: in such a case, we would still get a non-zero $\lambda$ term! That is, 
    \[\{u \in L^p_X : \dot u \in L^q_X\}\] and
    \[(X(t), H(t), X^*(t)) \text{ is a Gelfand triple; } \quad \{u \in L^p_X : \dot u \in L^q_{X^*}\cap L^q_X\}\]
    are fundamentally different since the derivative $\dot{(\cdot)}$ in each set is a different operator; in particular, the second is defined through the pivot space. One should take care to not confuse the two.
    \end{enumerate}
\end{remark}
By a simple application of Lemma \ref{lem:evolvingfundamental}, we can prove the next result.

\begin{prop}[Uniqueness of weak derivatives]\label{prop:genuniqueness}
Suppose $u\in L^1_X$ has weak time derivatives $v_1,v_2\in L^1_Y$. Then $v_1\equiv v_2$. 
\end{prop}


\begin{prop}[Strong derivatives are also weak derivatives]\label{prop:genstrongisweak}
Let $u\in C^1_X$ and $\dot u\in C^0_X$ be its strong time derivative. Then $u$ is also weakly differentiable with weak time derivative $\dot u$.
\end{prop}
We provide the proof later on page~\pageref{proof:siw} since we will need an additional result to prove it.

\begin{remark}
Proposition \ref{prop:genstrongisweak} shows that our notion of a weak derivative is indeed a generalisation of the strong derivative \eqref{eq:defnStrongMaterialDerivative}. It would perhaps seem more natural to define the weak derivative by pulling back to the reference domain with the maps $\phi_{-t}^X$, differentiating in the usual (weak) sense, and pushing forward with $\phi_t^Y$. Even though this is the case when $Y(t)\equiv X(t)$ (as per Remark \ref{rem:afterWeakDerDefn}), this approach does not lead to the same definition as above when the spaces do not coincide. On this topic, note further that:
\begin{itemize}
    \item[(i)] If $u\in L^p_X$ is weakly differentiable in the sense we defined above, it is not necessarily the case that $\phi_{-t}^X u$ has a weak time derivative (in the usual sense). Conditions under which this is true will be explored in \S \ref{sec:criteriaESE}.
    \item[(ii)] Even if $u\in L^p_X$ is such that $\phi_{-t}^X u$ is weakly differentiable, then a simple calculation shows that the function $\phi_t^Y (\phi_{-t}^X u)'$ does not satisfy an expression of the form \eqref{eq:genweakderiv} unless $Y(t) \equiv X(t)$. Indeed, it is easy to check that
    \begin{align*}
        \int_0^T \left\langle \phi_t^Y (\phi_{-t}^X u)', \eta\right\rangle_{Y(t), \, Y^*(t)} 
        &= -\int_0^T \left\langle \phi_t^Y \phi_{-t}^X u, \dot \eta\right\rangle_{Y(t), \, Y^*(t)} \qquad \forall \eta\in\mathcal{D}_{Y^*}.
    \end{align*}
\end{itemize}
\end{remark}
\subsection{Transport formula for smooth functions and further remarks}
Having defined a notion of weak time derivative, we now demonstrate that a transport formula of the form \eqref{eq:somethingToShow} holds for sufficiently smooth functions. 
\begin{lem}\label{lem:generalprop}
Let Assumption \ref{ass:generalcaseChanges} hold. Given $\sigma_1\in C^1_X$, $\sigma_2\in C^1_{Y^*}$, the map $t\mapsto \langle\sigma_1(t), \sigma_2(t)\rangle_{X(t),  X^*(t)}$ is absolutely continuous and for almost all $t\in [0,T]$,
        \begin{align*}
        \dfrac{d}{dt} \langle\sigma_1(t), \sigma_2(t)\rangle_{X(t),  X^*(t)} = \langle \dot \sigma_1(t), \sigma_2(t)\rangle_{X(t),  X^*(t)} + \langle\sigma_1(t), \dot \sigma_2(t)\rangle_{X(t),  X^*(t)} + \lambda(t; \sigma_1(t), \sigma_2(t)).
        \end{align*}
\end{lem}
For the proof, it becomes convenient to introduce the following notation and definitions.
\begin{definition}\label{defn:generalOperators}
For $t\in [0,T]$, we define the following objects:
\begin{itemize}
\item[(i)] the evolution of the duality pairing,
\begin{align*}
&\pi(t; \cdot, \cdot)\colon X_0 \times Y_0^* \to \R, \quad \pi(t;u,v) := \big\langle \phi_t^X u, (\phi_{-t}^Y)^* v \big\rangle_{Y(t),  Y^*(t)} = \big\langle \phi_t^X u, (\phi_{-t}^Y)^* v \big\rangle_{X(t),  X^*(t)};\\
&\hat{\lambda}(t;\cdot, \cdot)\colon X_0 \times Y_0^* \to \R, \quad \hat{\lambda}(t; u, v) := \frac{\partial}{\partial t} \, \pi(t; u, v).
\end{align*}
\item[(ii)] the map $\Pi_t\colon X_0\to Y_0$ defined by 
\[\Pi_t u := \phi_{-t}^Y \, \phi_t^X u,\]
which satisfies
\begin{align*}
\big\langle \Pi_t u, v\big\rangle_{Y_0,  Y_0^*} = \pi(t; u, v).
\end{align*}
\item[(iii)] the map $\hat\Lambda(t)\colon X_0 \to Y_0^{**}$ defined by 
   \begin{align*}
   \label{eq:defbiglambda}
            \qquad \langle \hat\Lambda(t) u_0, v_0\rangle_{Y_0^{**}, Y_0^*} := \hat\lambda(t; u_0, v_0). 
    \end{align*}

\end{itemize}
\end{definition}
The fact that $\pi(t; \cdot, \cdot)$ is defined over $X_0\times Y_0^*$ is motivated by the discussion preceding the definition of the weak time derivative above, allowing for the formulation in \eqref{eq:genweakderiv} with test functions in $\mathcal D_{Y^*}$. We see that $\lambda$, defined in \eqref{eq:definitionOfLambdaNew}, is the pushforward of the bilinear form $\hat{\lambda}$:
\begin{align*}
&\lambda(t; u,v) = \hat{\lambda}(t; \phi_{-t}^{X} u, (\phi_{t}^Y)^* v).
\end{align*}
For convenience, let us write Assumption \ref{ass:generalcaseChanges} in terms of the notation of $\pi$ and $\hat \lambda$.
\begin{remark}\label{rem:generalcase}
Assumption \ref{ass:generalcaseChanges} is equivalent to the following:
\begin{itemize}
    \item[(i)] the map $t\mapsto \pi(t; u, v)$ is continuously differentiable for each fixed $u\in X_0$, $v\in Y_0^*$ with derivative
    \begin{align*}
        \hat{\lambda}(t;\cdot, \cdot)\colon X_0 \times Y_0^* \to \R, \quad \hat{\lambda}(t; u, v) := \frac{\partial}{\partial t} \, \pi(t; u, v);
    \end{align*}
    \item[(ii)] for all $t\in [0,T]$, the map $(u, v)\mapsto \hat\lambda(t; u, v)$ is continuous; 
    \item[(iii)] there exists $C>0$ such that, for almost all $t\in [0,T]$ and all $u\in X_0$, $v\in Y_0^*$, 
    \begin{equation*}
    |\hat{\lambda}(t;u, v)| \leq C \, \|u\|_{X_0} \, \|v\|_{Y_0^*}.
    \end{equation*}
\end{itemize}
\end{remark}
With this, we obtain that $\hat\Lambda$ has a dual operator \[\hat\Lambda(t)^*\colon Y_0^{***}\to X_0^*.\]
It is worthwhile noting that for $u \in L^p(0,T;X_0)$, the map $t \mapsto \Pi_tu(t)$ is measurable from $(0,T) \to Y_0$ since $\phi_{(\cdot)}^X u(\cdot) \in L^p_X \subset L^p_Y$ and by definition of compatibility and \eqref{eq:equivalenceBetweenLpXAndStdBochner}, $\phi_{-(\cdot)}^Y\colon L^p_Y \to L^p(0,T;Y_0)$, ensuring measurability of the composition map.
\begin{remark}[The Gelfand triple case]\label{rem:observationsHatBetc}
Some observations regarding the definition above are timely.
\begin{itemize}
    \item[(i)] Consider $Y(t):=X^*(t)$ with maps $\phi_t^Y = (\phi_{-t}^X)^*$, and suppose that there exists a family of Hilbert spaces $H(t)$ such that $X(t)\ctsDense H(t)$. We suppose $H(t)$ evolves with maps $\phi_t^H$ satisfying $\phi_t^H|_{X_0}=\phi_t^X$ and that we have a Gelfand triple structure $X(t)\hookrightarrow H(t)\hookrightarrow X^*(t)$. In this case, the definition of the operator $\pi$ above becomes, for $u\in X_0$, $v\in X_0$,
        \begin{align*}
            \pi(t; u, v) = \langle \phi_t^X u, \phi_{t}^X v\rangle_{X(t),  X^*(t)} = \langle \phi_t^H u, \phi_t^H v\rangle_{X(t),  X^*(t)} = (\phi_t^H u, \phi_t^H v)_{H(t)},
        \end{align*}
    and this definition can be uniquely extended to $H_0\times H_0$ by density of $X_0$ in $H_0$. This also shows that the map $\Pi_t$ satisfies
    \begin{align*}
        \Pi_t\colon X_0 \to H_0 \subset X_0^*, \quad \Pi_t u = (\phi_t^H)^A \phi_t^X u,
    \end{align*}
    where $(\cdot)^A$ stands for the Hilbert adjoint. 
    We can extend the latter map to $H_0$ as the operator (still labelled $\Pi_t$)
    \begin{align*}
        \Pi_t\colon H_0 \to H_0 \subset X_0^*, \quad \Pi_t u = (\phi_t^H)^A 
        \phi_t^H u.
    \end{align*}
    In particular, when $X(t)$ is also a Hilbert space, we recover\footnote{In \cite{AlpEllSti15a}, the notations $T_t$ and $\hat b$ were used in place of $\Pi_t$ and $\pi$ respectively.} the definitions in \cite{AlpEllSti15a}.
    \item[(ii)] In the setting above, observe that the definition of the operator $\pi$, and consequently of $\Pi_t$ and $\hat \lambda$, can be expressed involving the flows and inner product solely of the intermediate Hilbert space $H(t)$, and as such all of these are independent of the base space $X(t)$ that is chosen.
\end{itemize}
\end{remark}



\begin{proof}[Proof of Lemma \ref{lem:generalprop}]
Let us first show that given $\hat{\sigma}_1\in C^1([0,T]; X_0)$ and $\hat{\sigma_2}\in C^1([0,T]; Y_0^*)$, the map $t\mapsto \pi(t; \hat{\sigma}_1(t), \hat{\sigma}_2(t))$ is in $C^1([0,T])$ and that for all $t\in [0,T]$,
        \begin{align}\label{eq:derivb}
        \dfrac{d}{dt} \pi(t; \hat{\sigma}_1(t), \hat{\sigma}_2(t)) = \pi(t; \hat{\sigma}_1'(t), \hat{\sigma}_2(t)) + \pi(t; \hat{\sigma}_1(t), \hat{\sigma}_2'(t)) + \hat{\lambda}(t; \hat{\sigma}_1(t), \hat{\sigma}_2(t)).
        \end{align}
To see this, start by considering for $h>0$  the difference quotient 
\begin{align*}
    \delta_h \pi (t; \hat \sigma_1(t), \hat\sigma_2(t))
    &= \dfrac{\pi(t + h; \hat \sigma_1(t+h), \hat\sigma_2(t+h)) - \pi (t; \hat \sigma_1(t+h), \hat\sigma_2(t+h))}{h}\\
    &\quad + \pi\left(t; \delta_h \hat\sigma_1(t), \hat\sigma_2(t+h)\right) + \pi\left(t; \sigma_1(t), \delta_h\hat\sigma_2(t)\right).
\end{align*}
The continuity of $\pi$ with respect to the second and third variables and the regularity of $\hat\sigma_1$, $\hat\sigma_2$ imply that, for all $t\in [0,T]$, the sum of the last two terms on the right-hand side above converges, as $h\to 0$, to 
\begin{align*}
 \pi(t; \hat\sigma_1'(t), \hat\sigma_2(t)) + \pi(t; \hat\sigma_1(t), \hat\sigma_2'(t)).
\end{align*}
We now use Assumption \ref{ass:generalcaseChanges} (or equivalently, the conditions in Remark \ref{rem:generalcase}) to
establish that for almost all $t\in [0,T]$,
\begin{align*}
    \dfrac{\pi(t + h; \hat \sigma_1(t+h), \hat\sigma_2(t+h)) - \pi (t; \hat \sigma_1(t+h), \hat\sigma_2(t+h))}{h} \to \dfrac{\partial \pi}{\partial t} (t; \hat\sigma_1(t), \hat\sigma_2(t)) = \hat\lambda(t; \sigma_1(t), \sigma_2(t)).
\end{align*}
Indeed, let us fix $t\in [0,T)$ and $h>0$ sufficiently small so that $t+h \leq T$. We have, using the absolute continuity of $s \mapsto \pi(s;u,v)$ for fixed $u$ and $v$,
\begin{align*}
    &\dfrac{\pi(t + h; \hat \sigma_1(t+h), \hat\sigma_2(t+h)) - \pi (t; \hat \sigma_1(t+h), \hat\sigma_2(t+h))}{h} - \hat\lambda(t; \sigma_1(t), \sigma_2(t)) = \\
    &\hskip 1cm = \dfrac{1}{h} \int_t^{t+h} \hat\lambda(s; \hat\sigma_1(t+h), \hat\sigma_2(t+h)) \;\mathrm{d}s - \hat\lambda(t; \hat\sigma_1(t), \hat\sigma_2(t)) \\
    &\hskip 1cm = \mathrm{I} + \mathrm{II} + \mathrm{III},
\end{align*}
where we have set 
\begin{align*}
    \mathrm{I} &:= \dfrac{1}{h} \int_t^{t+h} \hat\lambda(s; \hat\sigma_1(t+h), \hat\sigma_2(t+h)) - \hat\lambda(s; \hat\sigma_1(t), \hat\sigma_2(t+h) )\;\mathrm{d}s = \int_0^T \hat\lambda \left(s; \delta_h \hat\sigma_1(t), \hat\sigma_2(t+h)\right)\chi_{[t,t+h]}(s)\;\mathrm{d}s,
\end{align*}
\begin{align*}
    \mathrm{II} :=  \dfrac{1}{h} \int_t^{t+h} \hat\lambda(s; \hat\sigma_1(t), \hat\sigma_2(t+h)) - \hat\lambda(s; \hat\sigma_1(t), \hat\sigma_2(t) )\;\mathrm{d}s = \int_0^T\hat\lambda \left(s; \hat\sigma_1(t),\delta_h\hat\sigma_2(t)\right)\chi_{[t,t+h]}(s)\;\mathrm{d}s,
\end{align*}
and
\begin{align*}
    \mathrm{III}:= \dfrac{1}{h} \int_t^{t+h} \hat\lambda(s; \hat\sigma_1(t), \hat\sigma_2(t)) \;\mathrm{d}s - \hat\lambda(t; \hat\sigma_1(t), \hat\sigma_2(t)).
\end{align*}
Now observe that, for sufficiently small $h$, the integrands in $\mathrm{I}$ and $\mathrm{II}$ are uniformly bounded and converge pointwise to $0$, and so the Dominated Convergence Theorem implies that $\mathrm{I}, \mathrm{II} \to 0$ as $h\to 0$. Since for fixed $u, v$ the map $s\mapsto \hat\lambda(s; u,v)$ is integrable, it follows from Lebesgue's Differentiation Theorem that also $\mathrm{III}\to 0$ as $h\to 0$ 
for all $t\in [0,T]$ proving that $t\mapsto \pi(t; \sigma_1(t), \sigma_2(t))$ has a continuous derivative and is thus $C^1([0,T])$. We can reason similarly for $t\in (0,T]$ and $h>0$ with $t-h\ge 0$, and this will show \eqref{eq:derivb}. From here, the claimed statement can be obtained directly by taking $\hat\sigma_1(t):= \phi_{-t}^X \sigma_1(t)$ and $\hat\sigma_2(t) := \big(\phi_t^Y\big)^* \sigma_2(t)$. 
\end{proof}
With this transport formula at hand, we can prove our earlier claim that strong derivatives are also weak derivatives.

\begin{proof}[Proof of Proposition \ref{prop:genstrongisweak}]\label{proof:siw}
We start by observing that $\dot u\in C^0_X\subset C^0_Y\subset L^1_Y$. Given $\eta\in \mathcal{D}_{Y^*}$, we have
\begin{align*}
    \int_0^T \langle \dot u(t), \eta(t)\rangle_{Y(t), Y^*(t)} 
    &= -\int_0^T \pi(t; \phi_{-t}^X u(t), \left((\phi_{t}^Y)^* \eta(t)\right)') - \int_0^T \hat{\lambda}(t; \phi_{-t}^X u(t), (\phi_{t}^Y)^* \eta(t)) \quad \tag{by Lemma \ref{lem:generalprop}} \\
    &= -\int_0^T \langle u(t), \dot\eta(t) \rangle_{X(t),  X^*(t)} - \int_0^T \lambda(t; u(t), \eta(t)),
\end{align*}
which proves the claim. \qedhere
\end{proof}
\subsection{A characterisation of the weak time derivative}
We come now to an alternative characterisation of the weak time derivative related to the derivative of a duality product, which turns out to be useful in various situations (e.g. in the mechanics of applying the Galerkin method for existence of solutions to nonlinear PDEs, see \S \ref{sec:application}), cf. \cite[Lemma 1.1, \S III]{Tem77} for the non-moving case.
First, let us introduce some notation. For a Banach space $Z$, we denote by $\mathcal J_Z\colon Z\to Z^{**}$ the (linear and bounded) canonical injection into the double dual:
        \begin{align*}
            \left\langle \mathcal{J}_Z u, f \right\rangle_{Z^{**}, \, Z^*} := \left\langle f, u \right\rangle_{Z^*, \, Z}, \quad \forall f \in Z^*, \, u\in Z.
        \end{align*}
In part to avoid working with double and triple duals, it sometimes becomes useful to assume that
\begin{align}\label{ass:rangeOfJ}
 \hat\Lambda(t) u \in \mathrm{Range}(\mathcal{J}_{Y_0}), \quad \forall t\in [0,T], \,\, u\in X_0.
\end{align}
\begin{remark}\label{rem:newRemOnAss}
Regarding the assumption \eqref{ass:rangeOfJ}, note that
\begin{itemize}
    \item it is automatically satisfied if $Y_0$ is reflexive;
    \vspace{-0.2cm}
    \item the meaning of the assumption is that 
    \begin{equation}\label{eq:meaningOfJAss}
\forall u \in X_0,\; \exists y \in Y_0 : \langle \mathcal{J}_{Y_0}y, f \rangle_{Y_0^{**}, Y_0^*} = \langle f, y \rangle_{Y_0^*,  Y_0} = \langle \hat \Lambda (t) u, f \rangle_{Y_0^{**}, Y_0^*} \quad \forall f \in Y_0^*;        
    \end{equation}
\vspace{-0.6cm}
\item denoting the map $u \mapsto y$ in \eqref{eq:meaningOfJAss} by $y=Lu$, we can write
\[\langle f, Lu \rangle_{Y_0^*,  Y_0} = \langle \hat \Lambda (t) u, f \rangle_{Y_0^{**}, Y_0^*},\]
which suggests that $\hat \Lambda(t)$ can be identified as a map $\hat \Lambda(t)\colon X_0 \to Y_0$ and this is indeed what we shall do below whenever the assumption is in force. 
\end{itemize}
\end{remark}
\begin{prop}[Characterisation of the weak time derivative]\label{prop:temamType}
Assume \eqref{ass:rangeOfJ}. Let $u \in L^p_X$ and $g \in L^q_Y$. Then $\dot u = g$ if and only if 
\begin{equation}\label{eq:anotherAltForDerivative}
\frac{d}{dt}\langle u(t), (\phi_{-t}^Y)^* v  \rangle_{X(t),X^*(t)} = \langle g(t), (\phi_{-t}^Y)^* v \rangle_{Y(t), Y^*(t)} + \langle \Lambda(t)u(t), (\phi_{-t}^Y)^* v \rangle_{Y(t), Y^*(t)} \quad \forall  v \in Y_0^*.
\end{equation}
\end{prop}
\begin{proof}
Making the substitution $\eta = (\phi_{-(\cdot)}^Y)^* v$ for arbitrary $v \in Y_0^*$  in \eqref{eq:anotherAltForDerivative}, we find by definition of the  weak time derivative,
\[\int_0^T \psi'(t)\langle u(t), \eta(t) \rangle_{X(t),X^*(t)} = -\int_0^T \psi(t)\left(\langle g(t), \eta(t) \rangle_{Y(t), Y^*(t)} + \langle \Lambda(t)u(t), \eta(t) \rangle_{Y(t), Y^*(t)}\right)\quad \forall \psi \in \mathcal{D}(0,T).\]
Collecting terms, we may write this as
\begin{align*}
0
= \int_0^T \langle \psi'(t)\phi_{-t}^Y u(t) + \psi(t)\phi_{-t}^Y(g(t) + \Lambda(t)u(t)), v \rangle_{Y_0,Y^*_0}.
\end{align*}
Bringing the integral inside the first part of the duality pairing above, we get 
\[\frac{d}{dt}\phi_{-t}^Y u(t) = \phi_{-t}^Y (g(t)+\Lambda(t)u(t)).\]
Now, as $\phi_{-(\cdot)}^Y(g+\Lambda u) \in L^1(0,T;Y_0)$, this is equivalent to
\[\int_0^T \langle \phi_{-t}^Y u(t), \xi'(t) \rangle_{Y_0,Y_0^*} = -\int_0^T \langle \phi_{-t}^Y (g(t)+\Lambda(t)u(t)), \xi(t) \rangle_{Y_0,Y^*_0} \quad \forall \xi \in \mathcal D((0,T);Y_0^*).\]
Setting $\varphi := (\phi_{-(\cdot)}^Y)^*\xi\in \mathcal{D}_{Y^*}$ so that $\dot \varphi = (\phi_{-(\cdot)}^Y)^*\xi'$, we can pushforward the duality products above to obtain
\begin{align*}
\int_0^T \langle  u(t), \dot \varphi(t) \rangle_{Y(t),Y^*(t)} 
&= -\int_0^T \langle g(t), \varphi(t) \rangle_{Y(t),Y^*(t)} - \int_0^T \lambda (t; u(t), \varphi(t)).
\end{align*}
This being valid for every $\varphi \in \mathcal D_{Y^*}$ shows that $\dot u = g$ by definition. The reverse implication follows since every step in the above proof is an equivalence.
\end{proof}
\subsection{Evolving Sobolev--Bochner spaces}

Having defined an appropriate notion of weak time derivative, we consider in this section the definition and properties of evolving Sobolev--Bochner spaces, which are the spaces in which solutions to parabolic PDEs (on evolving spaces) typically lie in. These can be considered to be the time-evolving versions of $\mathcal{W}^{p,q}(X,Y)$ defined as 
\begin{align}
    \mathcal W^{p,q}(X_0,Y_0) = \{u\in L^p(0,T; X_0) \colon u' \in L^q(0,T;Y_0)\}.
\end{align}
To reiterate, we again are enforcing Assumption \ref{ass:sect4_assumption}.

\begin{defn}[The space $\W^{p,q}(X,Y)$]
For $p, q\in [1,\infty]$, define the space
\begin{align*}
\W^{p,q}(X,Y) &:= \left\{ u \in L^p_X \mid \dot u \in L^q_{Y}\right\}\quad\text{with norm}\quad     \|u\|_{\W^{p,q}(X,Y)} := \|u\|_{L^p_X} + \|\dot u\|_{L^q_Y}.
\end{align*}
\end{defn}

\begin{prop}
The space $\W^{p,q}(X, Y)$ is a Banach space.
\end{prop}
\begin{proof}
Let $\{u_n\}$ be a Cauchy sequence in $\W^{p,q}(X, Y).$ It follows that $u_n \to u$ in $L^p_X$ to some $u$ and $\dot u_n \to w$ in $L^q_Y$ to some $w$. We have for all $\eta\in \mathcal{D}_{Y^*}$, 
\begin{align*}
\int_0^T  \langle \dot u_n(t), \eta(t) \rangle_{Y(t), Y^*(t)}   &= - \int_0^T  \langle u_n(t), \dot \eta(t) \rangle_{X(t),  X^*(t)}  - \int_0^T \lambda(t; u_n(t), \eta(t)).
\end{align*}
It is immediate to pass to the limit in the first two terms, and for the last one we observe that, since $\lambda$ is bilinear,
\begin{align*}
    \left|\int_0^T \lambda(t; u_n(t), \eta(t)) - \int_0^T \lambda(t; u(t), \eta(t))\right| 
    &\leq C'\, \|\eta\|_{L^\infty_{Y^*}} \|u_n-u\|_{L^1_X} \to 0.
\end{align*}
We then have
\begin{align*}
\int_0^T  \langle w(t), \eta(t) \rangle_{Y(t), Y^*(t)} &= - \int_0^T  \langle u(t), \dot \eta(t) \rangle_{X(t),  X^*(t)}  - \int_0^T \lambda(t; u(t), \eta(t)),
\end{align*}
which shows, by uniqueness of weak derivatives (Proposition \ref{prop:genuniqueness}), that $w=\dot u$. 
\end{proof}

In Theorem \ref{thm:LpX}, we saw that $\phi_{(\cdot)}$ acts as an isomorphism between the spaces $L^p(0,T;X_0)$ and $L^p_X$ with an equivalence of norms. A natural question to ask is: under which conditions does $\phi_{(\cdot)}$ act as an isomorphism between $\mathcal W(X_0,Y_0)$ and $\W(X,Y)$ with an equivalence of norms? This question will be addressed in a later section. First, let us formalise this idea and give a simple density result under such an equivalence.

\begin{defn}\label{defn:equivalence}
We say there is an \emph{evolving space equivalence} between $\W^{p,q}(X,Y)$ and $\mathcal W^{p,q}(X_0, Y_0)$ if
$$v \in \W^{p,q}(X,Y) ~~\mbox{ if and only if}~~~  \phi^X_{-(\cdot)} v(\cdot) \in \mathcal W^{p,q}(X_0, Y_0),$$ and the following equivalence of norms holds:
\[C_1\norm{\phi^X_{-(\cdot)}v(\cdot)}{\mathcal W^{p,q}(X_0, Y_0)} \leq \norm{v}{\W^{p,q}(X,Y)} \leq C_2\norm{\phi^X_{-(\cdot)}v(\cdot)}{\mathcal W^{p,q}(X_0, Y_0)}.\]
\end{defn}
We may also say that $\W^{p,q}(X,Y)$ has the \emph{evolving space equivalence property} or that $\W^{p,q}(X,Y)$ and $\mathcal{W}^{p,q}(X_0,Y_0)$ are \emph{equivalent} instead of `evolving space equivalence'.

This notion of an evolving space equivalence is important as it ensures that properties of the classical spaces $\mathcal W^{p,q}(X_0, Y_0)$ carry over to the time-dependent $\W^{p,q}(X,Y)$. As mentioned, we investigate when such an equivalence exists in \S \ref{sec:criteriaESE}. For now, we prove the following useful lemma, which contains direct generalisations of classical embedding results.

\begin{lem}\label{lem:density}
    Suppose that there exists an evolving space equivalence between $\mathcal{W}^{p,q}(X_0, Y_0)$ and $\W^{p,q}(X, Y)$. 
    \begin{itemize}
        \item[(i)] The embedding $\W^{p,q}(X, Y)\hookrightarrow C^0_Y$ is continuous.
        \item[(ii)] The space $C^1_{X}$ is dense in $\W^{p,q}(X, Y)$.
    \end{itemize}
\end{lem}

\begin{proof}
The statement (i) is a consequence of the following series of implications:
\begin{align*}
    u\in \W^{p,q}(X,Y) \Longleftrightarrow \phi_{-(\cdot)}^X u(\cdot) \in \mathcal{W}^{p,q}(X_0, Y_0) \Longrightarrow \phi_{-(\cdot)}^X u(\cdot) \in C^0([0,T]; Y_0) \Longleftrightarrow u\in C^0_Y.
\end{align*}
To prove (ii), let $u\in\W^{p,q}(X,Y)$, so that $v(\cdot):=\phi^X_{-(\cdot)}u(\cdot)\in\mathcal W^{p,q}(X_0,Y_0)$. Take a sequence $(v_n)_n\subset C^1([0,T]; X_0)$ such that $v_n\to v$ in $\mathcal W^{p,q}(X_0,Y_0)$ as $n\to\infty$. Defining $u_n(\cdot):=\phi^X_{(\cdot)} v_n(\cdot)\in C^1_{X}$, we have, due to the evolving space equivalence,
\begin{align*}
    &\|u_n-u\|_{\W^{p,q}(X,Y)} \leq C \, \|\phi_{-(\cdot)}^X u_n - \phi_{-(\cdot)}^X u\|_{\mathcal{W}^{p,q}(X,Y)} = C \, \|v_n - v\|_{\mathcal{W}^{p,q}(X,Y)} \to 0 \quad \text{ as } n\to\infty. \qedhere
\end{align*}
\end{proof}
\subsection{Differentiating the duality product: transport theorem}

In this section we state and prove a transport theorem for general functions in the abstract spaces defined above.

\begin{theorem}[Transport theorem]\label{thm:transport}
Let either
\begin{itemize}
    \item[(i)] $p \in [2,\infty]$, $u\in \W^{p,p'}(X, Y)$ and $v\in \W^{p,p'}(Y^*, X^*)$
    
    \item[]or
    
    \item[(i')] $p \in [1,\infty]$, $u\in \W^{p,p}(X, Y)$ and $v\in \W^{p',p'}(Y^*, X^*)$,
\end{itemize}
and suppose that in either case the spaces involved have the evolving space equivalence property. Then the map 
\begin{align}\label{eq:transportfunction}
    t\mapsto \big\langle u(t), v(t)\big\rangle_{X(t),\, X^*(t)}
\end{align}
is absolutely continuous and we have, for almost all $t\in [0,T]$,
\begin{align}
\dfrac{d}{dt} \langle u(t), v(t)\rangle_{X(t),  X^*(t)} = \big\langle \dot u(t), v(t)\big\rangle_{Y(t),\, Y^*(t)} + \big\langle u(t), \dot v(t)\rangle_{X(t),\, X^*(t)} + \lambda(t; u(t), v(t))\label{eq:transportTheorem}.
\end{align}
\end{theorem}

\begin{proof}
Under either of the assumptions it follows that both \eqref{eq:transportfunction} and the right-hand side of \eqref{eq:transportTheorem} define functions in $L^1(0,T)$. This is clear for case (i'), and in case (i) simply observe that $p'\leq 2\leq p$, and thus $u\in L^p_X\subset L^{p'}_X$, so \eqref{eq:transportfunction} and the last term in \eqref{eq:transportTheorem} are also integrable. 

It therefore remains to prove that the right hand side of \eqref{eq:transportTheorem} is the weak derivative of \eqref{eq:transportfunction}. But this follows by density. Indeed, take sequences $\{u_n\}_n\subset C^1_X$, $\{v_n\}_n\subset C^1_{Y^*}$ such that
\begin{align*}
    u_n\to u \,\, \text{ in } \W^{p,p'}(X,Y) \quad \text{ and } \quad v_n \to v \,\, \text{ in } \W^{p,p'} (Y^*, X^*).
\end{align*}
We then have, using Lemma \ref{lem:generalprop}, 
\begin{align*}
    \dfrac{d}{dt} \langle u_n(t), v_n(t)\rangle_{X(t),  X^*(t)} = \langle \dot u_n(t), v_n(t) \rangle_{Y(t), Y^*(t)} + \langle u_n(t), \dot v_n(t)\rangle_{X(t),  X^*(t)} + \lambda(t; u_n(t), v_n(t)).
\end{align*}
Writing this in terms of the definition of the weak derivative and then passing to the limit, we find that \eqref{eq:transportTheorem} holds in the weak sense, giving the conclusion.
\end{proof}

\begin{remark}
Let us motivate the conditions on the exponents in the statement above.
Assume that $u\in \W^{p_1, q_1}(X, Y)$ and $v\in \W^{p_2, q_2}(Y^*, X^*)$ and suppose that these spaces have the evolving space equivalence property. The displayed equations in Theorem \ref{thm:transport} above reveal that conditions on the exponents are necessary:
\begin{itemize}
    \item \eqref{eq:transportfunction} must define an integrable function, but this is the case for any exponents $p_1, q_1, p_2, q_2$, due to the extra regularity $v\in C^0_{X^*}$;
    \item the right-hand side of \eqref{eq:transportTheorem} must also be integrable: 
    \begin{itemize}
        \item the first and second terms show that we must have $L^{p_2}\subset L^{q_1'}$ and $L^{p_1}\subset L^{q_2'}$;
        \item the last term requires $L^{p_1}\subset L^{p_2'}$.
    \end{itemize}
\end{itemize}
This shows that the extra term $\lambda(t; u(t), v(t))$ --- which is not present in the classical setting --- holds us back from stating a general result for $u, v\in \W^{p,p'}$, though in some applications we can find a way around this obstacle (as we will see in Sections \ref{sec:gelfandtriple} and \ref{sec:examples}).
\end{remark}

\subsection{Criteria for evolving space equivalence}\label{sec:criteriaESE}

Here, we focus on obtaining conditions that can be checked ensuring an evolving space equivalence (see Definition \ref{defn:equivalence}) between $\mathcal{W}^{p,q}(X_0,Y_0)$ and $\W^{p,q}(X,Y)$. The main result is the next theorem which states the precise conditions required; the reader is also referred to Theorem \ref{thm:ESEGelfandTriple} for the statement (and proof) of this theorem applied to the particular case of a Gelfand triple (the setting of which results in some simplifications in the conditions that are needed). We recall the operators and bilinear forms in Definition \ref{defn:generalOperators} and we write $\mathcal{M}(0,T;Z)$ to stand for the set of Bochner measurable maps $f\colon (0,T) \to Z$ into a Banach space $Z.$
\begin{theorem}[Criteria for evolving space equivalence]\label{thm:equivalence}
Let Assumption \ref{ass:generalcaseChanges} hold, let $p,q\in [1,\infty]$ and suppose that for all $t \in [0,T]$,  the range condition 
\eqref{ass:rangeOfJ} and the following hold:
\begin{align}
     &\text{$\Pi_t\colon X_0 \to Y_0$ has a linear extension } \bar \Pi_t\colon Y_0 \to Y_0 \text{ which is bounded uniformly in $t$,}\label{ass:onrangeTt}\\
    &\bar \Pi_{(\cdot)}u \in \mathcal{M}(0,T;Y_0) \text{ for all $u \in X_0$}, \label{ass:measurabilityOnTt}\\
    &\bar \Pi_t^{-1}\colon Y_0 \to Y_0 \text{ exists and is uniformly bounded in $t$},\label{ass:TtInvExistsAndBdd}\\
    &\bar \Pi_{(\cdot)}^{-1} v \in \mathcal{M}(0,T;Y_0) \text{ for all $v \in Y_0$},
    \label{ass:measurabilityOnTtInv}\\
    &(\bar \Pi^\dagger)^{-1}\colon \mathcal{W}^{p,q}(Y_0^*, Y_0^*) \to \mathcal{W}^{p,\, p\land q}(Y_0^*, Y_0^*) \text{ where $((\bar \Pi^\dagger)^{-1} f)(t) := (\bar \Pi_t^{*})^{-1}f(t)$}.\label{ass:remainsInSpace}
\end{align}
 Then there is an evolving space equivalence between $\mathcal W^{p,q}(X_0,Y_0)$ and $\W^{p,q}(X,Y)$.

\medskip

\noindent More precisely, 
\begin{itemize}
    \item[(i)] under \eqref{ass:onrangeTt} and \eqref{ass:measurabilityOnTt}, if $u\in \mathcal{W}^{p,q}(X_0,Y_0)$, then $\phi_{(\cdot)}^X u\in \W^{p,q}(X,Y)$ and 
        \begin{align*}
            \md \phi_t^X u(t) = \phi_t^Y \bar \Pi_t u'(t).
        \end{align*}
    \item[(ii)] under \eqref{ass:rangeOfJ} and \eqref{ass:onrangeTt}--\eqref{ass:remainsInSpace},
    if $u\in \W^{p,q}(X,Y)$, then $\phi_{-(\cdot)}^X u\in \mathcal W^{p,q}(X_0,Y_0)$ and 
        \begin{align*}
            \big(\phi_{-t}^X u(t)\big)' = \bar \Pi_t^{-1} \phi_{-t}^Y \dot u(t).
        \end{align*}
\end{itemize}
\end{theorem}
\begin{remark}\label{rem:measurabilityCanBeDropped}
Regarding these assumptions, let us make the following observations.
\begin{enumerate}[label=({\roman*})]
    \item If $X_0 \ctsDense Y_0$, then \eqref{ass:measurabilityOnTt} follows immediately from \eqref{ass:onrangeTt}. Indeed, for any $y \in Y_0$, take a sequence $x_n \in X_0$ with $x_n \to y$ in $Y_0$. As $\Pi_tx_n 
= \bar \Pi_t x_n \to \bar \Pi_t y$ in $Y_0$ and the pointwise limit of measurable functions is measurable,  the claim holds. 
 \item Assumptions \eqref{ass:TtInvExistsAndBdd} and \eqref{ass:measurabilityOnTtInv} imply that $(\bar \Pi^\dagger)^{-1} \colon L^p(0,T;Y_0^*) \to L^p(0,T;Y_0^*)$ is a bounded linear operator for any $p$.
\item Assumption \eqref{ass:remainsInSpace} is analogous to the assumption on the differentiability of $\Pi_t$ (or $\pi_t$).
\item One should bear in mind that $\Pi_t$ has an inverse  only on the set $\Pi_t(X_0)$ (i.e. its range):
\begin{align*}
    X_0 \overset{\phi_t^X}{\longrightarrow} X(t) \overset{\phi_{-t}^Y|_{X(t)}}{\longrightarrow} \phi_{-t}^Y \big(X(t)\big) \subset Y_0.
\end{align*}
It is not clear that $\Pi_t(X_0)$ is closed (and hence not necessarily a Hilbert space in its own right) so the dual of $\Pi_t^{-1}$ is not well defined in general. This is why we only talk about the inverses of $\bar \Pi$ and its dual operator.
\end{enumerate}
\end{remark}
The rest of this section is dedicated to proving this result, which will be done in a number of steps. 
We begin with some preliminaries: we have the pointwise dual maps
\[\Pi_t^*\colon Y_0^* \to X_0^* \quad\text{and}\quad \bar\Pi_t^*\colon Y_0^* \to Y_0^*\]
and it is not difficult to see that for every $f \in Y_0^*$,
\begin{equation}\label{eq:dualPiAndBarPiIdentity}
\langle \bar\Pi_t^* f, x \rangle_{Y_0^*, Y_0} =\langle \Pi_t^* f, x \rangle_{X_0^*,X_0} \quad\text{whenever $x \in X_0$}.    
\end{equation}
Let us construct the Nemytskii operators
\[(\Pi u)(t) := \Pi_t u(t), \quad (\bar \Pi  u)(t) = \bar \Pi_t u(t), \quad (\Pi^\dagger f)(t) = \Pi^*_t f(t), \quad (\bar \Pi^\dagger f)(t) = \bar \Pi^*_t f(t).\]
We have that
\begin{equation*}
    \begin{aligned}
    &\Pi\colon L^p(0,T;X_0) \to L^p(0,T;Y_0),\quad &&\bar \Pi  \colon L^p(0,T;Y_0) \to L^p(0,T;Y_0),\\
    &\Pi^\dagger \colon L^p(0,T;Y_0^*) \to L^p(0,T;X_0^*), &&\bar \Pi^\dagger \colon L^p(0,T;Y_0^*) \to L^p(0,T;Y_0^*),
    \end{aligned}
\end{equation*}
are all bounded and linear for any $p$: for the first two maps this follows respectively by definition (see \S \ref{sec:generalCase}) and by \eqref{ass:onrangeTt} and \eqref{ass:measurabilityOnTt}\footnote{The measurability of the image of the latter operator follows because $\bar \Pi_{(\cdot)}(\cdot)$ is by assumption a Carath\'eodory function.}, and the latter two because the dual of a bounded linear operator is also bounded and linear with the same operator norm. Note carefully that $\Pi^\dagger$ is in general \textit{not} the same as $\Pi^*$ (which is defined as the dual of $\Pi$) since we are not necessarily in the reflexive setting (and likewise for $\bar \Pi^\dagger$)! See the next remark for more on this.

\begin{remark}\label{rem:daggerVsDual}
\begin{enumerate}[label=({\roman*})]
\item Regarding $\Pi$ and $\bar \Pi$, we know that their dual operators satisfy by definition, for any $p \in [1,\infty],$
\begin{align*}
    &\Pi^* \colon L^p(0,T;Y_0)^* \to L^p(0,T;X_0)^*\quad\text{and}\quad 
    \bar \Pi^* \colon L^p(0,T;Y_0)^* \to L^p(0,T;Y_0)^*. 
\end{align*}
If it were the case that we could identify the duals of the above Bochner spaces with the expected spaces (which we can do for example in the reflexive setting for appropriate exponents $p$), then the above can be written as
\begin{align*}
    &\Pi^* \colon L^{p'}(0,T;Y_0^*) \to L^{p'}(0,T;X_0^*) \quad\text{and}\quad 
    \bar \Pi^* \colon L^{p'}(0,T;Y_0^*) \to L^{p'}(0,T;Y_0^*), 
\end{align*}
and it is easy to see in this case that 
\[\Pi^\dagger \equiv \Pi^* \quad\text{and}\quad \bar \Pi^\dagger \equiv \bar \Pi^*.\]
 \item Assumptions \eqref{ass:TtInvExistsAndBdd} and \eqref{ass:measurabilityOnTtInv} imply that $\bar \Pi^\dagger$ has an inverse given by
\[(\bar \Pi^\dagger)^{-1}\equiv (\bar \Pi^{-1})^\dagger \quad\text{where}\quad ((\bar \Pi^{-1})^\dagger f)(t) := (\bar \Pi_t^{-1})^*f(t)=(\bar \Pi_t^{*})^{-1}f(t),\]
and furthermore, both maps $$\bar \Pi^{-1}\colon L^p(0,T;Y_0) \to L^p(0,T;Y_0) \qquad\text{and}\qquad (\bar \Pi^\dagger)^{-1}\equiv (\bar \Pi^{-1})^\dagger  \colon L^p(0,T;Y_0^*) \to L^p(0,T;Y_0^*)$$ are bounded linear operators. 
\end{enumerate}
\end{remark}
The next proposition shows that $\Pi$ and $\bar \Pi^\dagger$ take differentiable functions into differentiable functions thanks to the assumptions on the differentiability of $\pi$ that were made earlier. Even though one does not usually distinguish between an element of a Banach space and its action as an element of the corresponding double dual space, in the proofs below, to emphasise that we do not assume reflexivity of neither $X_0$ nor $Y_0$, we will always write explicitly the canonical injections $\mathcal{J}_{X_0}, \mathcal{J}_{Y_0}, \mathcal{J}_{Y_0^*}.$ 

\begin{prop}[Differentiability of $\Pi u$]\label{prop:newCrit1}
Let $p,q\in [1,\infty]$ and suppose that \eqref{ass:onrangeTt} and \eqref{ass:measurabilityOnTt} hold.
If $u\in\mathcal{W}^{p,q}(X_0, Y_0)$, then $\Pi u$ 
satisfies 
\begin{align*}
\int_0^T \left\langle \Pi_t u(t), \varphi'(t) \right\rangle_{Y_0, \, Y_0^*} 
&= - \int_0^T \left\langle \bar \Pi_t u'(t), \varphi(t)\right\rangle_{Y_0, \, Y_0^*} +  \langle \hat\Lambda(t) u(t), \varphi(t) \rangle_{Y_0^{**}, \, Y_0^*}  \quad\forall \varphi\in \mathcal D((0,T); Y_0^*).
\end{align*}
In particular, if $\hat\Lambda(t)u(t)\in \mathrm{Range}(\mathcal{J}_{Y_0})$, then $\Pi u\in \mathcal{W}^{p, \, p\land q}(Y_0, Y_0)$ with 
\begin{align}
(\Pi u)'(t) 
= \bar \Pi_t u'(t) + \mathcal{J}_{Y_0}^{-1} \hat \Lambda(t) u(t).\label{eq:toProve1}
\end{align}
%
\end{prop}

\begin{proof}
Let us take $\varphi\in \mathcal D((0, T); Y_0^*)$ and $u\in C^1([0,T]; X_0)$ to obtain, by \eqref{eq:derivb} in the proof of Lemma \ref{lem:generalprop},
\begin{align*}
    \int_0^T \left\langle \varphi'(t), \Pi_tu(t) \right\rangle_{Y_0^*, Y_0} 
    &= - \int_0^T \left\langle \bar \Pi_t u'(t), \varphi(t) \right\rangle_{Y_0, Y_0^*} + \left\langle \hat \Lambda(t) u(t), \varphi(t) \right\rangle_{Y_0^{**}, Y_0^*}.
\end{align*}
If $\hat\Lambda u$ is in the range of $\mathcal{J}_{Y_0}$, then we can write the last term above as
\begin{align*}
   \int_0^T \left\langle  \mathcal{J}_{Y_0}^{-1} \hat \Lambda(t) u(t), \varphi(t) \right\rangle_{Y_0, Y_0^*},
\end{align*}
proving \eqref{eq:toProve1} (with the right-hand side belonging to $L^{p\land q}(0,T;Y_0)$) 
for $u\in C^1([0,T]; X_0)$. The conclusion now follows from the density of $C^1([0,T]; X_0)$ in $\mathcal W^{1,1}(X_0, Y_0)$ and the continuity of $\Pi, \bar \Pi$ and $\hat \Lambda$.
\end{proof}
Note that the assumption that was needed for \eqref{eq:toProve1} above is exactly \eqref{ass:rangeOfJ}.
Now we look for a converse of Proposition \ref{prop:newCrit1}. In order to do so, we need a preparatory result in the form of the next lemma.
\begin{lem}[Differentiability of $\bar \Pi^\dagger v$]\label{lem:splitLemma}
Let $p,q\in [1,\infty]$ and suppose that \eqref{ass:onrangeTt} and \eqref{ass:measurabilityOnTt} hold. If $v\in \mathcal{W}^{p,q}(Y_0^*, Y_0^*)$, then $\bar \Pi^\dagger v\in \mathcal W^{p, \, p\land q}(Y_0^*, X_0^*)$ with
\begin{align*}
(\bar \Pi^\dagger v)'(t) 
= \bar \Pi_t^* v'(t) + \hat \Lambda(t)^* \mathcal{J}_{Y_0^*}v(t).
\end{align*}
\end{lem}
\begin{proof}
We will first prove the intermediary result that 
$\Pi^\dagger v\in \mathcal W^{p,\, p\land q}(X_0^*, X_0^*)$ with \begin{equation}\label{eq:prelim1}
(\Pi^\dagger v)'(t) = (\Pi_t^* v(t))'= \bar \Pi_t^*v'(t)+\hat\Lambda(t)^* \mathcal{J}_{Y_0^*}v(t) \quad\text{in $X_0^*$}
\end{equation}  
for $v$ taken as stated in the lemma. Indeed, approximating with $v\in C^1([0,T]; Y_0^*)$ and denoting a test function by $\varphi\in \mathcal D((0,T); X_0)$, we have
\begin{align*}
\int_0^T \left\langle \Pi_t^*v(t), \varphi'(t)\right\rangle_{X_0^*, \, X_0}  
	 &= -\int_0^T \left\langle \varphi(t), \Pi_t^*v'(t)+\hat\Lambda(t)^* \mathcal{J}_{Y_0^*}v(t)\right\rangle_{X_0, X_0^*}. 
\end{align*}
Take $\varphi(t) = \psi(t)x$ where $\psi \in \mathcal D(0,T)$ and $x \in X_0$; this becomes
\begin{align}
    \int_0^T \psi'(t)\left\langle \Pi_t^* v(t),   x\right\rangle_{X_0^*, \, X_0}
    &= -  \int_0^T \psi(t)\left\langle x, \Pi_t^*v'(t)+\hat\Lambda(t)^* \mathcal{J}_{Y_0^*}v(t)\right\rangle_{X_0, X_0^*}.\label{eq:toRef2}
\end{align}
Manipulating and pulling the integrals inside the duality pairing, we get
\begin{align*}
     \left\langle \int_0^T \psi'(t)\Pi_t^* v(t) + \psi(t)(\Pi_t^*v'(t)+\hat\Lambda(t)^* \mathcal{J}_{Y_0^*}v(t)), x\right\rangle_{X_0^*, \, X_0} = 0.
\end{align*}
Since this is true for every $x \in X_0$, this gives, by definition of the weak time derivative,
\[(\Pi_t^* v(t))'= \Pi_t^*v'(t) +  \hat\Lambda(t)^* \mathcal{J}_{Y_0^*}v(t)\]
and here, using the identity \eqref{eq:dualPiAndBarPiIdentity} relating $\bar\Pi_t^*$ and $\Pi_t^*$ as well as a density argument for $v$, we deduce that \eqref{eq:prelim1} is satisfied for each $v \in \mathcal{W}^{p,q}(Y_0^*,Y_0^*)$.
Now, let us conclude. Again with $v\in C^1([0,T]; Y_0^*)$ and a test function $\varphi\in \mathcal D((0,T); X_0)$, we calculate
\begin{align*}
    \int_0^T \left\langle \bar \Pi_t^*v(t), \varphi'(t)\right\rangle_{X_0^*,X_0}
    &= -\int_0^T \left\langle \varphi(t), \bar \Pi_t^* v'(t) + \hat\Lambda(t)^* \mathcal{J}_{Y_0^*} v(t)\right\rangle_{X_0, \, X_0^*},
\end{align*}
and from here we follow the same argument elucidated above (beginning with the derivation of \eqref{eq:toRef2}) and this will show that $\bar \Pi^\dagger v \in \mathcal{W}^{p,q}(X_0^*, X_0^*)$. Since we already know that $\bar \Pi^\dagger v \in L^p(0,T;Y_0^*)$, the claim follows.
\end{proof}
We are now ready to provide a converse to Proposition \ref{prop:newCrit1}.
\begin{prop}[``Differentiability of $\Pi^{-1}v$"]\label{prop:evolvingreverse} Let $p, q\in [1,\infty]$. Suppose \eqref{ass:rangeOfJ} and \eqref{ass:onrangeTt}--\eqref{ass:remainsInSpace} hold. 
If $u \in L^p(0,T;X_0)$ is such that $\Pi u \in \mathcal{W}^{p,\, p\land q}(Y_0,Y_0),$ then $u \in \mathcal{W}^{p,q}(X_0,Y_0)$ with
\begin{align*}
 u' = \bar \Pi^{-1}(\Pi u)' - \bar \Pi^{-1} \mathcal{J}_{Y_0}^{-1} \hat\Lambda u.   
\end{align*}
\end{prop}
\begin{proof}
Let $\varphi\in \mathcal D((0,T); Y_0^*)$ and define $v := (\bar \Pi^\dagger)^{-1} \varphi$. By \eqref{ass:remainsInSpace}, $v \in \mathcal{W}^{p, p\land q}(Y_0^*, Y_0^*)$ 
and we can apply Lemma \ref{lem:splitLemma} to get (noting that $p \land (p \land q) = p \land q$)
\begin{align*}
\varphi' = \bar \Pi^\dagger v' + \hat \Lambda^* \mathcal{J}_{Y_0^*} v \quad\text{in $L^{p \land q}(X_0^*)$}. 
\end{align*}
Taking $u$ as stated, noting that
\begin{align*}
\langle u(t), \hat \Lambda^*(t) \mathcal{J}_{Y_0^*} v(t) \rangle_{X_0, X_0^*}
&= \langle \hat \Lambda(t) u(t),  v(t) \rangle_{Y_0, Y_0^*}
\end{align*}
(with the final equality as explained in Remark \ref{rem:newRemOnAss}), we find
\begin{align*}
    \int_0^T \langle u(t), \varphi'(t)\rangle_{X_0,X_0^*}  
    &=  -\int_0^T \langle \bar \Pi_t^{-1}(\Pi_t u(t))', \varphi(t)\rangle_{Y_0,Y_0^*} - \langle \bar \Pi_t^{-1}\hat\Lambda(t) u(t), \varphi(t) \rangle_{Y_0, Y_0^*}.
\end{align*}
Assumptions \eqref{ass:TtInvExistsAndBdd}, \eqref{ass:measurabilityOnTtInv} and the assumptions on $\hat\Lambda$ imply that 
\[\bar \Pi^{-1} \left((\Pi u)'\right)\in L^q(0,T; Y_0) \quad\text{and}\quad   \bar \Pi^{-1} \hat\Lambda u\in L^p(0,T;Y_0)\]
and thus $u\in \mathcal{W}^{p, \, p\land q}(X_0, Y_0)$ as desired. 
\end{proof}

Finally, we are able to prove the main result.
\begin{proof}[Proof of Theorem \ref{thm:equivalence}]
Suppose $u\in\mathcal{W}^{p,q}(X_0, Y_0)$, then immediately $\phi^X_{(\cdot)} u(\cdot)\in L^p_X$, so it remains to prove that this function has a weak time derivative in $L^q_Y$. Let $\eta\in\mathcal{D}_{Y^*}$, then
\begin{align*}
\int_0^T \big\langle \phi_t^X u(t), \dot \eta(t)\big\rangle_{X(t), \, X^*(t)} 
&= - \int_0^T \left\langle \phi_t^{Y} \big( \bar{\Pi}_tu'(t)\big), \eta(t)\right\rangle_{Y(t),\, Y^*(t)}  - \int_0^T \lambda(t; \phi_t^X u(t), \eta(t)),
\end{align*}
from where we conclude that $t\mapsto \phi_t^X u(t)$ has a weak time derivative 
as desired. 

For the converse direction, we begin by fixing $u\in \W^{p,q}(X,Y)$. By definition, 
for any $\eta\in\mathcal{D}_{Y^*}$, 
\begin{align*}
\int_0^T \big\langle \dot u(t), \eta(t)\big\rangle_{Y(t), Y^*(t)}  = -\int_0^T \big\langle u(t), \dot \eta(t)\big\rangle_{X(t),  X^*(t)}  - \int_0^T \lambda(t; u(t), \eta(t)),
\end{align*}
which we can pull back, arguing as in the previous paragraph and rearrange to obtain
\begin{align*}
\int_0^T \left\langle \phi_{-t}^Y \dot u(t), \left(\phi_{t}^{Y}\right)^* \eta(t)\right\rangle_{Y_0, Y_0^*} + \Big\langle \hat{\Lambda}(t)\phi_{-t}^X &u(t), \left(\phi_{t}^{Y}\right)^* \eta(t)\Big\rangle_{Y_0^{**}, Y_0^*}  = - \int_0^T \left\langle \Pi_t\left(\phi_{-t}^X u(t)\right), \big(\left(\phi_{t}^{Y}\right)^* \eta(t)\big)'\right\rangle_{Y_0, Y_0^*}.  
\end{align*}
Letting $\varphi := \left(\phi_{(\cdot)}^{Y}\right)^* \eta \in \mathcal{D}((0,T); Y_0^*)$ and using assumption \eqref{ass:rangeOfJ}, this is equivalent to
\begin{align*}
\int_0^T \left\langle \phi_{-t}^Y \dot u(t) + \mathcal{J}_{Y_0}^{-1} \hat\Lambda(t)\phi_{-t}^X u(t), \varphi(t)\right\rangle_{Y_0, Y_0^*} = - \int_0^T \left\langle \Pi_t\left(\phi_{-t}^X u(t)\right), \varphi'(t) \right\rangle_{Y_0, Y_0^*} ,
\end{align*}
from where we conclude that 
\begin{align*}
    \big(\Pi_{t} \phi^X_{-t} u(t) \big)' = \phi_{-t}^Y \dot u(t) + \mathcal{J}_{Y_0}^{-1} \hat\Lambda(t) \phi_{-t}^X u(t)
\end{align*}
with $\Pi\phi^X_{-(\cdot)} u\in \mathcal{W}^{p,p \land q}(X_0, Y_0)$. By Proposition \ref{prop:evolvingreverse} it now follows that $\phi_{-(\cdot)}^X u\in \mathcal{W}^{p,q}(X_0, Y_0)$.

The equivalence of norms is a result of the uniform boundedness of the flow maps and their inverses, and from the assumptions on $\hat \Lambda$ and $\bar \Pi$.
\end{proof}





\section{The Gelfand triple $X(t)\subset H(t)\subset X^*(t)$ setting}\label{sec:gelfandtriple}
We now specialise the theory and results of \S \ref{sec:generalCase} to the important case of a Gelfand triple
\begin{align*}
    X(t) \ctsDense H(t) \cts X^*(t)
\end{align*}
for all $t \in [0,T]$, that is, $X(t)$ is a reflexive Banach space continuously and densely embedded into a Hilbert space $H(t)$ which has been identified with its dual via the Riesz map. This setup arises frequently in the study of evolutionary variational problems and several concrete examples will be given in \S \ref{sec:examples} and \S \ref{sec:application}.

In the context of \S \ref{sec:generalCase}, we are taking $Y(t) := X^*(t)$ with the inclusion of $X(t)$ into $Y(t)$ given through compositions of the maps involved in the Gelfand triple.  Naturally, we wish to make use of the theory developed in the previous sections and the basic assumptions that one needs (namely, Assumption \ref{ass:generalcaseChanges}) translated into this Gelfand triple framework are as follows.
\begin{ass}\label{ass:gelfandTripleCase}
For all $t \in [0,T]$, assume the existence of maps
\[\phi_t^H\colon H_0 \to H(t), \qquad \phi_t^X := \phi_t^H|_{X_0} \colon X_0 \to X(t)\] 
such that
\[(H(t), \phi_t^H)_{t \in [0,T]} \quad\text{and}\quad (X(t), \phi_t^X)_{t \in [0,T]} \text{ are  compatible pairs.}\]
We assume the measurability condition \eqref{ass:measurabilityOfDual}, i.e.,
\[t \mapsto \norm{\phi_{-t}^*f}{X^*(t)} \text{ is measurable for all $f \in X_0^*$}.\]
Furthermore, suppose that
    \begin{itemize}
    \item[(i)] for fixed $u\in H_0$, 
    \[t\mapsto \|\phi_t^H u\|_{H(t)}^2 \text{ is continuously differentiable};\]
    \item[(ii)] 
for fixed $t\in [0,T]$, 
    \[(u,v)\mapsto \dfrac{\partial }{\partial t}(\phi_t^H u, \phi_t^H v)_{H(t)} \text{ is continuous,}
    \]
    and there exists $C>0$ such that, for almost all $t\in [0,T]$ and for any $u, v\in H_0$,
    \begin{align}\label{eq:lambdaGF}
        \left|\dfrac{\partial }{\partial t}(\phi_t^H u, \phi_t^H v)_{H(t)}\right| \leq C \|u\|_{H_0} \|v\|_{H_0}.
    \end{align}
\end{itemize}
\end{ass}
It follows that
\[(X^*(t), (\phi_{-t}^X)^*)_{t\in [0,T]}\text{ is a compatible pair.}\] 
Under the final assumption above, the map $\hat{\Lambda}(t)\colon X_0 \to X_0^*$, defined in Definition \ref{defn:generalOperators}, is in fact such that $\hat{\Lambda}(t)\colon H_0\to H_0^*$ is bounded and linear with
\begin{align*}
 \langle\hat{\Lambda}(t)u, v\rangle_{H_0^*, H_0} = \hat{\lambda}(t; u, v) \quad \forall u, v\in H_0.
\end{align*}
\begin{remark}\label{rem:changesGelfandTriple}
Parts (i) and (ii) of Assumption \ref{ass:gelfandTripleCase} say that, with\footnote{See Remark \ref{rem:observationsHatBetc} (i).}
\begin{equation}\label{eq:PiTAsHMap}
\Pi_t = (\phi_t^H)^A \phi_t^H, \quad \pi(t;u,v) = (\phi_t^H u, \phi_t^H v)_{H(t)},
\end{equation}
the map $(u,v)\mapsto \hat\lambda(t; u, v) = \partial \pi(t; u, v)\slash \partial t  \text{ is continuous}$
    and there exists $C>0$ such that, for almost all $t\in [0,T]$ and for any $u, v\in H_0$,
    \begin{align*}
        |\hat\lambda(t; u,v)| \leq C \|u\|_{H_0} \|v\|_{H_0}.
    \end{align*}
\end{remark}
Taking into view the Hilbert structure, the definition of the weak time derivative in \eqref{eq:genweakderiv} becomes the following.
\begin{definition}[Weak time derivative]
We say $u\in L^1_X$ has a \emph{weak time derivative} $v\in L^1_{X^*}$ if 
\begin{align*}
\int_0^T  (u(t), \dot \eta(t) )_{H(t)}  &= -\int_0^T \langle v(t), \eta(t)\rangle_{X^*(t),  X(t)} - \int_0^T \lambda(t; u(t), \eta(t)) \quad \forall \eta\in \mathcal{D}_{X}.
\end{align*}
\end{definition}
It is convenient to state Proposition \ref{prop:temamType} applied to this setting.
\begin{prop}[Characterisation of the weak time derivative]\label{prop:temamTypeGT}
Assume \ref{ass:rangeOfJ}. Let $u \in L^p_X$ and $g \in L^q_{X^*}$. Then $\dot u =g$ if and only if
\begin{equation*}
\frac{d}{dt}(u(t), \phi_{t}^H v)_{H(t)} = \langle g(t), \phi_{t}^X v \rangle_{X^*(t), X(t)} + \lambda(t; u(t), \phi_{t}^H v) \quad \forall  v \in X_0.
\end{equation*}
\end{prop}
\subsection{Differentiating the inner product: transport theorem}
We now specialise Theorem \ref{thm:transport} to this setting. We first obtain the extra regularity $\W(X, X^*)\hookrightarrow C^0_H$ as a consequence of the evolving space equivalence property, and then use it to obtain a general statement.

\begin{theorem}[Transport theorem in the Gelfand triple setting]\label{thm:transportGelfandTriple}
Let $p \in [1,\infty]$ and suppose that there exists an evolving space equivalence between $\mathcal{W}^{p,p'}(X_0, X_0^*)$ and $\W^{p,p'}(X, X^*)$. Then
\begin{itemize}
    \item[(i)] the embedding $\W^{p,p'}(X, X^*)\hookrightarrow C^0_H$ is continuous;
    \item[(ii)] given $u, v\in \W^{p,p'}(X, X^*)$, the map 
\begin{align}\label{eq:TTgelfand1}
    t\mapsto (u(t), v(t))_{H(t)}
\end{align}
is absolutely continuous and we have, for almost all $t\in [0,T]$,
\begin{align}
\dfrac{d}{dt} (u(t), v(t))_{H(t)} = \big\langle \dot u(t), v(t)\big\rangle_{X^*(t),  X(t)} + \big\langle u(t), \dot v(t)\rangle_{X(t),  X^*(t)} + \lambda(t; u(t), v(t))\label{eq:transportTheoremGelfandTriple}.
\end{align}
\end{itemize}

\end{theorem}

\begin{proof}
The proof of (i) follows from 
\begin{align*}
    u\in \W^{p,p'}(X,X^*) \Longleftrightarrow \phi_{-(\cdot)}^X u\in \mathcal W^{p,p'}(X_0, X_0^*) \Longrightarrow  \phi_{-(\cdot)}^X u\in C\left([0,T]; H_0\right) \Longleftrightarrow u\in C^0_H,
\end{align*}
where we have used the assumption and the fact that $\phi_t^H|_{X_0} = \phi_t^X$. We now turn to the proof of (ii). The fact that \eqref{eq:TTgelfand1} is an element of $L^1(0,T)$ and that \eqref{eq:transportTheoremGelfandTriple} is the weak time derivative of \eqref{eq:TTgelfand1} follows as in the proof of Theorem \ref{thm:transport}, so it suffices now to check that the right-hand side of \eqref{eq:transportTheoremGelfandTriple} is also in $L^1(0,T)$. Due to (i) and the stronger assumption \eqref{eq:lambdaGF} we may conclude with
\begin{equation*}
    \int_0^T |\lambda(t; u(t), v(t))| \;\mathrm{d}t \leq C \int_0^T \|u(t)\|_{H(t)} \|v(t)\|_{H(t)} \;\mathrm{d}t \leq C\, T \, \|u\|_{C^0_H} \|v\|_{C^0_H}.\qedhere
\end{equation*}
\end{proof}
Let us now study criteria for the spaces $\W(X,X^*)$ and $\mathcal{W}(X_0,Y_0)$ to be equivalent like in \S \ref{sec:criteriaESE}.
\subsection{Criteria for evolving space equivalence}

The evolving space equivalence criteria of Theorem \ref{thm:equivalence} tailored to the situation under consideration are as follows. It is worth pointing out that these conditions are considerably easier to check in practice than the ones given in \cite[Theorem 2.33]{AlpEllSti15a}.
\begin{theorem}[Criteria for evolving space equivalence in the Gelfand triple setting]\label{thm:ESEGelfandTriple}
Let Assumption \ref{ass:gelfandTripleCase} hold. If for all $t \in [0,T]$,
\begin{align}
    &\Pi_t\colon X_0 \to X_0 \text{ is bounded uniformly in $t$}\label{ass:gtTtrange},\\
    &\Pi_t^{-1}\colon X_0 \to X_0 \text{ exists and is bounded uniformly in $t$}\label{ass:gtTtInvertible},\\
    &\Pi_{(\cdot)}^{-1} u \in \mathcal{M}(0,T;X_0) \text{ for all $u \in X_0$} \label{ass:gtmeasurabilityOnTtAndInv},\\
    &\Pi^{-1}\colon \mathcal{W}^{p,q}(X_0, X_0) \to \mathcal{W}^{p, p\land q}(X_0, X_0) \label{ass:gtremainsInSpace},
\end{align}
then $\W^{p,q}(X,X^*)$ and $\mathcal{W}^{p,q}(X_0,X_0^*)$ are equivalent.

\medskip

\noindent More precisely,  with $\bar \Pi_t$ as in \eqref{eq:defnOfBarPiGT},
\begin{itemize}
    \item[(i)] under \eqref{ass:gtTtrange}, if $u\in \mathcal{W}^{p,q}(X_0,X_0^*)$, then $\phi_{(\cdot)}^X u\in \W^{p,q}(X,X^*)$ and 
        \begin{align}
            \md \phi_t^X u(t) = (\phi_{-t}^X)^* \bar \Pi_t u'(t).\label{eq:equiv1GT}
        \end{align}
    \item[(ii)]     
    under \eqref{ass:gtTtrange}--\eqref{ass:gtremainsInSpace},  if $u\in \W^{p,q}(X,X^*)$, then $\phi_{-(\cdot)}^X u\in \mathcal W^{p,q}(X_0,X_0^*)$ and 
        \begin{align}
            \big(\phi_{-t}^X u(t)\big)' = \bar \Pi_t^{-1} (\phi_{t}^X)^* \dot u(t).\label{eq:equiv2GT}
        \end{align}
\end{itemize}
\end{theorem}
\begin{proof}
The idea is to verify the assumptions of Theorem \ref{thm:equivalence}. Since we are in the reflexive setting, assumption \eqref{ass:rangeOfJ} is automatic and Remark \ref{rem:daggerVsDual} applies and we do not need to distinguish between $\Pi^\dagger$ and $\Pi^*$. 
Assumption \eqref{ass:gtTtrange} implies that the existence of the dual $\Pi^\#_t\colon X_0^* \to X_0^*$ to $\Pi_t$ considered as an operator $\Pi_t\colon X_0 \to X_0$ which is defined (as usual) by
\[\langle \Pi_t^\# f, x \rangle_{X_0^*,X_0} := \langle f, \Pi_t x \rangle_{X_0^*, X_0}\qquad \forall f \in X_0^*,\; x \in X_0.\]
Now, if $f \in X_0$, the right-hand side equals $(f, \Pi_t x)_{H_0}.$ On the other hand, because \eqref{ass:gtTtrange} is in force, by the self-adjoint property of  $\Pi_t\colon X_0 \to X_0^*$, 
\[\langle \Pi_t f, x \rangle_{X_0^*,X_0} = (f, \Pi_t x)_{H_0} \qquad \forall f, x \in X_0.\]
This shows that $\Pi_t^\#|_{X_0}  \equiv \Pi_t$ and hence we may take as an extension (of $\Pi_t$)
\begin{equation}\label{eq:defnOfBarPiGT}
\bar \Pi_t := \Pi_t^\#.
\end{equation}
Observe that $\bar \Pi_t \colon X_0^* \to X_0^* \text{ is bounded uniformly in $t$}$ because $\Pi_t \colon X_0 \to X_0$ is bounded uniformly by assumption and taking the dual preserves norms. This gives \eqref{ass:onrangeTt}. The measurability assumption \eqref{ass:measurabilityOnTt} follows by Remark \ref{rem:measurabilityCanBeDropped}.

Let us now see that the inverse of $\bar \Pi_t$ exists and that \eqref{ass:TtInvExistsAndBdd} is verified. Thanks to \eqref{ass:gtTtInvertible}, we may define $(\Pi_t^{-1})^\#\colon X_0^* \to X_0^*$ as the dual of $\Pi_t^{-1}\colon X_0 \to X_0$. 
We also see that, arguing as above,
\[(\Pi_t^{-1})^\#|_{X_0} = (\Pi_t^{-1})^* = \Pi_t^{-1},\]
i.e., $(\Pi_t^{-1})^\#$ extends $\Pi_t^{-1}$. We claim that $(\Pi_t^{-1})^\#$ is indeed the inverse of $\bar \Pi_t$. To see this, take $y \in X_0^*$ and a sequence $x_n \in X_0$ with $x_n \to y$ in $X_0^*.$ It follows that 
\[(\Pi_t^{-1})^\#\bar\Pi_t x_n = \Pi_t^{-1}\Pi_t x_n = x_n \to y\]
but by continuity, the left-hand side converges to $(\Pi_t^{-1})^\#\bar\Pi_t y$, and hence we have shown that $(\Pi_t^{-1})^\# = (\bar\Pi_t)^{-1}$ (in the sense of the left inverse; the right inverse follows by the same argument). The remaining claims in assumption \eqref{ass:TtInvExistsAndBdd} follow by the same reasoning as above. Assumption \eqref{ass:measurabilityOnTtInv} on the measurability is implied by \eqref{ass:gtmeasurabilityOnTtAndInv} and (again) a density argument just as in Remark \ref{rem:measurabilityCanBeDropped}. 

By reflexivity, it follows that $\bar \Pi_t^* \equiv \Pi_t$ and hence $(\bar \Pi_t^*)^{-1} = \Pi_t^{-1}$, so that \eqref{ass:gtremainsInSpace} directly gives \eqref{ass:remainsInSpace}. The conclusion now follows from Theorem \ref{thm:equivalence}.
\end{proof}
\begin{remark}
It is important to emphasise that $\bar\Pi_t := \Pi^\#_t$ defined in the proof above is, in general, different to $\Pi_t^*$, the dual of $\Pi_t\colon X_0 \to X_0^*$. 
\end{remark}

\begin{remark}
The result above, and the more general Theorem \ref{thm:equivalence}, are a generalisation of the results previously obtained by the first and last authors in \cite[Theorem 2.33]{AlpEllSti15a}. Indeed, the assumptions in Theorems \ref{thm:equivalence} and \ref{thm:ESEGelfandTriple} imply the assumption in \cite[Theorem 2.33]{AlpEllSti15a} that $\Pi_t$ maps functions in $\mathcal W(X_0,X_0^*)$ to the same space, and are more detailed than those in \cite[Theorem 2.33]{AlpEllSti15a} making them easier to verify. With regards to the operators $\hat S(t), \hat D(t)$ appearing in \cite{AlpEllSti15a}, an analysis of our proof shows that we have $\hat S(t) = \bar\Pi_t$ and $\hat D(t)\equiv 0$, and thus the assumptions in \cite[Theorem 2.33]{AlpEllSti15a} on those operators are in fact guaranteed by those on $\bar\Pi_t$ in our result.
\end{remark}

\subsection{Alternative criteria for the assumption \eqref{ass:gtremainsInSpace}}
For some applications, it may turn out that \eqref{ass:gtremainsInSpace} (or \eqref{ass:remainsInSpace}) is too cumbersome or inconvenient to verify in practice (as will be the case in one of the examples we consider below), so we would like to have alternative criteria to replace it. This is what we focus on now. Defining the Hilbert adjoint $\xi_t := (\phi_{-t}^H)^A$, it follows that the pair $(H(t), \xi_t)_{t\in [0,T]}$ is compatible if
\begin{equation}\label{ass:forMeasOfAdj}
    \text{$H_0$ is separable or $t \mapsto \norm{(\phi_{-t}^H)^A u}{H(t)}$ is measurable for $u \in H_0$}.
\end{equation}
\begin{lem}\label{lem:altGTCriteria}
Under \eqref{ass:gtTtrange}, \eqref{ass:gtTtInvertible}, \eqref{ass:gtmeasurabilityOnTtAndInv}, \eqref{ass:forMeasOfAdj} and if for all $t \in [0,T]$,
\begin{align}
&\text{Assumption \ref{ass:generalcaseChanges} holds for the maps $\xi_t$,}\label{ass:newOneForcheck}\\
&\text{$\hat\Lambda^\xi(t)\colon X_0\to X_0^*$ satisfies $\hat\Lambda^\xi(t)(X_0)\subset X_0$,}\label{ass:rangeOfHatLambdaForNewCheck}\\
&(\Pi_t^{-1})^*\colon X_0 \to X_0 \text{ exists and is bounded uniformly in $t$}\label{ass:TtInvDualSmooth},
\end{align}
then assumption \eqref{ass:gtremainsInSpace} holds.
\end{lem}

\begin{proof}
These assumptions allow us to apply the theory developed in this section now with the maps $\xi_t$. The proof of Theorem \ref{thm:ESEGelfandTriple} above shows that $\bar \Pi_t := \Pi_t^\#\colon X_0^*\to X_0^*$ is an extension of $\Pi_t$ to $X_0^*$. Likewise, the map $\Pi^\xi = \xi_t^A \xi_t = \Pi_t^{-1}\colon X_0\to X_0$ has an extension $\bar \Pi^\xi_t = (\Pi_{t}^{-1})^*$ to $X_0^*$, which by \eqref{ass:TtInvDualSmooth} satisfies \eqref{ass:onrangeTt}, \eqref{ass:measurabilityOnTt}. We can, using \eqref{ass:rangeOfHatLambdaForNewCheck}, thus apply Proposition \ref{prop:newCrit1} to $\Pi^\xi=\Pi^{-1}$ (with the $X_0$ and $Y_0$ in the statement of the proposition chosen to be $X_0$), which implies \eqref{ass:gtremainsInSpace} 
with 
\begin{align*}
    &(\Pi^{-1}_t u(t))' = (\Pi^{-1}_t)^* u'(t) + \hat \Lambda^\xi(t) u(t) \quad \forall u\in \mathcal{W}^{p,q}(X_0, X_0).\qedhere
\end{align*}
\end{proof}

\begin{remark}
Note that the map $(\Pi_t^{-1})^*\colon X_0^*\to X_0^*$ relates to $\Pi_t^\#$ via   $(\Pi_t^{-1})^* = (\Pi_t^\#)^{-1}$.
\end{remark}

\subsection{Evolving space equivalence for the space $\W(X,H)$}

It can sometimes be the case that solutions to PDEs have the time derivative belonging not just to $L^q_{X^*}$ but the more regular space $L^q_H$. In this case, we say that solutions belong to  $\W(X,H)$ and it can be useful to know under which circumstances this space is equivalent to $\mathcal{W}(X_0,H_0)$. 
\begin{theorem}[Criteria for regularity of evolving space equivalence in the Gelfand triple setting]\label{thm:ESEGelfandTripleForH}
Let the assumptions of Theorem \ref{thm:ESEGelfandTriple} hold. 
Then $\W^{p,q}(X,H)$ and $\mathcal{W}^{p,q}(X_0,H_0)$ are equivalent.
\end{theorem}
\begin{proof}
We first need some basic properties of the various adjoint and dual maps. An easy calculation shows that  $(\phi_t^X)^* |_{X(t)} \equiv (\phi_t^H)^A\colon X(t) \to H_0$ and hence, by density of $X(t) \subset H(t)$,
\begin{align}
(\phi_t^X)^* |_{H(t)} \equiv (\phi_t^H)^A.\label{eq:gtNiceRestr1}
\end{align}
By the same reasoning, 
\begin{align}
(\phi_{-t}^X)^* |_{H_0} \equiv (\phi_{-t}^H)^A.\label{eq:gtNiceRestr2}
\end{align}
Now, from the formula \eqref{eq:equiv1GT}, for $u \in \mathcal{W}^{p,q}(X_0, H_0) \subset \mathcal{W}^{p,q}(X_0, X_0^*)$, we have $\md \phi_t^X u(t) = (\phi_{-t}^X)^* \bar \Pi_t u'(t)$. Since $\bar\Pi_t = \Pi_t$ on $X_0$ and the right-hand side is well defined and bounded from $H_0$ into $H_0$ (see \eqref{eq:PiTAsHMap}), we have $\bar\Pi_t|_{H_0} = \Pi_t$ too. Thanks to this and utilising the additional regularity that $u' \in L^q(0,T;H_0)$, we get
\begin{align*}
    \md \phi_t^X u(t) &= (\phi_{-t}^X)^*  \Pi_t u'(t) = (\phi_{-t}^X)^*  (\phi_t^H)^A\phi_t^H u'(t) = \phi_t^H u'(t)
\end{align*}
where for the last equality we used \eqref{eq:gtNiceRestr2}.

In the other direction, taking $u \in \W^{p,q}(X,H)$, now the formula \eqref{eq:equiv2GT} gives\footnote{Let us note that $t \mapsto (\phi_t^H)^Aw(t)$ is measurable for every $w \in L^p_H$ from $(0,T)$ to $H_0$ since $(\phi_t^H)^Aw(t) = \Pi_t\phi_{-t}^H w(t)$ is measurable as remarked in \S \ref{sec:generalCase}.}
\[(\phi_{-t}^X u(t))' = \bar \Pi_t^{-1} (\phi_{t}^X)^* \dot u(t) = \bar \Pi_t^{-1} (\phi_t^H)^A\dot u(t) = \Pi_t^{-1} (\phi_t^H)^A\dot u(t) = ((\phi_t^H)^A\phi_t^H)^{-1} (\phi_t^H)^A\dot u(t) = \phi_t^H \dot u(t)\]
where we made use of \eqref{eq:gtNiceRestr1} and the fact that $\bar\Pi_t^{-1}$ can be defined on $H_0$ (just as we argued above).
\end{proof}
The proof reveals that a function $u \in \W(X,H)$ has a weak time derivative given by
\[\dot u(t) = \phi_{-t}^H (\phi_{-t}^X u(t))'\]
which is a natural generalisation of the formula for the strong time derivative.

\section{The Aubin--Lions lemma in evolving spaces}\label{sec:aubinlions}
Our aim is to generalise the following result (see e.g. \cite[Lemma 7.7]{Rou05}).

\medskip

\noindent \textbf{Aubin--Lions lemma}. Let $X$, $Y$ and $Z$ be Banach spaces such that $X$ is separable and reflexive. Suppose $X \ctsCompact Z$ is compact and $Z \cts Y$ is injective. Then $\mathcal W^{p,q}(X, Y)\ctsCompact L^p(0,T; Z)$ is also compact for any $1<p<\infty$ and $1\leq q\leq \infty$

\medskip

The Aubin--Lions lemma provides a compactness result which is often used in the study of nonlinear evolutionary equations. The first result on the compact embedding of spaces of Banach-valued functions was shown by Aubin \cite{aubin1963analyse}, then it was extended by Dubinski\u{i} \cite{dubinskii1965weak, BarSuli} and improved by Simon in \cite{simon1986compact}. For more details, see \cite{CJL}.

In recent years, motivated by applications in biology \cite{EllStiVen12} and fluid dynamics \cite{vcanic2020moving}, the topic of extending the previous results to the case when the target set is a family of time-evolving spaces has become very popular. We refer the interested reader to \cite{KnobKrech} for the discussion about the origin of time-varying problems and its applications. Among first tasks in this direction is to define a weak time derivative in the moving setting and to consider the corresponding Sobolev--Bochner spaces. This has been done for example in \cite{evseev2019sobolev} where the authors construct a generalisation of an $L^p$ direct integral. One of the first proofs of a compactness lemma in the case of a moving domain is considered for the treatment of incompressible Navier--Stokes equations in moving domains  and is presented  in  \cite{sauer1970existence}. For similar results, see \cite{moussa2016some, boudin2009global, lan2020quasilinear}.
 We now state and prove our Aubin--Lions-type compactness based on the spaces that we have introduced. We work under the following assumption:

\begin{ass}\label{ass:aubin_lions}
In addition to the compatible pairs
\begin{align*}
    \big(X(t), \phi_t^X\colon X_0\to X(t)\big)_{t\in [0,T]} \qquad \text{ and } \qquad \big(Y(t), \phi_t^Y\colon Y_0\to Y(t)\big)_{t\in [0,T]}, \quad\quad
\end{align*}
with $X(t) \subset Y(t)$ (just as in \S \ref{sec:timedep}), we assume the existence of an additional family of Banach spaces 
$$\big(\{Z(t)\}_{t\in [0,T]}, \big(\phi_t^Z\colon Z_0\to Z(t)\big)_{t\in [0,T]}\big)$$
such that $(Z(t), \phi_t^Z)_{t \in [0,T]}$ is compatible and $X_0\overset{c}{\hookrightarrow} Z_0\hookrightarrow Y_0$. We also assume $$\phi_t^Z|_{X_0} = \phi_t^X.$$
\end{ass}

\begin{theorem}[Aubin--Lions lemma]
Under Assumption \ref{ass:aubin_lions}, suppose that $\mathcal W^{p,q}(X_0,Y_0)$ and $\W^{p,q}(X,Y)$ are equivalent in the sense of Definition \ref{defn:equivalence}. For any $p\in (1,\infty)$ and $q\in [1,\infty]$, the embedding 
\[\W^{p,q}(X, Y)\ctsCompact L^p_Z\]
is compact.
\end{theorem}
\begin{proof}
Suppose $(u_n)_n$ is a bounded sequence in $\W^{p,q}(X, Y)$, then by the equivalence of spaces $(\phi_{-(\cdot)}^X u_n(\cdot))_n$ is bounded in $\mathcal W^{p,q}(X_0, Y_0)$. By the classical Aubin--Lions lemma, it has a convergent subsequence in $L^p(0, T; Z_0)$, say $(\phi_{-(\cdot)}u_{n_k}(\cdot))_k$. Using the uniform boundedness of $\phi_{(\cdot)}^Z$, $(\phi_{(\cdot)}
^Z \phi_{-(\cdot)}^X u_{n_k})_k = (u_{n_k})_{k}$ also converges in $L^p_Z$, proving the result.
\end{proof}

\begin{remark}
As shown above, assuming the evolving space equivalence property makes the proof of the Aubin--Lions lemma straightforward. It is not the aim of this section to obtain the most general statement but rather to prove that, within the setting of an evolving space equivalence, the classical results on fixed domains carry over to the time-dependent framework. It is worthwhile mentioning that compactness results in the spirit of the Aubin--Lions lemma have been obtained in certain evolving space applications, with assumptions weaker than the ones we present above. See for instance \cite[Theorem 3.1]{MuhCan19}.  
\end{remark}

\part{Applications}

\section{Examples of function spaces on evolving domains and surfaces}\label{sec:examples}

In the following examples we consider spaces of Lebesgue integrable or Sobolev functions over evolving domains and surfaces. We will prove that the theory of this paper can be applied to these cases, which should be useful when studying a wide variety of evolutionary problems on moving domains and surfaces. In particular, we will show that evolving space equivalences hold, which can be rather non-trivial.

\paragraph{Evolving domains and surfaces.} Let us begin with the basic assumptions and notations that we need in order to describe evolving domains and surfaces. In what follows, $T\in (0,\infty)$ is a fixed positive real number.
\begin{ass}\label{ass:evolv_spaces_examples}
We assume the following.
\begin{itemize}
    \item[(i)] Let
\[\text{$\mathcal M_0$ be a bounded $C^2$ domain $\Omega_0 \subset \mathbb{R}^n$ or a $C^2$ $n$-dimensional hypersurface $\Gamma_0 \subset \mathbb{R}^{n+1}$},\]
with $\Omega_0$ connected and $\Gamma_0$ closed (i.e., compact and without boundary) and connected.
    \item[(ii)] Define 
\begin{align}\label{eq:dimension}
        d = 
        \begin{cases}
        n & \text{ if } \mathcal M_0 = \Omega_0 \\
        n + 1 & \text{ if } \mathcal M_0 = \Gamma_0
        \end{cases}.
    \end{align}
     Let $$\mathbf{w}\colon [0,T]\times \mathbb{R}^{d} \to \mathbb{R}^{d}\in C^0([0,T],C^2(\mathbb{R}^{d}, \mathbb{R}^{d}))$$ be a given vector field that we interpret to be a velocity field. 
     We define a \textit{flow map} $$\Phi^0_{(\cdot)}\colon [0,T] \times \mathbb{R}^{d} \to \mathbb{R}^{d}$$ via the ODE
         \begin{align*}
        \frac{d}{dt}\Phi^0_t(p) &= \mathbf{w}(t,\Phi^0_t(p)), \quad p\in\overline{\mathcal M_0},\\
        \Phi^0_0 &= \mathrm{Id}\quad \text{ on } \overline{\mathcal M_0}.
    \end{align*}
    \item[(iii)]Denoting $\mathcal M(t) := \Phi_t^0(\mathcal M_0)$,
    \begin{itemize}
        \item[(iii.a)] $\Phi_t^0\colon \overline{\mathcal M_0}\to \overline{\mathcal M(t)}$ is a $C^2$-diffeomorphism satisfying $\Phi_t^0(\mathcal M_0) = \mathcal M(t)$ and $\Phi_t^0(\partial\mathcal M_0) = \partial\mathcal M(t)$;
        \item[(iii.b)] $\Phi_t^0|_{\mathcal M_0}\colon \mathcal M_0\to \mathcal M(t)$ and $\Phi_t^0|_{\partial \mathcal M_0}\colon \partial\mathcal M_0\to \partial\mathcal M(t)$ are also $C^2$-diffeomorphisms.
    \end{itemize}
   %
\end{itemize}
\end{ass}

We refer to the family $\{\mathcal M(t)\}_{t\in [0,T]}$ as an \textit{evolving domain/surface}. It follows from the assumption above that $\Phi_{(\cdot)}^0 \in C^1([0,T], C^2(\mathbb{R}^d,\mathbb{R}^d))$. Furthermore we denote 
\[\Phi^t_0 := (\Phi^0_t)^{-1}.\] 

\begin{remark}
The regularity required in the assumption above is sufficient for the applications we have in mind, including cell biology or biomembranes, see e.g. \cite{venkataraman2011modeling, AlpEllTer17}, where one is led to consider PDEs on smooth surfaces. It would be natural to contemplate a more general framework in which the underlying domain is less regular or in which the transformations between the domains do not preserve the initial smoothness. This would be interesting from the point of view of applications, allowing for a treatment of more complex structures, as well as from the analysis side by including more ambitious systems arising from free boundary problems. We leave these considerations for future work.
\end{remark}

In the next sections, we study the following cases involving Gelfand triples:
\begin{enumerate}[label=({\roman*})]
    \item $H(t)=L^2(\mathcal M(t))$ with $X(t)=W^{1,r}(\mathcal  M(t))$,
    \item $H(t)=H^1(\mathcal M(t))$ with $X(t)=W^{2,r}(\mathcal  M(t))$,
    \item $H(t)=H^{-1}(\Omega(t))$ with $X(t)=L^p(\Omega(t)) \cap H^{-1}(\Omega(t))$,
    \end{enumerate}
    and the non-Gelfand triple examples
    \begin{enumerate}[resume, label=({\roman*})]
    \item $X(t)=W^{k,r}(\Gamma(t))$ with $Y(t)=L^1(\Gamma(t))$ (for $k=0,1$),
    \item $X(t)=W^{2,r}_0(\Omega(t))$ with $Y(t)=W^{1,1}_0(\Omega(t))$.
\end{enumerate} 
We stress that these spaces are independent of the flow map $\Phi_0^t$. Before we proceed, we need to introduce some more concepts and properties. 

\paragraph{Pushforward and pullback maps.} For functions $u\colon \mathcal{M}_0 \to \mathbb{R}$, we define the pushforward map $\phi_t$ by 
 \begin{equation}\label{eq:defnOfPhiEg}
\phi_t u := u \circ \Phi^t_0.     
 \end{equation}
Its inverse $\phi_{-t} v = v \circ \Phi^0_t$ acting on functions $v \colon \mathcal{M}(t) \to \mathbb{R}$ is called the pullback map. 

\paragraph{Differential operators and integration by parts} The notation $g(t)$ will be used to refer to the Riemannian metric tensor associated to $\mathcal{M}(t)$ and $\grad_{g(t)}$ will stand for the usual gradient when $\mathcal{M}(t)=\Omega(t)$ and the surface gradient (or tangential gradient) when $\mathcal{M}(t)=\Gamma(t)$; the latter can be seen as the projection of the gradient (of a suitable extension) of the function onto the tangent space. We write $\mathbf{D}\Phi^0_t$ for the Jacobian matrix of partial derivatives of $\Phi^0_t$ (which, in case $\mathcal{M}(t) = \Gamma(t)$, refers to the tangential partial derivatives with respect to the ambient space). Note that, in either case, this denotes an $(n+1)\times (n+1)$ matrix.

The integration by parts formula on surfaces \cite[Theorem 2.10]{DziEll13-a} for sufficiently smooth functions is
\begin{equation*}
    \int_{\Gamma(t)} u\partial_i v = -\int_{\Gamma(t)} v\partial_i u + \int_{\Gamma(t)} uvh_\Gamma(t)\nu_i(t),
\end{equation*}
where $\partial_i$ refers to the $i$th component of $\grad_{g(t)}$, $\nu(t)$ is the unit normal vector on $\Gamma(t)$, and $h_\Gamma(t)$ is the mean curvature of $\Gamma(t)$ defined as the sum of the principal curvatures.

Defining the determinant of the Jacobian matrix
\[J^0_t := |\mathrm{det}\mathbf{D}\Phi^0_t|,\]
from continuity and $J^0_0=1$ we have its uniform boundedness: there exists a constant $C_J > 0$ such that
\begin{equation*}
    0< C_J^{-1} \leq J^0_t \leq C_J \quad \forall t \in [0,T].
\end{equation*}
Moreover, from the regularity assumptions on the velocity field, it follows $J^0_{(\cdot)} \in C^1([0,T] \times \mathcal{M}_0)$ and 
\begin{equation}
    \frac{d}{dt} J_t^0 = \phi_{-t} (\nabla_{g(t)} \cdot \mathbf w(t))J^0_t.\label{eq:derivativeOfJ0t}
\end{equation}
We also sometimes use the following transport formula (see \cite[Equation (5.8)]{DziEll13-a} in the case of an evolving surface):
\begin{align}\label{eq:transportFormulaGradient}
    \frac{d}{dt}\int_{\mathcal{M}(t)} \grad_{g(t)} u \cdot \grad_{g(t)} v 
    &= \int_{\mathcal{M}(t)}\grad_{g(t)} \dot u \cdot \grad_{g(t)} v + \grad_{g(t)} u \cdot  \grad_{g(t)} \dot v + \grad_{g(t)} u^{\T} \mathbf{H}(t) \grad_{g(t)} v,
\end{align}
where the notation $(\cdot)^\T$ means the transpose of the matrix and we defined the deformation tensor
\begin{align*}
    \mathbf{H} := (\grad_{g} \cdot \mathbf{w})\mathrm{Id} - \left(\mathbf{D}_{g}\mathbf{w} + (\mathbf{D}_{g}\mathbf{w})^{\T}\right).
\end{align*}
We refer the reader to \cite{DziEll13-a, Dziri2001EulerianDF, AlpEllSti15b} and citations therein for full details on (evolving) hypersurfaces and their definitions in this context.

For later use it is convenient to introduce the following positive-definite (with a constant that is uniform in time) matrix
and its determinant
\begin{equation*}
    \mathbf{A}_t^0 :=
\begin{cases}
 (\mathbf{D}\Phi_t^0)^{\T} \mathbf{D}\Phi_t^0 \quad &\text{ if } \mathcal M(t)=\Omega(t), \\
 (\mathbf{D}_{g_0}\Phi_t^0)^{\T} \mathbf{D}_{g_0}\Phi_t^0 + \nu_0\otimes \nu_0 \quad &\text{ if } \mathcal M(t) = \Gamma(t),
\end{cases} 
\qquad \qquad a_t^0 := \det \mathbf{A}_t^0.
\end{equation*}
When $\mathcal M(t)=\Gamma(t)$, we have that (see Proposition 4.1 of \cite{ChuDjuEll20})
\[(\mathbf{A}_t^0)^{-1} = \phi_{-t}((\mathbf{D}_{g(t)}\Phi^t_0) (\mathbf{D}_{g(t)}\Phi^t_0)^\T) + \nu_0 \otimes \nu_0.\]
\paragraph{Transformation of differential operators}We record the following expressions (see  \cite[Proposition 2.29, Lemma 2.30, Lemma 2.62, Equation (2.91), p.~64]{SZShape} for the flat case, and \cite[Section 3]{ChuDjuEll20} for surfaces):
\begin{align*}
    J^t_0 &= \phi_t ((J^0_t)^{-1}),
\end{align*}
and, given sufficiently smooth functions $u\colon \mathcal M(t)\to \R$ and $v\colon \mathcal M_0\to \R$, we have:
\begin{itemize}
    \item[(i)] for the gradient operator, via the chain rule for tangential gradients,
    \begin{align*}
        \grad_{g_0} \left(\phi_{-t} u\right) = (\mathbf{D}_{g_0}\Phi_t^0)^\T \phi_{-t} \left(\grad_{g(t)} u\right),
    \end{align*}
    and to invert the formula in the case of a surface we need again to add the term corresponding to the normal component, yielding 
    \begin{align}
        \phi_{-t} \left(\grad_{g(t)} u\right) &= \mathbf{D}_{g_0}\Phi_t^0 (\mathbf A_t^0)^{-1}\grad_{g_0} \left(\phi_{-t}  u\right)\label{eq:pullbackOfGradient},
    \end{align}
    \item[(ii)] for the Laplace--Beltrami operator, 
    \begin{align}
    \phi_{-t}(\Delta_{g(t)} u) &= \dfrac{1}{\sqrt{a_t^0}} \grad_{g_0} \cdot\left(\sqrt{ a_t^0} (\mathbf{A}_t^0)^{-1}\grad_{g_0}\phi_{-t}u\right)\label{eq:pullbackOfLaplacian},\\
    \phi_{t}(\Delta_{g_0} v) &= \dfrac{1}{\sqrt{a_0^t}} \grad_{g(t)} \cdot\left(\sqrt{ a_0^t} (\mathbf{A}_0^t)^{-1}\grad_{g(t)}\phi_{t}v\right)\label{eq:pushforwardOfLaplacian}.
\end{align}
\end{itemize}
In the next two sections we explore some particular examples.
%

\subsection{Gelfand triple examples}\label{sec:gelf}

In the following, we omit the calculations and proofs of the evolving space equivalence property and refer to \S \ref{sec:proofs} for these details.

\subsubsection{$L^2(\mathcal{M}(t))$ pivot space}\label{sec:L2pivotspace}
In this subsection we present the most commonly occurring case where the pivot space is an $L^2$ space, namely
\[H(t) := L^2(\mathcal{M}(t)).\]
This example was already analysed (for $\mathcal{M}(t) = \Gamma(t)$ and various $X(t)$) in \cite{AlpEllSti15b} but due to its importance and universal role in many applications, we will treat it afresh here for the convenience of the reader and for completeness. 

Let $r \geq 2$ and define $X(t) := W^{1,r}(\mathcal{M}(t))$ and $Y(t) := X^*(t) = W^{1,r}(\mathcal{M}(t))^*$; for $X(t)$, we take the usual norm
\[\norm{u}{W^{1,r}(\mathcal{M}(t))} := \left(\int_{\mathcal{M}(t)}|u|^r + |\grad_{g(t)} u|^r\right)^{1\slash r}.\]
Hence, we have the Gelfand triple structure
\begin{equation*}
    W^{1,r}(\mathcal{M}(t)) \subset L^2(\mathcal{M}(t)) \subset W^{1,r}(\mathcal{M}(t))^*.
\end{equation*}
We denote by $\phi_t$ the pushforward map defined above in \eqref{eq:defnOfPhiEg}. It is an easy calculation to verify that, under Assumption \ref{ass:evolv_spaces_examples}, the pairs $(L^2, \phi_t)_{t \in [0,T]}$ and $(W^{1,r}, \phi_t)_{t \in [0,T]}$ are compatible. By using the transport formula, we can establish:

\begin{restatable}{lem}{gelfoneone}
Under Assumption \ref{ass:evolv_spaces_examples}, we have
\begin{equation*}
\lambda(t; u,v) 
       =        \int_{\mathcal{M}(t)} uv \nabla_{g(t)} \cdot \mathbf{w}(t).
\end{equation*}
\end{restatable}

This leads to the definition:

\begin{definition}[$L^2(M)$ weak time derivative]\label{def:l2materialderivative}
A function $u\in L^2_{X}$ has a weak time derivative $\dot u \in L^2_{X^*}$ if and only if 
\begin{align*}
\int_0^T \langle \dot u(t), \eta(t) \rangle_{X^*t), X(t)} &=-\int_0^T\int_{\mathcal{M}(t)}u(t)\dot \eta(t) - \int_0^T\int_{\mathcal{M}(t)}u(t) \eta(t) \nabla_{g(t)} \cdot\mathbf{w}(t)\quad \forall \eta\in \mathcal{D}_{X}.
\end{align*}
\end{definition}

We can then prove:

\begin{restatable}{prop}{eseLTwo}Under Assumption \ref{ass:evolv_spaces_examples}, given $r\geq 2$ and for any $p,q\in [1,\infty]$, there exists an evolving space equivalence between the spaces $\mathcal{W}^{p,q}(W^{1,r}(\mathcal{M}_0), W^{1,r}(\mathcal{M}_0)^*)$ and $\W^{p, q}(W^{1,r}, (W^{1,r})^*)$.
\end{restatable}

\paragraph{Applications.}\label{subsec:applications}
There are numerous examples of PDEs on evolving domains or surfaces with $L^2$ as the pivot space. Some equations are analysed in \cite{AlpEllSti15b, AlpEll15, AlpEllTer17}, and here we mention a few of them.

\medskip

\noindent (1) The archetypal equation (on a surface) is the surface advection-diffusion equation
\begin{equation*}
\begin{aligned}
\dot u - \Delta_g u + u \nabla_g  \cdot \mathbf{w} &= 0 &&\text{on } \Gamma(t),\\
    u(0) &= u_0 &&\text{on } \Gamma_0,
\end{aligned}
\end{equation*}
where $u_0 \in L^2(\Gamma_0)$.
In this case, $X(t) = H^1(\Gamma(t))$ and the evolving space equivalence and well posedness are proved in \cite{AlpEllSti15b}. 

\medskip

\noindent (2) Similar results can be derived for systems of equations with bulk-surface interactions. Here we mention the coupled bulk-surface system that was studied in \cite{AlpEllSti15b}, in which case both $\Omega(t), \Gamma(t)\subset \R^{n+1}$ and we have $\Gamma(t)=\partial\Omega(t)$:
\begin{equation*}
\begin{aligned}
   \dot u - \Delta_\Omega u + u \nabla_\Omega \cdot  \mathbf{w} &= f  &&\text{on } \Omega(t), \\
     \dot u - \Delta_\Gamma v + v \nabla_\Gamma \cdot  \mathbf{w} + \nabla_\Omega u \cdot \nu &= g  &&\text{on } \Gamma(t), \\
    \nabla_\Omega u \cdot \nu &= \beta v - \alpha u  &&\text{on } \Gamma(t), \\
    u(0) &= u_0 &&\text{on } \Omega_0, \\
    v(0) &= v_0  &&\text{on } \Gamma_0,
\end{aligned}
\end{equation*}
where $u_0 \in H^1(\Omega_0)$, $v_0 \in H^1(\Gamma_0)$, $\alpha, \beta > 0$ are given constants. Setting $X(t) = H^1(\Omega(t)) \times H^1(\Gamma(t))$ and $H(t) = L^2(\Omega(t)) \times L^2(\Gamma(t))$, one can show existence for the system (see \cite[\S 5.3]{AlpEllSti15b}). The analysis and properties of a more complicated and nonlinear coupled bulk-surface system can be found in \cite{AlpEllTer17}.

\medskip

\noindent (3) Moreover in \cite[\S 5.4.1]{AlpEllSti15b} the authors considered  the fractional Sobolev space $X(t) = H^{1/2}(\Gamma(t))$ and proved that $\W(X,X^*)$ and  $\mathcal{W}(X_0,X_0^*) $ are equivalent --- a fact which was used to aid with the study of the fractional porous medium equation 
\begin{equation*}
\begin{aligned}
    \dot u + (- \Delta_{g})^{1/2}(u^m) + u \nabla_{g} \cdot \mathbf{w} &= 0  &&\text{on $ \Gamma(t)$}, \\
    u(0) &= u_0  &&\text{on $\Gamma_0$},
\end{aligned}
\end{equation*}
in \cite{AlpEllFractional}. Here, $m \geq 1$, $u_0 \in L^\infty(\Gamma_0)$ , $u^m := |u|^{m-1} u$ and $(- \Delta_{g(t)})^{1/2}$ is a square root of the Laplace--Beltrami operator on $\Gamma(t)$. 

\medskip

\noindent (4) Another example is the Cahn--Hilliard system on an evolving surface $\{\Gamma(t)\}_{t\in [0,T]}$
\begin{equation*}
    \begin{aligned}
        \dot u + u \sgrad\cdot\mathbf{w} &= \Delta_g \mu \,\,\, &\text{ in } \Gamma(t),\\
        -\Delta_g u + W'(u) &= \mu \,\,\, &\text{ in } \Gamma(t),\\
        u(0) &= u_0,
    \end{aligned}
\end{equation*}
    where $W$ is a given potential. This is analysed in \cite{EllRan15} with $W(r)=(r^2-1)^2/4$, where the authors obtain, for $u_0\in H^2(\Gamma_0)$, $u\in \W^{\infty, 2}(H^1, L^2)$. This has been generalised in \cite{CaeEll21} by the second and last authors for a wider class of potentials and $u_0\in H^1(\Gamma_0)$, where conditions are obtained so that the solution $u\in\W^{\infty,2}(H^1, H^{-1})$ and $\mu\in L^2_{H^1}$. 
    
\subsubsection{$H^1(\mathcal{M}(t))$ pivot space}\label{sec:H1pivotspace}

Beside  the standard choice of $L^2$ as pivot space, another possibility for a pivot space is $H^1$.  A typical example is the bi-Laplace (also called biharmonic) equation which involves a fourth order elliptic operator and is important in applied mechanics, in particular in the theory of elasticity. The equation is analysed for example in  \cite[\S 3, 4.7.5, Example 5]{lions2012non}. 

Let $H(t) = H^1(\mathcal{M}(t))$ with $\phi_t\colon H_0 \to H(t)$ as in \eqref{eq:defnOfPhiEg}.
In this example we work with 
\[X(t)=W^{2,r}(\mathcal{M}(t))\quad \text{for $r \geq 2$}.\]
 We start by verifying that $\phi_t$ takes $X_0$ into $X(t)$. We require more regularity for $\mathbf{w}$ and $\Phi$, namely
\begin{align}\label{eq:extraregularityH1}
    \mathbf{w} \in C^0\left([0,T]; C^3(\R^{d}, \R^{d})\right) \quad \text { and } \quad \Phi_0^{(\cdot)} \in C^1\left([0,T]; C^3(\R^{d}, \R^{d})\right),
\end{align}
where $d$ is as in \eqref{eq:dimension}. As before, under Assumption \ref{ass:evolv_spaces_examples} and the extra regularity \eqref{eq:extraregularityH1}, it is easy to show that the pairs $(H^1, \phi_t)_{t \in [0,T]}$ and $(W^{2,r}, \phi_t)_{t \in [0,T]}$ are compatible. Again by using the differentiation formulas we can prove:

\begin{restatable}{lem}{gelftwoone}
For $u, v \in H(t)$,
\begin{align*}
    \lambda(t; u,v) 
    = \int_{\mathcal{M}(t)} uv \sgrad \cdot \mathbf{w}(t) + \sgrad u^{\T} \mathbf{H}(t)  \sgrad v.
\end{align*}
\end{restatable}

\begin{definition}[$H^1(M)$ weak time derivative]\label{eq:h1weakderiv}
A function $u\in L^p_{X}$ has a weak time derivative $\dot u \in L^q_{X^*}$ if and only if 
\begin{align*}
\int_0^T \langle \dot u(t), \eta(t) \rangle_{X^*(t), X(t)} &=-\int_0^T \int_{\mathcal{M}(t)}u(t)\dot \eta(t)- \int_0^T \int_{\mathcal{M}(t)}u(t)\eta(t)\sgrad \cdot \mathbf{w}(t)+\sgrad u(t)^{\T} \mathbf{H}(t) \sgrad  \eta(t)
\end{align*}
for all $\eta\in \mathcal{D}_{X}$.
\end{definition}

Also in this case we establish the evolving space equivalence property:
\begin{restatable}{prop}{eseHOne}\label{prop:eseH1}
Under Assumption \ref{ass:evolv_spaces_examples} and \eqref{eq:extraregularityH1}, for any $p,q\in [1,\infty]$, there exists an evolving space equivalence between the spaces $\mathcal{W}^{p,q}(W^{2,r}(\mathcal{M}_0), W^{2,r}(\mathcal{M}_0)^*)$ and $\W^{p, q}(W^{2,r}, (W^{2,r})^*)$.
\end{restatable}

\paragraph{Application.} We explore an example which motivates the choice of $H^1_0$ as a pivot space. We present it in the fixed domain setting for simplicity, but it can be easily generalised to an evolving domain or hypersurface. Let $\Omega \subset \mathbb{R}^n$ be a sufficiently regular bounded domain. We consider the bi-Laplace equation
\begin{align}
    \frac{\partial u}{\partial t} + \Delta^2 u &= f  \text{ in } \Omega \times (0,T), \label{biLapl}\\
    u= \frac{\partial \Delta u}{\partial \nu} &= 0  \text{ on } \partial \Omega \times (0,T), \nonumber\\
    u(x,0) &= u_0 \text{ in } \Omega \nonumber.
\end{align}
Let $H := H_0^1(\Omega)$ with the standard scalar product 
    $(u,v)_H :=  \int_\Omega \grad u \cdot \grad v$
and define the subspace
\begin{equation*}
    V:= \left\{ v  \in H :  \frac{\partial}{\partial x_i} \Delta v \in L^2(\Omega), \,\, i=1, \dots, n \right\}, \qquad 
    \| v \|_V^2 := \| v \|^2_H + \sum_{i=1}^n \left\| \frac{\partial}{\partial x_i} \Delta v \right\|^2_{L^2(\Omega)}.
\end{equation*}
The duality pairing between $V^*$ and $V$ is defined by 
\begin{align}\label{pairing}
    \langle g, v\rangle_{V^*, \, V} := \langle g, -\Delta v\rangle_{H^{-1}(\Omega), \, H^1_0(\Omega)}.
\end{align}
We select $u_0 \in H$ and $f \in L^2(0,T; V^*)$. Taking $v\in V$, we can formally multiply \eqref{biLapl} by $-\Delta v$, integrate by parts and use \eqref{pairing} to obtain 
\begin{align}\label{weakform}
    \langle u'(t), v\rangle_{V^*, \, V} + \int_{\Omega} \nabla (\Delta u(t)) \cdot \nabla (\Delta v) = \langle f(t), v\rangle_{V^*, \, V}  \quad \forall v\in V.
\end{align}
By \cite[\S, Prop. 4.5]{lions2012non}, there exists $u\in L^2(0,T; V)$ with $u'\in L^2(0,T; V^*)$ such that \eqref{weakform} holds, i.e.,
\begin{align*}
    u'(t) + \Delta^2 u(t) = f(t) \quad \text{ in } V^*.
\end{align*}
If we assume more regularity on the forcing term, namely $f \in L^2(0,T; H_0^1(\Omega))$, and then set $v := - \Delta g$, \eqref{weakform} reads as
\begin{equation}\label{transEq}
    (u'(t), v)_{L^2(\Omega)} + (-\grad (\Delta u), \grad v)_{L^2(\Omega)} = (f(t), v)_{L^2(\Omega)}.
\end{equation}
So the equation holds weakly for every $v$ in the set $W:=\{v : v=-\Delta g \text{ for some } g\in V\}$, which contains $H^1(\Omega)$, hence \eqref{transEq}
holds for all $v \in H^1(\Omega)$. By \cite[\S 2, Sect. 9.9]{lions2012non}, $u$ satisfies \eqref{biLapl}.

\subsubsection{$H^{-1}(\Omega(t))$ pivot space}\label{sec:H-1pivotspace}
The choice of $H^{-1}$ as a pivot space appears in the study of very weak solutions of certain evolutionary problems following an idea of Brezis \cite{MR0394323}, see for example \cite[\S 2.3]{Lio69}, \cite[\S III, Example 6.C]{Showalter} and \cite{PMEInDual}; once we have introduced some notation, we will motivate the study through the porous medium equation. Inspired by this as well as the aforementioned literature, we consider the case of \[X(t) = L^p(\Omega(t)) \cap H^{-1}(\Omega(t)) \qquad \text{and} \qquad H(t)=H^{-1}(\Omega(t)), \qquad \text{with } p\in (1,+\infty)\]
on a bounded evolving domain $\{\Omega(t)\}_{t\in [0,T]}$ in $\R^n$, where $H^{-1}(\Omega(t))$ is the dual space of $H^1_0(\Omega(t))$ which we endow with the inner product
\begin{align*}
    (u, v)_{H^1_0(\Omega(t))} = \int_{\Omega(t)} \nabla u \cdot \nabla v.
\end{align*}
With $-\Delta_t \colon H^1_0(\Omega(t)) \to H^{-1}(\Omega(t))$ denoting the Dirichlet Laplacian on $\Omega(t)$, we endow the pivot space $H(t)$ with the inner product defined by\footnote{Given $u\in H^{-1}(\Omega(t))$, the function $(-\Delta_t)^{-1}u\in H^1_0(\Omega(t))$ is the unique weak solution $w$ of the elliptic problem
\begin{equation*}
    \begin{aligned}
        -\Delta_t w &= u &&\text{on $\Omega(t)$},\\
        w &= 0 &&\text{on $\partial \Omega(t)$}.
    \end{aligned}
\end{equation*}} 
\[(u,v)_{H(t)} := \langle u, (-\Delta_t)^{-1}v \rangle_{H^{-1}(\Omega(t)), H^1_0(\Omega(t))}.\]
We then identify $H(t)\equiv H(t)^*$ via the Riesz map (with respect to this inner product). The norm of $f\in X(t)$ is defined as 
        \begin{align*}
            \|f\|_{X(t)} = \|f\|_{L^p(\Omega(t))} + \|f\|_{H^{-1}(\Omega(t))},
        \end{align*}
and with this, $X(t)$ is a separable and reflexive Banach space. Observe that $X(t)\ctsDense H(t)$ as $X(t)$ contains $\mathcal D(\Omega(t))$. For simplicity of notation, we will  denote the Laplacian by
\[\mathrm{L}_t = -\Delta_t.\] 

\begin{remark}
Some important observations are timely:
\begin{enumerate}[label=(\roman*)]
    \item  In contrast to the previous section, we do \textbf{not} identify $H^1_0(\Omega(t))$ with $H^{-1}(\Omega(t))$ via the Riesz map, but rather $H^{-1}(\Omega(t))$ with \emph{its} dual.

    \item   The inner product above indeed defines a norm on $H^{-1}(\Omega(t))$ which is equivalent to the usual dual norm. 
    \item Since $p\in (1,\infty)$ and $\mathrm{L}_t$ is uniformly elliptic we have the regularity 
        \begin{align}\label{eq:Lpregularity}
            u\in L^p(\Omega(t)) \quad \Longrightarrow \quad \mathrm{L}_t^{-1}u \in W^{2,p}(\Omega(t)) \,\, \text{ with } \,\, \|\mathrm{L}_t^{-1} u\|_{W^{2,p}(\Omega(t))} \leq C \|u\|_{L^p(\Omega(t))}
        \end{align}
    by Calder\'{o}n--Zygmund theory for elliptic equations (see for instance \cite[\S 9.2]{Jost}). The constant $C>0$ above can be taken to be independent of $t$. 
    \item We identify $L^p(\Omega(t))$ with $L^{p'}(\Omega(t))^*$ so that, in rigour, 
    \begin{align*}
        X(t)= L^p(\Omega(t))\cap H^{-1}(\Omega(t)) \equiv (L^{p'}(\Omega(t))^*\cap H^{-1}(\Omega(t)) \quad \text{ and } \quad X^*(t) = L^{p'}(\Omega(t)) + H^{-1}(\Omega(t)).
    \end{align*}
    Given $f\in X(t)$ and $g=g_1+g_2\in X^*(t)$, the duality pairing is given by 
    \begin{align*}
        \left\langle g, f\right\rangle_{X^*(t), \, X(t)} = \langle g_1, f\rangle_{L^{p'}(\Omega(t)), \, L^p(\Omega(t))} + (g_2, f)_{H(t)} = \int_{\Omega(t)} g_1 f  + (g_2, f)_{H(t)}. 
    \end{align*}
    This identification of $L^p(\Omega(t))$ with $L^{p'}(\Omega(t))^*$ (giving rise to a second identification!) does not lead to any contradictions as we do not identify $X(t)$ with $X^*(t)$. In fact, 
    $X^*(t)$ is strictly larger than $X(t)$. 
    \item If $n=1,2$ we have $X(t)\equiv L^p(\Omega(t))$, but in higher dimensions this space is generally strictly smaller than $L^p(\Omega(t))$. Observe however that we have 
        \begin{align*}
            X(t) = L^p(\Omega(t)) \quad \text{ if } p\geq 2n/(n+2),
        \end{align*}
    as in this case the well-known Sobolev embedding $H^1_0(\Omega(t))\hookrightarrow L^{p'}(\Omega(t))$ holds.
\end{enumerate}
\end{remark}

As a change of notation, let $\psi_t\colon H^1_0(\Omega_0) \to H^1_0(\Omega(t))$ be the map that we called $\phi_t$ (defined in \eqref{eq:defnOfPhiEg}) in the previous examples. Note that since we are working over flat domains $\Omega(t)$, we have $\mathbf{A}_t^0=(\mathbf D\Phi_t^0)^\T \mathbf{D}\Phi_t^0$, which simplifies the formulae \eqref{eq:pullbackOfGradient}, \eqref{eq:pullbackOfLaplacian}, \eqref{eq:pushforwardOfLaplacian}. In particular, we note that $\mathbf{D}\Phi_t^0$ is invertible and
\begin{align*}
    \psi_t (\mathbf{D}\Phi_t^0) = (\mathbf D\Phi_0^t)^{-1}.
\end{align*}
We again assume the extra regularity in \eqref{eq:extraregularityH1} in order to use the results of the previous section. We now define 
$$\phi_t \colon H_0 \to H(t) \quad \text{ by } \quad \phi_t  := (\psi_{-t})^*.$$ 
The action of this map is as follows: given $f\in H_0$, $u\in H^1_0(\Omega(t))$, we have
\begin{equation}\label{eq:actionOfMapH-1}
\langle \phi_t f, u \rangle_{H(t), H^1_0(\Omega(t))} := \langle f, \psi_{-t} u \rangle_{H_0, H^1_0(\Omega_0)} = \int_{\Omega_0} \nabla \mathrm{L}_0^{-1} f \cdot \nabla \psi_{-t} u = \int_{\Omega(t)} \psi_t\left((J_t^0)^{-1}\mathbf{D}\Phi_t^0\sgrad \mathrm{L}_0^{-1} f\right)\cdot \nabla u,
\end{equation}
allowing us to identify 
\begin{align}\label{eq:formulaPhiH-1}
    \phi_t f = - \sgrad \cdot \psi_t \left( (J_t^0)^{-1}\mathbf{D}\Phi_t^0\sgrad \mathrm L_0^{-1} f\right).
\end{align}
Analogously, for $g\in H^{-1}(\Omega(t))$, we have $$\phi_{-t} g = - \sgrad \cdot  \left(J_t^0\mathbf{D}\Phi_t^0 \psi_{-t} (\sgrad \mathrm L_t^{-1} g)\right).$$
We can perform similar calculations to compute the adjoint maps $\phi_t^A$ and $\phi_{-t}^A$: given $u\in H(t)$ and $v\in H_0$, 
\begin{align*}
    (u, \phi_t v)_{H(t)} 
    = \int_{\Omega(t)} \psi_t((J_t^0)^{-1}\mathbf{D}\Phi_t^0) \nabla (\psi_t \mathrm L_0^{-1} v) \cdot \nabla \mathrm{L}_t^{-1} u 
    &= \int_{\Omega_0} \nabla (\mathrm L_0^{-1} v) \cdot \nabla (\psi_{-t}\mathrm L_t^{-1} u) \\
    &= \langle v, \psi_{-t}\mathrm L_t^{-1} u \rangle_{H^{-1}(\Omega_0), \, H^1_0(\Omega_0)} \\
    &= (v, \mathrm L_0 \psi_{-t} \mathrm L_t^{-1} u)_{H_0},
\end{align*}
from where we obtain that $\phi_t^A \colon H(t) \to H_0$ and $\phi_{-t}^A \colon H_0 \to H(t)$ satisfy
\begin{align}
    \phi_t^A u &= \mathrm L_0 \psi_{-t} \mathrm L_t^{-1} u\quad\text{and}\quad  \phi_{-t}^A u = \mathrm L_t \psi_{t} \mathrm L_0^{-1} u.\label{eq:adjoint_H-1}
\end{align}
Due to \eqref{eq:Lpregularity} these also satisfy
\begin{align*}
    \phi_{t}^A|_{X(t)} \colon X(t) \to X_0 \quad \text{ and } \quad \phi_{-t}^A|_{X_0} \colon X_0 \to X(t).
\end{align*}
It is important to note that since we identify $H^{-1}(\Omega(t))$ with its dual, the maps $\phi_t^*$ are also defined with $\phi_t^*\colon H^{-1}(\Omega(t)) \to H^{-1}(\Omega_0)$, and up to composition with the Riesz map and its inverse they coincide with $\phi_t^A$ calculated above. In particular, the map $\phi_t^* = (\psi_{-t}^*)^*$ is not the same as $\psi_{-t}$. This is another manifestation of the fact that we are not identifying $H^1_0$ with its dual. 

We observe also that, if $f\in X_0$, then $f\in L^p(\Omega_0)$ and due to \eqref{eq:Lpregularity} we have $\mathrm L_0^{-1} f\in W^{2,p}(\Omega_0)$. In particular, we can integrate by parts in \eqref{eq:actionOfMapH-1} to obtain, for $u\in H^1_0(\Omega(t))\cap L^{p'}(\Omega(t))$, the simpler formula
\begin{align}\label{eq:L1distribution_H-1}
    \langle \phi_t f, u \rangle_{H(t), H^1_0(\Omega(t))} 
    = \int_{\Omega_0} f \psi_{-t} u = \int_{\Omega(t)} J^t_0\psi_t f u .
\end{align}
Under Assumption \ref{ass:evolv_spaces_examples} and \eqref{eq:extraregularityH1}, it follows that the pairs $(H^{-1}, \phi_t)_{t \in [0,T]}$ and $(L^p\cap H^{-1}, \phi_t)_{t \in [0,T]}$ are compatible. The proof of the next lemma is complicated and is given in \S \ref{sec:proofs}. In this example we need to assume the additional regularity
\begin{align*}
\Phi_0^{(\cdot)} \in C^2\left([0,T]; C^2(\R^{d}, \R^{d})\right).
\end{align*}
\begin{restatable}{lem}{defnOfLambdaHMO}\label{lem:defnOfLambdaHMO}Under Assumption \ref{ass:evolv_spaces_examples} and \eqref{eq:extraregularityH1}, we have
\begin{align*}
    \lambda(t; u,v) &= \int_{\Omega(t)} \mathbf{H}(t) \nabla (\mathrm{L}_t^{-1}u) \cdot \nabla (\mathrm{L}_t^{-1} v).
\end{align*}
\end{restatable}

Thus the definition of a weak time derivative is the following.
\begin{definition}[$H^{-1}(\Omega)$ weak time derivative]
A function $u\in L^p_{X}$ has a weak time derivative $\dot u \in L^q_{X^*}$ if and only if 
\begin{align*}
\int_0^T \langle \dot u(t), \eta(t) \rangle_{X^*(t), X(t)} &= - \int_0^T (u(t), \dot{\eta}(t))_{H(t)} - \int_0^T \int_{\Omega(t)} \mathbf{H}(t) \nabla (\mathrm{L}_t^{-1}u) \cdot \nabla (\mathrm{L}_t^{-1} \eta) \quad \forall \eta\in \mathcal{D}_{X}.
\end{align*}

\end{definition}

We can finally conclude.

\begin{restatable}{prop}{eseHMO}
Under Assumption \ref{ass:evolv_spaces_examples} and \eqref{eq:extraregularityH1}, for any $p, q\in [1,\infty]$, there exists an evolving space equivalence between $\mathcal{W}^{p,q}(X_0, X_0^*)$ and $\W^{p, q}(X, X^*)$.
\end{restatable}


\paragraph{Applications.}
Let us motivate, again in the simpler case of a fixed domain, this choice of pivot space by giving more details for the \emph{porous medium equation} (PME) as considered in \cite[\S III, Example 6.C]{Showalter}:
\begin{equation}\label{eq:PME}
    \begin{aligned}
    u' - \Delta \Psi(u) &= f&&\text{on $(0,T) \times \Omega$},\\
    \Psi(u) &=0 &&\text{on $(0,T) \times \partial \Omega$},\\
    u(0) &= u_0 &&\text{on $ \Omega$},
    \end{aligned}
\end{equation}
where $\Psi(u) := |u|^{m-1}u$ (or an appropriate generalisation) with $m:=p-1$ and $f \in L^{p'}(0,T;X^*)$ for $X=H^{-1}(\Omega) \cap L^p(\Omega)$. If we take the inner product of the equation in $H^{-1}(\Omega)$ with an element 
$g \in L^p(\Omega) \cap H^{-1}(\Omega),$ we get
\[(u'(t), g)_H + \int_\Omega \Psi(u(t))g= \int_\Omega (-\Delta)^{-1}f(t)g \qquad \forall g \in X.\]
Suppose that $p \geq 2n\slash (n+2)$ so that $X=L^p(\Omega)$. Define $\tilde f \in L^{p'}(0,T;X^*)$ by
\[\langle \tilde f(t), v \rangle := \int_\Omega f(t)v \quad \text{for $v \in X$}\]
and $A\colon X \to X^*$ and $\mathcal{B}\colon X \to X^*$  by 
\[\langle A(u),v \rangle := \int_\Omega \Psi(u)v \qquad\text{and for $u,v \in H$,}\qquad \langle \mathcal{B}u, v \rangle := (u,v)_{H},\]  it follows by \cite[Proposition 6.2, \S III.6]{Showalter} that there is a unique $u \in L^p(0,T;X)$ with $(\mathcal{B}u)' \in L^{p'}(0,T;X^*)$ such that
\[(\mathcal{B}u(t))' + A(u(t)) = \tilde f(t) \quad \text{in $X^*$}.\]
Since $\langle (\mathcal{B}u(t))', v \rangle = (u'(t), v)_H$, this implies that
\[\langle u'(t), (-\Delta)^{-1}g\rangle + \langle \Psi(u(t)), g \rangle = \langle \tilde f(t), g \rangle  \qquad \forall g \in X.\]
Setting $v:=(-\Delta)^{-1}g$, we get existence of solutions for the very weak formulation of \eqref{eq:PME}:
\[\int_\Omega u'(t)v  + \Psi(u(t))(-\Delta)v = \int_\Omega f(t)(-\Delta)v \qquad \forall v \in H^1_0(\Omega) : \Delta v \in L^p(\Omega).\]
Under the additional regularity $f \in L^{p'}(0,T;H)$, replacing the definition of $\tilde f$ above by
\[\langle \tilde f(t), v \rangle := \int_\Omega v(-\Delta)^{-1}f(t) \quad \text{for $v \in X$},\]
so that $\tilde f \in L^{p'}(0,T;H^1_0(\Omega))$, then by \cite[Corollary 6.2 and Proposition 6.3, \S III.6]{Showalter} we have existence of the equation in $L^{p'}(0,T;H)$ and $\Psi(u) \in L^{p'}(0,T;H^1_0(\Omega))$ (so the boundary condition is satisfied). The equation in \eqref{eq:PME} holds pointwise a.e. in time in $H^{-1}(\Omega)$ and the initial condition is satisfied in the sense that $u(t) \to u_0$ as $t \to 0$ in $H^{-1}(\Omega)$. This concept of solution is called the $H^{-1}$-solution of the PME. See \cite[\S 6.7]{MR2286292} in this context. 

Of a similar form to this problem is the \emph{Stefan problem} on a moving domain $\{\Omega(t)\}_{t\in [0,T]}$:
    \begin{align*}
        \dot e - \Delta_g u + e\sgrad\cdot\mathbf{w} &= f \,\,\, \text{ in } \Omega(t), \\
        e(0) &= e_0, \\
        e&\in \mathcal{E}(u),
    \end{align*} 
    where the maximal monotone graph $\mathcal{E}$ is defined via
    \begin{align*}
        \mathcal{E}(r) = 
        \begin{cases} 
        r &\text{ for } r<0 \\
        [0,1] &\text{ for } r=0 \\
        r+1 &\text{ for } r>0
        \end{cases},
    \end{align*}
    which was considered by the first and final authors in \cite{AlpEll15}. For $f\in L^1_{L^1}$ and $e_0\in L^1(\Omega_0)$, the authors look for $u, e\in L^1_{L^1}$, and for $f\in L^\infty_{L^\infty}$ and $e_0\in L^\infty(\Omega_0)$ one looks for $u\in L^2_{H^1}$ and $e\in L^\infty_{L^\infty}$.

\subsection{Non-Gelfand triple examples}\label{sec:nongelf}

In the previous examples we obtained the definition of the weak derivative for three different cases in which there is a pivot Hilbert space, whose inner product structure we could exploit to establish the evolving space equivalence property of the evolving Sobolev--Bochner spaces. To conclude this section, we now consider several examples in which we do not assume the existence of a pivot space. We fix, for all the examples below, 
\[r\in (1,2).\]
Again, all proofs are relegated to \S \ref{sec:proofs}.

\subsubsection{$L^r(\Gamma(t))\hookrightarrow L^1(\Gamma(t))$}
The simplest example one can consider is obtained by taking $X(t)=L^r(\Gamma(t))$ and $Y(t)=L^1(\Gamma(t))$, where the evolution of $\{\Gamma(t)\}$ is determined by the flow map \eqref{eq:defnOfPhiEg}. As in Remark \ref{rem:afterWeakDerDefn}, we have $\Pi_t = \text{Id}_{X_0}$ for all $t$, and it is immediate to see:

\begin{restatable}{lem}{nongelf11}\label{lem:nongelf11}
Under Assumption \ref{ass:evolv_spaces_examples}, we have
\begin{align*}
    \lambda(t; u,v) &= 0.
\end{align*}
\end{restatable}

We then have the usual integration by parts formula:

\begin{definition}
A function $u\in L^p_{L^r}$ has a weak time derivative $\dot u \in L^q_{L^1}$ if and only if 
\begin{align*}
\int_0^T\int_{\Gamma(t)}\dot u(t) \eta(t) &= - \int_0^T \int_{\Gamma(t)} u(t) \dot{\eta}(t) \quad \forall \eta\in \mathcal{D}_{X}.
\end{align*}
\end{definition}

It follows immediately that:

\begin{restatable}{prop}{nongelf12}
Under Assumption \ref{ass:evolv_spaces_examples}, for any $p, q\in [1,\infty]$, there exists an evolving space equivalence between $\mathcal{W}^{p,q}(L^r(\Gamma_0), L^1(\Gamma_0))$ and $\W^{p, q}(L^r, L^1)$.
\end{restatable}

\subsubsection{$W^{1,r}(\Gamma(t))\hookrightarrow L^1(\Gamma(t))$}
Consider $X(t) = W^{1,r}(\Gamma(t))$ and $Y(t)=L^1(\Gamma(t))$ where the flow maps of each are defined as in \eqref{eq:defnOfPhiEg} but are different to each other, say 
\begin{align*}
    \phi_t^X u = u\circ \Phi_0^t \quad \text{ and } \quad \phi_t^Y = u \circ \widetilde{\Phi}_0^t,
\end{align*}
where $\Phi_t^0$ and $\widetilde{\Phi}_t^0$ are flows determined by given velocity fields $\mathbf{w}$ and $\widetilde{\mathbf{w}}$, respectively. We assume that these have the same normal component (indeed they must otherwise the surfaces will be different) but with potentially different tangential parts, say $\mathbf w_\tau$ and $\widetilde{\mathbf w}_\tau$. In general, we denote quantities of interest (such as the determinant of the Jacobian) using the notation $\widetilde{(\cdot)}$ for the corresponding quantity derived from $\widetilde \Phi^t_0$. We have the following expression for the extra term in the definition of the weak derivative.

\begin{restatable}{lem}{nongelf21}\label{lem:nongelf21}
Under Assumption \ref{ass:evolv_spaces_examples}, we have
\begin{align*}
    \lambda(t;u,v)&= \int_{\Gamma(t)} \left(\phi_t^X(\mathbf D\Phi^0_t)^\T \sgrad u\right) \cdot  \left(\mathbf D\Phi_0^t\left(\widetilde{\mathbf w}_\tau(t)- \mathbf w_\tau(t)\right)  \right)  v.
\end{align*}
\end{restatable}

Therefore in this case:

\begin{definition}
A function $u\in L^p_{W^{1,r}}$ has a weak time derivative $\dot u \in L^q_{L^1}$ if and only if 
\begin{align*}
\int_0^T \int_{\Gamma(t)} \dot u(t) \eta(t) &= - \int_0^T \int_{\Gamma(t)} u(t) \dot{\eta}(t) +  \int_0^T\int_{\Gamma(t)} \left(\phi_t^X(\mathbf D\Phi^0_t)^\T \sgrad u\right) \cdot  \left(\mathbf D\Phi_0^t\left(\widetilde{\mathbf w}_\tau(t)- \mathbf w_\tau(t)\right)  \right)  \eta(t) \quad \forall \eta\in \mathcal{D}_{X}.
\end{align*}
\end{definition}

\begin{remark}
Observe that in the case where $\mathbf w$ and $\widetilde {\mathbf w}$ have the same tangential component, we do indeed recover the situation of the previous example. 
\end{remark}

It is useful to note here that
\begin{align*}
    \Pi_t \colon W^{1,r}(\Gamma_0) \to L^1(\Gamma_0), \quad & \Pi_t u = u \circ \Phi_0^t \circ \widetilde \Phi_t^0.
\end{align*}

\begin{restatable}{prop}{nongelf22}
Under Assumption \ref{ass:evolv_spaces_examples}, for any $p, q\in [1,\infty]$, there exists an evolving space equivalence between $\mathcal{W}^{p,q}(W^{1,r}(\Gamma_0), L^1(\Gamma_0))$ and $\W^{p, q}(W^{1,r}, L^1)$.
\end{restatable}

\subsubsection{$W_0^{2,r}(\Omega(t))\hookrightarrow W_0^{1,1}(\Omega(t))$}

As a final example we take $X(t) = W_0^{2,r}(\Omega(t))$ and $Y(t)=W_0^{1,1}(\Omega(t))$  under the same assumptions as the previous case. In this case, $\Pi_t$ has the same formula as above, but we note that
\begin{align*}
    \pi(t; u, v) = \left\langle \Pi_t u, v\right\rangle_{W^{1,1}_0(\Omega_0), \, W^{-1, \infty}(\Omega_0)}
\end{align*}
and so we need a representation for elements of $W^{-1,\infty}(\Omega_0)$. By \cite[Proposition 9.20]{Bre11}, given $f\in W^{-1,\infty}(\Omega_0)$, there exist $f_1, \dots, f_n\in L^\infty(\Omega_0)$ such that 
\begin{align}\label{eq:actionex3}
    \left\langle f, u\right\rangle_{W^{-1,\infty}(\Omega_0), \, W^{1,1}_0(\Omega_0)} = -\sum_{i=1}^n \int_{\Omega_0} f_i D_i u.
\end{align}
In other words, writing $\underline{\mathbf f}=(f_1,\dots, f_n)$, the functional $f$ acts on $W^{1,1}$ as the operator $\nabla\cdot \mathbf \underline{\mathbf f}$; in what follows we always identify $f\equiv \underline{\mathbf f}$ and define its action on $W^{1,1}_0(\Omega_0)$ by \eqref{eq:actionex3}. 
\newcommand{\vv}{\underline{\mathbf v}}
\begin{restatable}{lem}{nongelf31}\label{lem:nongelf31}
Under Assumption \ref{ass:evolv_spaces_examples}, we have
\begin{align*}
    \lambda(t; u, \vv) = \int_{\Omega(t)} \vv \cdot \nabla \left( (\mathbf D \Phi_0^t)^{-\T} \nabla u \cdot \mathbf D\Phi_0^t \left( \widetilde{\mathbf w}_\tau(t)-\mathbf w_\tau(t)  \right) \right).
\end{align*}
\end{restatable}
This leads to the definition:

\begin{definition}
A function $u\in L^p_{W_0^{2,r}}$ has a weak time derivative $\dot u \in L^q_{W_0^{1,1}}$ if and only if 
\begin{align*}
    \int_0^T \int_{\Omega(t)} \dot u \,  \underline{\pmb\eta} = - \int_0^T \int_{\Omega(t)} u \, \dot{\underline{\pmb\eta}} \, - \int_0^T \int_{\Omega(t)}\nabla \left( (\mathbf D \Phi_0^t)^{-\T} \nabla u \cdot \mathbf D\Phi_0^t \left( \widetilde{\mathbf w}(t)-\mathbf w(t)  \right) \right) \cdot \underline{\pmb\eta}  \quad \forall \underline{\pmb{\eta}}\in \mathcal D_{W^{-1, \infty}}.
\end{align*}
\end{definition}

Similar calculations as before lead to the main result:

\begin{restatable}{prop}{nongelf32}
Under Assumption \ref{ass:evolv_spaces_examples}, for any $p, q\in [1,\infty]$, there exists an evolving space equivalence between $\mathcal{W}^{p,q}(W^{2,r}_0(\Omega_0), W^{1,1}_0(\Omega_0))$ and $\W^{p, q}(W^{2,r}_0, W^{1,1})$.
\end{restatable}

\begin{remark}
The techniques of the previous examples can be extended to deal with the case of  $\mathbb W^{p,q}(W^{k,r}, W^{k-1, 1})$, $k\geq 2$. 
\end{remark}

\subsection{Proofs of evolving space equivalence}\label{sec:proofs}

We now provide the proofs of the results stated in \S \ref{sec:gelf} and \S \ref{sec:nongelf}. For readability we restate all the results.

\subsubsection{$L^2$ pivot space}

\gelfoneone*
\begin{proof}
Since we have a Gelfand triple structure, by Remark \ref{rem:observationsHatBetc}, the evolution of the duality pairing has the form
\begin{equation*}
    \pi(t;u, v) = \int_{\mathcal{M}_0} u_0 v_0 J_t^0.
\end{equation*}
By simply differentiating and using the formula \eqref{eq:derivativeOfJ0t} for differentiating the determinant of the Jacobian and then pushing forward, we obtain the desired expression.
\end{proof}

\eseLTwo*
\begin{proof}
From \S \ref{sec:L2pivotspace}, we see that $\Pi_ t\colon H_0 \to H_0$ 
is defined by 
$\Pi_t u := u J_t^0$ with inverse $\Pi_t^{-1} u :=u\slash J_t^0$. The regularity assumptions on the velocity field imply that $J^0_{(\cdot)}, (J_{(\cdot)}^0)^{-1} \in C^1([0,T], C^1(\mathbb R^d, \mathbb R^d))$ and hence
\begin{equation*}
    \|\Pi_t u \|_{W^{1,r}(\mathcal{M}_0)}  \leq C\| u \|_{W^{1,r}(\mathcal{M}_0)}
\end{equation*}
where $C$ depends on the $L^\infty(0,T;W^{1,\infty}(\mathcal{M}_0))$ norm of $J^0_t$. We can prove in the same way that the inverse $\Pi_t^{-1}$ is bounded as well.  It is not difficult to check that $\Pi^{-1}\colon \mathcal{W}^{p,q}(X_0, X_0) \to \mathcal{W}^{p,q}(X_0,X_0)$ due to the smoothness assumptions on $\Phi^0_t$ and hence the evolving space equivalence holds by Theorem \ref{thm:ESEGelfandTriple}.
\end{proof}

\subsubsection{$H^1$ pivot space}

\gelftwoone*

\begin{proof}
We see that for $u, v \in H_0,$ by using the formula  \eqref{eq:transportFormulaGradient} for differentiating the Dirichlet energy,
\begin{align*}
    \hat\lambda(t; u,v) 
    &= \frac{d}{dt}\left(\int_{\mathcal{M}(t)}\phi_t u \phi_t v + \sgrad \phi_t u^{\T} \sgrad \phi_t v\right) = \int_{\mathcal{M}(t)}\phi_t u\phi_t v\sgrad \cdot \mathbf{w}(t) + \sgrad \phi_t u^{\T} \mathbf{H}(t)\sgrad \phi_t v.
\end{align*}
This then immediately implies the result. 
\end{proof}

The proof of Proposition \ref{prop:eseH1} (evolving space equivalence between the spaces $\mathcal{W}^{p,q}(W^{2,r}(\mathcal{M}_0), W^{2,r}(\mathcal{M}_0)^*)$ and $\W^{p, q}(W^{2,r}, (W^{2,r})^*)$) requires us to check the conditions of Theorem \ref{thm:ESEGelfandTriple}, which we will do now in a series of lemmas.

Firstly, writing $(\Pi_tu,v)_{H_0} = (\phi_t u, \phi_t v)_{H(t)}$ and at the same time expanding the inner product on the left-hand side, 
\begin{align*}
    (\Pi_tu, v)_{H_0} = \int_{\mathcal{M}_0}\Pi_tuv + \sgrad \Pi_tu \cdot  \sgrad v &=  
    \int_{\mathcal{M}_0} u vJ^0_t + (\mathbf{D}\Phi_t^0(\mathbf{A}^0_t)^{-1}\sgrad u)^{\T} \mathbf D\Phi_t^0(\mathbf{A}^0_t)^{-1}\sgrad v J_t^0 \\ 
    &= \int_{\mathcal{M}_0} u vJ^0_t + \sgrad u^{\T} \mathbf{B}_t^0\sgrad v,
\end{align*}
where we denoted 
\begin{equation}\label{operatorB}
\mathbf{B}_t^0 := (\mathbf{A}^0_t)^{-\T}(\mathbf D\Phi_t^0)^\T \mathbf{D} \Phi_t^0 (\mathbf{A}_t^0)^{-1} J_t^0. 
\end{equation}
By comparing these two expressions, we are able to obtain relevant properties of $\Pi_t.$

\begin{lem}
Under Assumption \ref{ass:evolv_spaces_examples} and \eqref{eq:extraregularityH1}, we have $\Pi_t\colon X_0 \to X_0$ is a bounded linear map.
\end{lem}

\begin{proof}
Given $u \in X_0$, setting $w=\Pi_t u,$ we have by the above displayed equation
\begin{align}
    \int_{\mathcal{M}_0} wv + \sgrad w^{\T} \sgrad v = \int_{\mathcal{M}_0} uv J^0_t + \sgrad u^{\T} \mathbf{B}_t^0\sgrad v \quad \forall v \in H_0.\label{eq:TtEquationH1example}
\end{align}
As a function of $v$, the right-hand side is clearly an element of $H^1(\mathcal{M}_0)^*$, so by the Lax--Milgram lemma, there exists a unique $w \in H^1(\mathcal{M}_0)$ satisfying the above equation. By smoothness, we can rewrite this as
\[
\int_{\mathcal{M}_0} wv + \sgrad w^{\T} \sgrad v = \int_{\mathcal{M}_0} (u J^0_t - \sgrad \cdot (\mathbf{B}_t^0\sgrad u))v \quad \forall v \in H_0,\]
i.e., $w$ is a weak solution $w-\Delta_\Gamma w = (u J^0_t -  \sgrad \cdot (\mathbf{B}_t^0\sgrad u)) \in L^r(\Gamma_0)$. 
We may apply elliptic regularity theory (by making use of the usual estimates, e.g. \cite[\S 9.2]{Jost} on Euclidean balls and using a patching argument to extend to the manifold case if $\mathcal{M}(t) = \Gamma(t)$, as is standard)  to this variational formulation to deduce that $w \in W^{2,r}(\mathcal M_0)$ as well as
\[\norm{w}{W^{2,r}(\mathcal M_0)} \leq C\norm{u J^0_t - \sgrad \cdot (\mathbf{B}_t^0\sgrad u)}{L^r(\mathcal M_0)}.\qedhere\]
\end{proof}

\begin{lem}
Under Assumption \ref{ass:evolv_spaces_examples} and \eqref{eq:extraregularityH1}, the map $\Pi_t\colon X_0 \to X_0$ is invertible with uniformly bounded inverse with $t \mapsto \Pi^{-1}_t w$ measurable. Hence $\Pi^{-1}\colon L^r(0,T;X_0) \to L^r(0,T;X_0).$
\end{lem}


\begin{proof}
In this case, one needs to show that given $w \in X_0$, there exists $u \in X_0$ such that \eqref{eq:TtEquationH1example} holds and the proof is almost identical to the previous lemma after realising that the right-hand side of \eqref{eq:TtEquationH1example} is an equivalent inner product on $H_0$. The measurability follows because $J^0_t$ and $\mathbf{B}_t^0$ are continuous.
\end{proof}

\begin{lem}
Under Assumption \ref{ass:evolv_spaces_examples} and \eqref{eq:extraregularityH1}, we have $\Pi^{-1}\colon \mathcal W^{p,q}(X_0,X_0) \to \mathcal W^{p, \, p\land q}(X_0,X_0)$ for any $p,q\in [1,\infty]$.
\end{lem}

\begin{proof}
We shall first show that $\Pi^{-1}\colon C^1([0,T];X_0) \to \mathcal{W}^{p, \, p\land q}(X_0,X_0)$ and then extend by density. Take $w \in C^1([0,T];X_0)$ and set $u=\Pi^{-1}w$. We have that $u(t)$ satisfies 
\begin{align*}
    \int_{\mathcal{M}_0} u(t)vJ^0_t + \sgrad u(t)^{\T} \mathbf{B}_t^0\sgrad v  = \int_{\mathcal{M}_0} w(t)v + \sgrad w(t)^{\T} \sgrad v\quad \forall v \in H_0.
\end{align*}
Taking the difference at times $t+h$ and $t$, this becomes, for all $v \in H_0$,
\begin{align*}
    \int_{\mathcal{M}_0} \delta_h u(t) v J^0_{t+h} + u(t)\delta_h J_t^0 v + \sgrad \delta_h u(t)^{\T} \mathbf{B}_{t+h}^0\sgrad v &+ \sgrad u(t)^{\T} \delta_h \mathbf{B}_t^0\sgrad v = \int_{\mathcal{M}_0} \delta_h w(t) v + \sgrad \delta_h w(t)^{\T}\sgrad v.
\end{align*}
Now, adding and subtracting $y(t)$ where $y(t)$ is defined as the solution of
\begin{equation}
\int_{\mathcal{M}_0}y(t)vJ^0_t + \sgrad y(t) \mathbf{B}_t^0\sgrad v = \int_{\mathcal{M}_0}w'(t)v + \sgrad w'(t) \sgrad v - u(t)v(J^0_t)' - \sgrad u(t) (\mathbf{B}_t^0)'\sgrad v \quad \forall v \in H_0\label{eq:defnOfGEquation},
\end{equation}
we obtain
\begin{align*}
    &\int_{\mathcal{M}_0} \left(\delta_h u(t)-y(t)\right)vJ^0_{t+h} + u(t)\delta_h J_t^0v + \sgrad \left(\delta_h u(t)^{\T} - y(t)\right) \mathbf{B}_{t+h}^0\sgrad v + \sgrad u(t)^{\T} \delta_h\mathbf{B}_t^0\sgrad v \\
    &\quad+ \int_{\mathcal{M}_0} y(t)vJ^0_{t+h} + \sgrad y(t)\mathbf{B}_{t+h}^0\sgrad v = \int_{\mathcal{M}_0} \delta_h w(t) v + \sgrad \delta_h w(t)^{\T} \sgrad v.
\end{align*}
Observe that, using the definition of $y(t)$, the final term on the left-hand side is
\begin{align*}
    \int_{\mathcal{M}_0} y(t)vJ^0_{t+h} + \sgrad y(t)\mathbf{B}_{t+h}^0\sgrad v &= \int_{\mathcal{M}_0} y(t)v(J^0_{t+h}-J^0_t) + \sgrad y(t)(\mathbf{B}_{t+h}^0-\mathbf{B}_t^0)\sgrad v\\
    &\quad+ \int_{\mathcal{M}_0}w'(t)v + \sgrad w'(t) \sgrad v - u(t)(J^0_t)'v - \sgrad u(t) (\mathbf{B}_t^0)'\sgrad v,
\end{align*}
so the above becomes
\begin{align*}
    \int_{\mathcal{M}_0} &\left(\delta_h u(t) -y(t)\right)vJ^0_{t+h} + u(t)\left(\delta_h J^0_t-(J^0_t)'\right)v + \sgrad \left(\delta_h u(t)^{\T} - y(t)\right) \mathbf{B}_{t+h}^0\sgrad v\\
    &+ \int_{\mathcal{M}_0} \sgrad u(t)^{\T}\left(\delta_h\mathbf{B}_t^0-(\mathbf{B}_t^0)'\right)\sgrad v + \int_{\mathcal{M}_0} y(t)v(J^0_{t+h}-J^0_t) + \sgrad y(t)(\mathbf{B}_{t+h}^0-\mathbf{B}_t^0)\sgrad v\\
    &\quad\quad= \int_{\mathcal{M}_0} \left(\delta_h w(t)-w'(t)\right)v + \sgrad \left(\delta_h w(t)^{\T} - \sgrad w'(t)\right) \sgrad v.
\end{align*}
Taking $v=\delta_h u(t)-g$ and using Young's inequality with $\epsilon$ multiple times, we find
\begin{align*}
C\norm{\delta_h u(t)-y(t)}{H_0}^2 &\leq \norm{\delta_h w(t) -w'(t)}{H_0}^2 + \norm{\delta_h J^0_t-(J^0_t)'}{L^\infty(\mathcal{M}_0)}^2\norm{u(t)}{}^2 \\ 
&\quad + \norm{\delta_h \mathbf{B}_t^0- (\mathbf{B}_t^0)'}{L^\infty(\mathcal{M}_0)}^2\norm{\sgrad u(t)}{L^2(\mathcal{M}_0)}^2 + \norm{J^0_{t+h}-J_t^0}{L^\infty(\mathcal{M}_0)}^2 \norm{y(t)}{L^2(\mathcal{M}_0)}^2 \\ &\quad +\norm{\mathbf{B}_{t+h}^0-\mathbf{B}_t^0}{L^\infty(\mathcal{M}_0)}^2\norm{\sgrad y(t)}{L^2(\Omega)}^2,
\end{align*}
which shows that $u$ is strongly differentiable; more precisely, $u=\Pi^{-1}w \in C^1([0,T];H_0)$ with $u' = g$.

By the same reasoning as the previous lemma applied to the weak formulation for $y(t)$ (see \eqref{eq:defnOfGEquation}), we obtain in fact that
\[\norm{(\Pi_t^{-1}w (t))'}{W^{2,r}(\mathcal{M}_0)} \leq C(\norm{w'(t)}{W^{2,r}(\mathcal{M}_0)} + \norm{\Pi_t^{-1}w(t)}{W^{2,r}(\mathcal{M}_0)}).\]
Hence $\Pi^{-1}\colon C^1([0,T];X_0) \to \mathcal W^{p,\,p\land q}(X_0,X_0)$ is such that $\Pi^{-1}\colon \mathcal W^{p,q}(X_0,X_0) \to \mathcal W^{p, \, p\land q}(X_0,X_0)$ is bounded. By density, we obtain the result.
\end{proof}
\eseHOne*
\begin{proof}
Having checked all conditions of Theorem \ref{thm:ESEGelfandTriple} above, the result follows.
\end{proof}

\subsubsection{$H^{-1}$ pivot space}

To provide the expression for $\lambda$ in Lemma \ref{lem:defnOfLambdaHMO}, we now verify that Assumption \ref{ass:generalcaseChanges} is satisfied. Given $u \in H_0$, we must check that $\norm{\phi_t u}{H(t)}^2$ is differentiable. Define $w(t)\in H^1_0(\Omega(t))$ by 
\begin{align}\label{eq:defnOfW}
   \mathrm{L}_t w(t) = \phi_t u,
\end{align}
so that, as we argued above,
\begin{equation}\label{eq:gradientW}
    \norm{\phi_t u}{H(t)}^2 
    = \int_{\Omega(t)} |\grad_t w(t)|^2.
\end{equation}
Observe that the right-hand side of \eqref{eq:defnOfW} is clearly in $C^\infty_H$ with zero time derivative, and hence as is the left-hand side, i.e., $\mathrm{L} w \in C^\infty_H$ with
\[\md(\mathrm{L}w) = 0.\]
To prove that \eqref{eq:gradientW} is differentiable, we need $w$ itself to belong to $C^1_{H^1_0}$, which the next lemma shows is the case.  In the proof below we make use of the notation $\delta_h$ again to denote the difference quotient.

\begin{lem}\label{lem:identityForDualDerivative}
Under Assumption \ref{ass:evolv_spaces_examples} and \eqref{eq:extraregularityH1}, for $u \in H_0$, we have  $w \equiv \mathrm{L}_{(\cdot)}^{-1}\phi_{(\cdot)} u \in C^1_{H^1_0}$ and $\dot w$ satisfies, for all $t\in [0,T]$,
\[\int_{\Omega(t)}\grad \dot w(t) \cdot \grad \varphi = -\int_{\Omega(t)}\grad w(t)^{\T}\mathbf{H}(t)\grad  \varphi \qquad \forall \varphi\in H^1_0(\Omega(t)).\]
\end{lem}

\begin{proof}
Let us show that $w \in C^1_{H^1_0}$ by proving that $\tilde w:=\psi_{-(\cdot)}w \in C^1((0,T);H^1_0(\Omega_0))$. Due to 
\eqref{eq:pullbackOfGradient} and reusing the notation $\mathbf B_t^0$ from \eqref{operatorB}, we see from \eqref{eq:defnOfW} that $\tilde w$ satisfies
\begin{align*}
    \int_{\Omega_0}\grad \tilde w(t)^{\T}\mathbf{B}_t^0\grad\psi_{-t}\varphi = \langle \psi_{-t}^*u, \varphi \rangle = \langle u, \psi_{-t}\varphi \rangle \qquad 
\forall \varphi \in H^1_0(\Omega(t)).
\end{align*}
Hence
\begin{align}
    \int_{\Omega_0}\grad \tilde w(t)^{\T}\mathbf{B}_t^0\grad\eta  = \langle u, \eta \rangle \qquad \forall \eta \in H^1_0(\Omega_0).\label{eq:toGetSmoothness}
\end{align}
Take two times $t, s \geq 0$ and consider the difference of the above equality at those times:
\begin{align*}
    \int_{\Omega_0}\grad (\tilde w(t)^{\T}-\tilde w(s)^{\T})\mathbf{B}_t^0\grad\eta + \tilde w(s)^{\T}(\mathbf{B}_t^0-\mathbf{B}_s^0)\grad \eta  = 0.
\end{align*}
Taking $\eta = \tilde w(t)-\tilde w(s)$, this implies the bound
\[C\norm{\grad \tilde w(t)-\grad \tilde w(s)}{L^2(\Omega_0)} \leq \norm{\tilde w(s)}{L^2(\Omega_0)}\norm{\mathbf{B}_t^0-\mathbf{A}_s^0}{L^\infty(\Omega_0)},\]
and the right-hand side clearly tends to zero as $t \to s$, proving that $\tilde w \in C^0([0,T];H^1_0(\Omega_0)).$

Regarding the derivative, let $h >0$ and take the difference in \eqref{eq:toGetSmoothness} between times $t+h$ and $t$ and divide by $h$:
\begin{equation}\label{eq:differenceQuotientUW}
    \int_{\Omega_0} \grad\delta_h \tilde w(t) \mathbf{B}_{t+h}^0\grad\eta  + \int_{\Omega_0}\grad \tilde w(t)^{\T}\delta_h \mathbf{B}_t^0\grad\eta = 0.
\end{equation}
We now show that the difference quotient for $\tilde w(t)$ converges to the (unique) solution $v(t) \in H^1_0(\Omega(t))$ of 
\[\int_{\Omega_0}\grad v(t)^{\T}\mathbf{B}_t^0\grad \eta = -\int_{\Omega_0}\grad \tilde w(t)^{\T}(\mathbf{B}_t^0)'\grad \eta \qquad \forall \eta \in H^1_0(\Omega_0).\]
In \eqref{eq:differenceQuotientUW}, if we add and subtract the same term, we see
\begin{align*}
    \int_{\Omega_0}\grad (\delta_h\tilde w(t)- v(t))^{\T}\mathbf{B}_{t+h}^0\grad\eta  + \grad v(t)^{\T} \mathbf{B}_{t+h}^0 \grad \eta +  \grad \tilde w(t)^{\T}\delta_h \mathbf{B}_t^0\grad\eta = 0,
\end{align*}
and here adding and subtracting $\int_{\Omega_0}\grad v(t)^{\T} \mathbf{B}_t^0\grad \eta$ and using the equation defining $v(t),$ we end up with
\begin{align*}
    \int_{\Omega_0}\left(\nabla \delta_h \tilde w(t)-\grad v(t)\right)^{\T}\mathbf{B}_{t+h}^0\grad\eta  + \grad v(t)^{\T} \left(\mathbf{B}_{t+h}^0-\mathbf{B}_t^0\right) \grad \eta 
    + \grad \tilde w(t)^{\T}\left(\delta_h\mathbf{B}_t^0-(\mathbf{B}_t^0)'\right)\grad\eta = 0.
\end{align*}
Taking $\eta$ appropriately, using positive-definiteness and smoothness of $\mathbf{A}$, we get
\begin{align*}
    C\norm{\delta_h \nabla \tilde w(t)-\grad v(t)}{L^2(\Omega_0)}  \leq \norm{\grad v(t)}{L^2(\Omega_0)}&\norm{\mathbf{B}_{t+h}^0-\mathbf{B}_t^0}{L^\infty(\Omega_0)} + \norm{\grad \tilde w(t)}{L^2(\Omega_0)}\norm{\delta_h\mathbf{B}_t^0-(\mathbf{B}_t^0)'}{L^\infty(\Omega_0)},
\end{align*}
and in the limit $h \to 0$, the right-hand side tends to zero and hence
\[\delta_h \tilde w(t) \to v(t) \quad \text{in $H^1_0(\Omega_0)$}\]
but then we must have that $\tilde w'$ exists and $\tilde w' \equiv v.$ By considering the equation defining $v=\tilde w'$ and making a similar argument to how we showed that $\tilde w$ is continuous, we can show that $\tilde w' \in C^0([0,T];H^1_0(\Omega_0))$. Pushing forward the integrals defining $\tilde w'(t)$, 
we see that
\begin{align*}
    \int_{\Omega_0}\grad \tilde w(t)^{\T}(\mathbf{B}_t^0)'\grad \eta 
    &= \int_{\Omega(t)}J^t_0\grad w(t)^{\T}\psi_t(\mathbf{D}\Phi^0_t)\psi_t((\mathbf{B}_t^0)')\psi_t(\mathbf{D}\Phi^0_t)^{\T}\grad \varphi\\
    &= \int_{\Omega(t)}J^t_0\grad w(t)^{\T}(\mathbf{D}\Phi^t_0)^{-1}\psi_t((\mathbf{B}_t^0)')(\mathbf{D}\Phi^t_0)^{-\T}\grad \varphi.
\end{align*}
The identity in Lemma \ref{lem:actionOfMPrime} gives a simplification of the right-hand side above and provides the desired result. 
\end{proof}
\defnOfLambdaHMO*
\begin{proof}
Using the transport formula \eqref{eq:transportFormulaGradient} on \eqref{eq:gradientW} and plugging the result of the previous lemma in, we derive
\begin{align*}
    \frac{d}{dt}\norm{\phi_t u}{H(t)}^2 &= \int_{\Omega(t)}2\grad \dot w(t) \cdot \grad w(t) - \mathbf{H}(t) \grad w(t)\cdot \grad w(t)   = \int_{\Omega(t)} \mathbf{H}(t)\nabla w(t) \cdot \nabla w(t).
\end{align*}
We have then that
\begin{align*}
    \hat\lambda(t; u_0,v_0) &:= \frac{1}{4}\left(\frac{d}{dt}\norm{\phi_t(u_0 + v_0)}{H(t)}^2 - \frac{d}{dt}\norm{\phi_t(u_0 - v_0)}{H(t)}^2\right)\\
    &= \frac{1}{4}\Big(\int_{\Omega(t)} \mathbf{H}(t) \nabla z(t) \cdot \nabla z(t)- \int_{\Omega(t)} \mathbf{H}(t) \nabla y(t) \cdot \nabla y(t)\Big),
\end{align*}
where $z(t)$ and $y(t)$ are defined via
    $\mathrm{L}_t z(t) = \psi_{-t}^* (u_0+v_0)$ and $\mathrm{L}_t y(t) = \psi_{-t}^* (u_0-v_0).$
Defining also
\[\mathrm{L}_t w(t) = \psi_{-t}^* u_0\qquad\text{and}\qquad \mathrm{L}_t v(t) = \psi_{-t}^* v_0,\]
and using linearity, the above simplifies to
\begin{align*}
    \hat\lambda(t; u_0,v_0) = \int_{\Omega(t)} \mathbf{H}(t)\nabla w(t) \cdot \nabla v(t) = \int_{\Omega(t)} \mathbf{H}(t) \nabla (\mathrm{L}_t^{-1}\psi_{-t}^* u_0) \cdot \nabla (\mathrm{L}_t^{-1}\psi_{-t}^* v_0),
\end{align*}
and a simple calculation shows that Assumptions \ref{ass:generalcaseChanges} (ii), (iii) (see Remark \ref{rem:generalcase}) are also satisfied. Pushing forward to $\Omega(t)$ now yields the desired expression.
\end{proof}

We now check the evolving space equivalence result for this example again by verifying the assumptions of Theorem \ref{thm:ESEGelfandTriple}. 

\begin{lem}\label{lem:formulaPiH-1}
Under Assumption \ref{ass:evolv_spaces_examples} and \eqref{eq:extraregularityH1}, for $u\in H_0$, we have 
\[\Pi_t u = \mathrm{L}_0 \psi_{-t} \mathrm{L}_t^{-1} \phi_{t}u\]
and $\Pi_t\colon X_0 \to X_0$ is uniformly bounded and invertible  with uniformly bounded and measurable (in time) inverse.
\end{lem}

\begin{proof}
The formula follows directly from \eqref{eq:adjoint_H-1}. 
Recalling that $\phi_{t}$ (resp. $\phi_{-t}$) maps $X_0$ to $X(t)$ (resp. $X(t)$ to $X_0$) and is bounded, we can easily see that $\Pi_t\colon X_0 \to X_0$ is bounded due to the elliptic regularity of \eqref{eq:Lpregularity}. It also has an inverse defined by $\Pi_t^{-1} = \phi_{-t}\mathrm{L}_t\psi_t \mathrm{L}_0^{-1}$ with the same properties. Measurability of $\Pi_t u$, $\Pi_t^{-1} u$, for $u\in H_0$, follows from the fact that the composition of measurable maps is measurable. 
\end{proof}
Thus, we have the fulfilment of \eqref{ass:gtTtrange}, \eqref{ass:gtTtInvertible} and \eqref{ass:gtmeasurabilityOnTtAndInv}.
\begin{lem}
Under Assumption \ref{ass:evolv_spaces_examples} and \eqref{eq:extraregularityH1}, the conditions of Theorem \ref{thm:ESEGelfandTriple} are fulfilled. 
\end{lem}
\begin{proof}
We shall make use of the alternative criteria provided in Lemma \ref{lem:altGTCriteria} here to verify \eqref{ass:gtremainsInSpace}. First, as we already stated, note that $\left(H^1_0(\Omega(t)), \psi_t\right)_t$ is a compatible pair and furthermore, Assumption \ref{ass:generalcaseChanges} is satisfied (the associated operators $\pi^\psi$, $\hat \lambda^\psi$ satisfy the conditions in Remark \ref{rem:generalcase}). In this setting, we have\footnote{The inner product we use on $H^1_0(\Omega(t))$ has no lower order term and thus no term involving the divergence of the velocity appears in the expression defining $\hat\lambda^\psi$.} (see Definition \ref{eq:h1weakderiv})
\begin{align}\label{eq:lambdapsi}
    \hat\lambda^\psi(t; u, v) 
    = \int_{\Omega(t)}  \mathbf{H}(t)\nabla (\psi_t u) \cdot \nabla (\psi_t v).
\end{align}

\noindent \underline{\eqref{ass:newOneForcheck}}: Defining 
\[\xi_t := (\phi_{-t}^H)^A,\] it follows from separability of $H_0$ that $(H(t), \xi_t)_t$ is a compatible pair. We now observe that, for fixed $u\in H_0$, 
\begin{align*}
   \|\xi_t u\|_{H(t)}^2 = \|\mathrm{L}_t \psi_t \mathrm{L}_0^{-1} u\|_{H(t)}^2 = \|\psi_t \mathrm{L}_0^{-1} u\|_{H^1_0(\Omega(t))}^2,
\end{align*}
and this is continuously differentiable since $\left(H^1_0(\Omega(t)), \psi_t\right)_{t\in[0,T]}$ satisfies Assumption \ref{ass:generalcaseChanges}. The remaining points follow immediately; simply note that from the calculation above we obtain
\begin{align*}
    \pi^\xi(t; u, v) = (\psi_t \mathrm{L}_0^{-1}u, \psi_t\mathrm{L}_0^{-1} v)_{H^1_0(\Omega(t))} = \pi^\psi(t; \mathrm{L}_0^{-1}u, \mathrm{L}_0^{-1}v) \quad \Longrightarrow \quad \hat\lambda^\xi(t; u, v) = \hat\lambda^\psi(t; \mathrm{L}_0^{-1}u, \mathrm{L}_0^{-1}v).
\end{align*}

\  
 
\noindent \underline{\eqref{ass:rangeOfHatLambdaForNewCheck}}: If $u \in X_0$ and $v\in H_0$, then we have from using the relation between $\hat \lambda^\xi$ and $\hat \lambda^\psi$ and the formula for the latter in \eqref{eq:lambdapsi} that
\begin{align*}
    \hat\lambda^\xi(t; u, v) 
    &= -\left\langle\nabla\cdot(\mathbf{H}(t) \nabla (\psi_{t} \mathrm{L}_0^{-1}u)), \psi_{t}\mathrm{L}_0^{-1}v\right\rangle_{H^{-1}(\Omega(t)), \, H^1_0(\Omega(t))} \\
    &= -\left\langle \mathrm{L}_0^{-1} J_0^t \psi_{-t}\nabla\cdot(\mathbf{H}(t) \nabla (\psi_{t} \mathrm{L}_0^{-1}u)), v\right\rangle_{H^{-1}(\Omega_0), \, H^1_0(\Omega_0)}
\end{align*}
where we used \eqref{eq:L1distribution_H-1} (or rather the inverse of the expression given by that formula) since $u \in X_0$ and the fact that $\mathrm{L}_0^{-1}$ is self-adjoint in the above manipulation.
With this, we can identify
\begin{align*}
    \hat\Lambda^\xi(t)u = -\mathrm{L}_0^{-1} J_0^t \psi_{-t}\nabla\cdot(\mathbf{H}(t) \nabla (\psi_{t} \mathrm{L}_0^{-1}u))
\end{align*}
and as $u\in L^p(\Omega_0)$ then \eqref{eq:Lpregularity} implies that also $\hat\Lambda^\xi(t)u\in L^p(\Omega_0),$ proving the claim.

\ 

\noindent \underline{\eqref{ass:TtInvDualSmooth}}: We have already shown that $\Pi_t\colon X_0 \to X_0$ is a bijection with inverse given by 
    $\Pi_t^{-1} u = \phi_{-t} \mathrm L_t \psi_t \mathrm L_0^{-1} u,$
and from this formula we can immediately identify 
\begin{align*}
    (\Pi_t^{-1})^*\colon X_0^*\to X_0^*, \quad (\Pi_t^{-1})^* f = \mathrm L_0^{-1} \phi_{-t} \mathrm L_t \phi_{-t}^* f.
\end{align*}
In particular, if $f\in X_0$, we have
\begin{align*}
    (\Pi_t^{-1})^* f = \mathrm L_0^{-1} \phi_{-t} \mathrm L_t \phi_{-t}^* f = \mathrm L_0^{-1} \phi_{-t} \mathrm L_t \phi_{-t}^A f,
\end{align*}
which is bounded due to \eqref{eq:Lpregularity} and the formula in \eqref{eq:adjoint_H-1}. 
\end{proof}
Now an application of Theorem \ref{thm:ESEGelfandTriple} yields the following.
\eseHMO*

\subsubsection{$L^r(\Gamma(t))\hookrightarrow L^1(\Gamma(t))$}

As in Remark \ref{rem:afterWeakDerDefn}, we have $\Pi_t = \text{Id}_{X_0}$ for all $t$, and, given $p, q\in [1,+\infty]$, a function $u\in L^p_{L^r}$ has weak time derivative $\dot u\in L^q_{L^1}$ if 

$$  \int_0^T \int_{\Gamma(t)} \dot u \eta 
= - \int_0^T \int_{\Gamma(t)} u \dot \eta, ~~ \forall \eta\in \mathcal D_{L^\infty}. $$

The evolving space equivalence property for $\W^{p,q}(L^r, L^1)$ follows immediately. The same is true if we take $X(t)=W^{k,r}(\Gamma(t))$, $Y(t)=W^{k,1}(\Gamma(t))$, for general $k\in \N$.

\subsubsection{$W^{1,r}(\Gamma(t))\hookrightarrow L^1(\Gamma(t))$}

It is clear that both pairs $(X(t), \phi_t^X)_t$ and $(Y(t), \phi_t^Y)_t$ are compatible and it is easy to check that the dual map of $\phi_t^Y$ and its inverse are given by
\begin{align*}
    (\phi_{-t}^Y)^* \colon L^{\infty}(\Gamma_0)\to L^{\infty}(\Gamma(t)), \quad &(\phi_{-t}^Y)^* v = \widetilde J_0^t \phi_t^Y v = \widetilde J_0^t \, v\circ \widetilde\Phi_0^t, \\
   (\phi_t^Y)^*\colon L^{\infty}(\Gamma(t))\to L^{\infty}(\Gamma_0), \quad &(\phi_t^Y)^* v = \widetilde J_t^0 \, \phi_{-t}^Y v = \widetilde J_t^0 \, v\circ\widetilde \Phi_t^0.
\end{align*}
Recall that
\begin{align*}
    \pi(t; u, v) = \left\langle \Pi_t u, v\right\rangle_{L^1(\Gamma_0), \, L^\infty(\Gamma_0)} = \int_{\Gamma_0} \Pi_t u v,
\end{align*}
and since $ \Pi_t = \phi_{-t}^Y \phi_t^X $, we have by the chain rule
\begin{align}\label{eq:difficultchainrule}
\begin{split}
    \dfrac{d}{dt}\left( \Pi_t u\right)=  \dfrac{d}{dt} \left( \phi_{-t}^Y \phi_t^X u\right) &= \phi_{-t}^Y\phi_t^X \sgrad u\cdot \left[ \phi_{-t}^Y(\partial_t\Phi_0^t) + \phi_{-t}^Y(\mathbf D\Phi_0^t) \phi_{-t}^Y \widetilde{\mathbf w}(t) \right] \\
    &= \phi_{-t}^Y \phi_t^X \sgrad u \cdot \phi_{-t}^Y \left[\partial_t\Phi_0^t + \mathbf D\Phi_0^t \widetilde{\mathbf w}(t) \right]
\end{split}
\end{align}
from where we identify 
\begin{align*}
    \hat\lambda(t; u,v) = \dfrac{d}{dt} \pi(t; u, v) = \int_{\Gamma_0} \phi_{-t}^Y \phi_t^X \sgrad u \cdot \phi_{-t}^Y \left[\partial_t\Phi_0^t + \mathbf D\Phi_0^t \, \widetilde{\mathbf w}(t) \right] v.
\end{align*}
Pushing forward then yields, for $u\in X(t)$ and $v\in Y^*(t)$,
\begin{align}
    \lambda(t; u,v) 
    &= \int_{\Gamma(t)} \phi_t^X(\mathbf D\Phi^0_t)^\T \sgrad u \cdot \left( \partial_t\Phi_0^t  +   \mathbf D\Phi_0^t\widetilde{\mathbf w}_\tau(t)\right) v \nonumber.
\end{align}
Let us assume that $\Phi^{(\cdot)}_0 \in C^1([0,T];L^2(\mathbb{R}^d; \mathbb{R}^d))$.
Differentiating with respect to $t$ the identity
\begin{align*}
\Phi_0^t \circ \Phi_t^0 (p) = p, \quad p\in \Gamma_0,    
\end{align*}
we obtain, for all $p\in \Gamma_0$,
\begin{align*}
    0 = (\partial_t \Phi_0^t) (\Phi_t^0(p)) + \mathbf D \Phi_0^t (\Phi_t^0(p)) \partial_t \Phi_t^0(p) = (\partial_t \Phi_0^t) (\Phi_t^0(p)) + \mathbf D \Phi_0^t (\Phi_t^0(p)) \mathbf w (t, \Phi_t^0(p)).
\end{align*}
Pushing forward to $\Gamma(t)$ the above is equivalent to
\begin{align*} 
    \partial_t \Phi_0^t + \mathbf D \Phi_0^t \, \mathbf w  = 0,
\end{align*}
which we plug into the expression above to find
\begin{align*}
    \lambda(t;u,v)&= \int_{\Gamma(t)} \left(\phi_t^X(\mathbf D\Phi^0_t)^\T \sgrad u\right) \cdot \left( \mathbf D\Phi_0^t\widetilde{\mathbf w}_\tau(t)-\mathbf D \Phi_0^t\mathbf w(t)  \right)  v \\
    &= \int_{\Gamma(t)} \left(\phi_t^X(\mathbf D\Phi^0_t)^\T \sgrad u\right) \cdot  \left(\mathbf D\Phi_0^t\left(\widetilde{\mathbf w}_\tau(t)- \mathbf w_\tau(t)\right)  \right)  v.
\end{align*}
We then conclude that, given $p, q\in [1,+\infty]$, a function $u\in L^p_{W^{1,r}}$ has weak time derivative $\dot u\in L^q_{L^1}$ if 
\begin{align*}
    \int_0^T \int_{\Gamma(t)} \dot u  \eta = - \int_0^T \int_{\Gamma(t)} u \dot \eta \, - \int_0^T \int_{\Gamma(t)}\left(\phi_t^X(\mathbf D\Phi^0_t)^\T \sgrad u\right) \cdot  \left(\mathbf D\Phi_0^t\left(\widetilde{\mathbf w}_\tau(t)- \mathbf w_\tau(t)\right)  \right) \eta \quad \forall \eta\in \mathcal D_{L^\infty}.
\end{align*}
We now aim to explore the conditions of Theorem \ref{thm:equivalence}. Assumption \ref{ass:generalcaseChanges} on the regularity of $\lambda$ is easily seen to be true from the expression of $\lambda$ above. The condition \eqref{ass:rangeOfJ} is also satisfied; in fact, it follows from the formula for $\lambda$ that, given $u\in W^{1,r}(\Gamma_0)$, $\hat\Lambda(t)u\in L^1(\Gamma_0)\subset \mathcal J_{L^1}(L^1(\Gamma_0))$. So we are left to verify the remaining conditions stated in Theorem \ref{thm:equivalence}. We note that all the operators involved can be calculated explicitly. Indeed, we have 
\begin{align*}
    \Pi_t \colon W^{1,r}(\Gamma_0) \to L^1(\Gamma_0), \quad & \Pi_t u = u \circ \Phi_0^t \circ \widetilde \Phi_t^0, \\
    \overline\Pi_t \colon L^1(\Gamma_0)\to L^1(\Gamma_0), \quad &\overline\Pi_t u = u \circ \Phi_0^t \circ \widetilde \Phi_t^0, \\
    \overline \Pi_t^{-1}\colon L^1(\Gamma_0)\to L^1(\Gamma_0), \quad &\overline \Pi_t^{-1}u = u \circ \widetilde \Phi_0^t \circ \Phi_t^0,
\end{align*}
which can easily be seen to satisfy \eqref{ass:onrangeTt}, \eqref{ass:measurabilityOnTt}, \eqref{ass:TtInvExistsAndBdd}, \eqref{ass:measurabilityOnTtInv}, as well as the adjoints
\begin{align*}
    \overline \Pi_t ^* \colon L^\infty(\Gamma_0)\to L^\infty(\Gamma_0), \quad &\overline \Pi_t ^* v = J_t^0 \, \widetilde J_0^t\circ \Phi_t^0 \, v \circ \widetilde \Phi_0^t \circ \Phi_t^0, \\
    (\overline \Pi_t ^*)^{-1} \colon L^\infty(\Gamma_0)\to L^\infty(\Gamma_0), \quad &(\overline \Pi_t ^*)^{-1} v = J_0^t\circ \widetilde\Phi_t^0\, \widetilde J_t^0 \, v \circ \Phi_0^t \circ \widetilde\Phi_t^0,
\end{align*}
which satisfy \eqref{ass:remainsInSpace}. It is important to observe that the adjoint operator above is calculated as the $L^1$-$L^\infty$ adjoint, which simplifies its explicit expression (see example below for a more involved case). It then follows that the space $\mathbb W^{p,q}(W^{1,r}, L^1)$ enjoys the evolving space equivalence property.

\subsubsection{$W_0^{2,r}(\Omega(t))\hookrightarrow W_0^{1,1}(\Omega(t))$}

We need to find the adjoint of $\phi_{t}^Y$: given $u\in W^{1,1}_0(\Omega_0)$ and $\underline{\mathbf v}\in W^{-1,\infty}(\Omega_0)$, we have
\begin{align*}
    \left\langle \underline{\mathbf v}, \phi_{t}^Y u\right\rangle_{W^{-1, \infty}(\Omega(t)), W^{1,1}_0(\Omega(t))} &= -\int_{\Omega(t)} \underline{\mathbf v}\cdot \nabla (\phi_{t}^Y u) 
    = - \int_{\Omega_0} \widetilde J_t^0 (\mathbf D\widetilde\Phi_t^0)^{-1} \phi_{-t}^Y \,\underline{\mathbf{v}} \cdot \nabla u
\end{align*}
from where we conclude 
\begin{align*}
    (\phi_{t}^Y)^*\colon W^{-1,\infty}(\Omega(t)) \to W^{-1, \infty}(\Omega_0), \quad (\phi_{t}^Y)^* \underline{\mathbf v} = \widetilde J_t^0 (\mathbf D\widetilde\Phi_t^0)^{-1} \phi_{-t}^Y \, \underline{\mathbf v}.
\end{align*}
We then have 
\begin{align*}
     \pi(t; u, \underline{\mathbf v}) = \int_{\Omega_0} \underline{\mathbf v}\cdot \nabla \Pi_t u,
\end{align*}
and, recalling the formula in \eqref{eq:difficultchainrule}, this leads to 

\begin{align*}
    \hat\lambda(t; u, \underline{\mathbf{v}}) = \dfrac{d}{dt} \pi(t; u, \vv) = \int_{\Omega_0} \vv \cdot \nabla \left( \dfrac{d}{dt}\Pi_t u \right) = \int_{\Omega_0} \vv \cdot \nabla \left( \phi_{-t}^Y \phi_t^X \sgrad u \cdot \phi_{-t}^Y \left[\partial_t\Phi_0^t + \mathbf D\Phi_0^t \widetilde{w}(t) \right] \right).
\end{align*}
We finally push forward to time $t$: given $u\in W_0^{1,r}(\Omega(t))$, $\vv\in W^{-1,\infty}(\Omega(t))$, 
\begin{align*}
    \lambda(t; u, \vv) 
    &= \int_{\Omega(t)} \vv \cdot \nabla \left( (\mathbf D \Phi_0^t)^{-\T} \nabla u \cdot \left[\partial_t\Phi_0^t + \mathbf D\Phi_0^t \widetilde{w}(t) \right] \right).
\end{align*}
Again assuming that $\Phi^{(\cdot)}_0 \in C^1([0,T];L^2(\mathbb{R}^d; \mathbb{R}^d))$, we can reason as in the previous example to obtain
\begin{align*}
    \partial_t \Phi_0^t + \mathbf D\Phi_0^t \mathbf w =0,
\end{align*}
so that
\begin{align}\label{eq:lambdaex3}
    \lambda(t; u, \vv) = \int_{\Omega(t)} \vv \cdot \nabla \left( (\mathbf D \Phi_0^t)^{-\T} \nabla u \cdot \mathbf D\Phi_0^t \left( \widetilde{\mathbf w}_\tau(t)-\mathbf w_\tau(t)  \right) \right).
\end{align}
Hence in this case, given $p, q\in [1,+\infty]$, a function $u\in L^p_{W^{2,r}}$ has weak time derivative $\dot u\in L^q_{W^{1,1}}$ if 
\begin{align*}
    \int_0^T \int_{\Omega(t)} \dot u \,  \underline{\pmb\eta} = - \int_0^T \int_{\Omega(t)} u \, \dot{\underline{\pmb\eta}} \, - \int_0^T \int_{\Omega(t)}\nabla \left( (\mathbf D \Phi_0^t)^{-\T} \nabla u \cdot \mathbf D\Phi_0^t \left( \widetilde{\mathbf w}(t)-\mathbf w(t)  \right) \right) \cdot \underline{\pmb\eta}, \quad \forall \underline{\pmb{\eta}}\in \mathcal D_{W^{-1, \infty}}.
\end{align*}
We now analyse the conditions for the evolving space equivalence. From \eqref{eq:lambdaex3} it follows that
\begin{align*}
    \Lambda(t) u(t) = (\mathbf D \Phi_0^t)^{-\T} \nabla u \cdot \mathbf D\Phi_0^t \left( \widetilde{\mathbf w}_\tau(t)-\mathbf w_\tau(t)\right),
\end{align*}
and thus Assumption \ref{ass:generalcaseChanges} is satisfied. Equation \eqref{ass:rangeOfJ} also holds since, for any $u\in W^{2,r}_0(\Omega_0)$, we have $\hat\Lambda(t) u\in W^{1,1}_0(\Omega_0)\subset \mathcal J_{W_0^{1,1}}(W^{1,1}_0(\Omega_0))$. Now, as in the previous case, the extension of $\Pi_t$ to the larger space is trivial:
\begin{align*}
    \Pi_t \colon W^{2,r}_0(\Omega_0) \to W^{1,1}_0(\Omega), \quad & \Pi_t u = u \circ \Phi_0^t \circ \widetilde \Phi_t^0, \\
    \overline\Pi_t \colon W^{1,1}_0(\Omega_0)\to W_0^{1,1}(\Omega_0), \quad &\overline\Pi_t u = u \circ \Phi_0^t \circ \widetilde \Phi_t^0, \\
    \overline \Pi_t^{-1}\colon W^{1,1}_0(\Omega_0)\to W^{1,1}_0(\Omega_0), \quad &\overline \Pi_t^{-1}u = u \circ \widetilde \Phi_0^t \circ \Phi_t^0,
\end{align*}
which can easily be seen to satisfy \eqref{ass:onrangeTt}, \eqref{ass:measurabilityOnTt}, \eqref{ass:TtInvExistsAndBdd}, \eqref{ass:measurabilityOnTtInv}. We now work to identify the adjoint $$\bar\Pi_t^*\colon W^{-1,\infty}(\Omega_0) \to W^{-1,\infty}(\Omega_0).$$
Let $\underline{\mathbf v}\in W^{-1,\infty}(\Omega_0)$ and $u\in W^{1,1}_0(\Omega_0)$. A careful application of the formulas at the beginning of this chapter shows that 
\begin{align*}
    \langle \vv, \bar\Pi_t u\rangle_{W^{-1,\infty}, \, W^{1,1}_0} 
    &= -\int_{\Omega_0} \left(\left( \widetilde J_0^t\circ\Phi_t^0\right) J_t^0 \left[(\mathbf A_t^0)^{-\mathrm T} \, \mathbf D\widetilde\Phi_t^0\circ \widetilde \Phi_0^t \circ \Phi_t^0 \, \mathbf D\Phi_t^0\right] \underline{\mathbf v} \circ \widetilde\Phi_0^t \circ \Phi_t^0 \right) \cdot \nabla u
\end{align*}
from where we identify 
\begin{align*}
    \bar\Pi_t^*\colon W^{-1,\infty}(\Omega_0) \to W^{-1,\infty}(\Omega_0), \quad \bar\Pi_t^*\vv = \left( \widetilde J_0^t\circ\Phi_t^0\right) J_t^0 \left[(\mathbf A_t^0)^{-\mathrm T} \, \mathbf D\widetilde\Phi_t^0\circ \widetilde \Phi_0^t \circ \Phi_t^0 \, \mathbf D\Phi_t^0\right] \underline{\mathbf v} \circ \widetilde\Phi_0^t \circ \Phi_t^0,
\end{align*}
which is invertible with
\begin{align*}
    (\bar\Pi_t^*)^{-1}\colon W^{-1,\infty}(\Omega_0) \to W^{-1,\infty}(\Omega_0), \quad (\bar\Pi_t^*)^{-1}\vv = 
    \tilde J_t^0 \, J_0^t\circ \Phi_t^0 \, \left[\mathbf D\widetilde\Phi_t^0\circ \widetilde \Phi_0^t \circ \Phi_t^0 \, \mathbf D\Phi_t^0\right]^{-1} \underline{\mathbf v} \circ \Phi_0^t \circ \widetilde\Phi_t^0.
\end{align*}
Under our setting, because the coefficient is uniformly bounded in $t$ again it is easily checked that \eqref{ass:remainsInSpace} is also satisfied. This implies that $\mathbb W^{p,q}(W^{1,r}_0(\Omega_0), W^{1,1}(\Omega_0))$ enjoys the evolving space equivalence property.


\section{Well-posedness for a nonlinear monotone equation}\label{sec:application}
In this final section, we establish some results regarding existence and uniqueness of weak solutions for a class of nonlinear equations in order to illustrate the applicability of the functional framework developed in the text and how it can be used to formulate general problems in a Banach space setting. 


Let $p \in (1,\infty)$. For generality, we consider a family of (not necessarily linear) operators $A(t)\colon X(t) \to X^*(t)$ defined on a separable, reflexive Banach space $X(t)$ satisfying the following properties: for all $u,v \in X(t)$,
\begin{enumerate}[label=(\roman*)]
    \item (Measurability) the map $t \mapsto \langle A(t)u, v \rangle_{X^*(t),X(t)}$ is measurable\vspace{-.1cm}
    \item (Monotonicity) $\left\langle A(t)u - A(t) v, u-v\right\rangle_{X^*(t), \, X(t)}\geq 0$\vspace{-.1cm}
    \item (Hemicontinuity) the map $s \mapsto \langle A(t)(u+sv), v \rangle_{X^*(t), X(t)}$  is continuous (from $\mathbb{R}$ to $\mathbb{R}$)\vspace{-.1cm}
    \item (Boundedness) there exists a constant $C_b>0$ independent of $t$ and $c_b \in L^{p'}(0,T)$ such that 
        \begin{align*}
            \left|\left\langle A(t)u, v\right\rangle_{X^*(t), \, X(t)}\right| \leq C_b \|u\|^{p-1}_{X(t)} \|v\|_{X(t)} + c_b(t)\|v\|_{X(t)}
        \end{align*}
        \vspace{-.5cm}
    \item (Coercivity) there exist $C_c>0$ and $c_c \geq 0$ independent of $t$ such that
        \begin{align*}
            \left\langle A(t)u, u\right\rangle_{X^*(t), \, X(t)} \geq C_c \|u\|_{X(t)}^p - c_c.
            \end{align*}
    \end{enumerate}
These are the standard assumptions that are made for nonlinear monotone problems \cite[\S 30.2]{MR1033498}. We assume a Gelfand triple structure
\[X(t) \subset H(t) \subset X^*(t)\] 
and we suppose that $X(t)$ and $H(t)$ are evolving under a map $\phi_t$  and $X^*(t)$ is evolving under the dual map $\phi_{-t}^*$, such that 
\begin{align*}
    (X(t), \phi_t)_{t \in [0,T]}, \quad (H(t), \phi_t)_{t \in [0,T]}, \quad (X^*(t), \phi_{-t}^*)_{t \in [0,T]}
\end{align*}
are all compatible pairs. Furthermore, we assume the equivalence of $\W(X,X^*)$ and $\mathcal{W}(X_0,X_0^*)$. We refer to the previous section for examples of such spaces and proofs of the evolving space equivalence.

Defining the superposition operator $(Au)(t) = A(t)u(t)$, we consider the equation
\begin{equation}\label{eq:generalEquation}
\begin{aligned}
    \dot u + Au +  \Lambda u  &= f &&\text{in $L^{p'}_{X^*}$}, \\
    u(0) &= u_0 &&\text{in $H_0$}.
\end{aligned}
\end{equation}
\begin{defn}\label{def:weaksolution}
    Given $f\in L^{p'}_{X^*}$ and $u_0\in H_0$, a \textit{weak solution} of \eqref{eq:generalEquation} is a function $u\in \W^{p,p'}(X, X^*)$ satisfying 
    \begin{equation*}
        \begin{aligned}
         \int_0^T \langle \dot u(t), v(t) \rangle_{X^*(t), \, X(t)} + \langle A(t)u(t), v(t) \rangle_{X^*(t), \, X(t)} + \lambda(t;u(t), v(t))
         &= \int_0^T\langle f(t), v(t) \rangle_{X^*(t), \, X(t)} \qquad \forall  v\in L^{p}_{X},\\
         u(0)&=u_0.
    \end{aligned}
    \end{equation*}
\end{defn}
Our aim in this section is to prove the next result.
\begin{theorem}\label{thm:application}
Under the above assumptions (i)-(v), given $f\in L^{p'}_{X^*}$ and $u_0\in H_0$, there exists a unique weak solution $u \in \W^{p,p'}(X,X^*)$ to \eqref{eq:generalEquation}.
\end{theorem}
A concrete example of \eqref{eq:generalEquation} is the  evolutionary $p$-Laplace equation as the next example demonstrates. Aside from this, equations of the form \eqref{eq:generalEquation} may arise as regularisations of  PDEs with a more complicated structure. 
\begin{eg}[The $p$-Laplace equation on an evolving surface/domain]
Let $p\in (1,\infty)$ and take $\mathcal{M}(t)$ to be an evolving surface $\Gamma(t)\subset \R^3$ or domain $\Omega(t) \subset \R^2$ under the same regularity assumptions as in \ref{ass:evolv_spaces_examples}. Define
\begin{align*}
    X(t) = \begin{cases}
    W^{1,p}_0(\mathcal{M}(t)) &: \text{if $\mathcal{M}(t)=\Omega(t)$}\\
    W^{1,p}(\mathcal{M}(t)) &: \text{if $\mathcal{M}(t)=\Gamma(t)$}
    \end{cases}
\end{align*}
and the $p$-Laplace operator $-\Delta_{g(t)}^p\colon X(t) \to X^*(t)$ which has the action
\begin{align*}
    \langle -\Delta_{g(t)}^p u, v \rangle_{X^*(t), X(t)} := \int_{\mathcal{M}(t)} |\grad_{g(t)} u|^{p-2}\grad_{g(t)} u \cdot \grad_{g(t)} v.
\end{align*}
Take a constant $\alpha > 0$\footnote{The choice of $\alpha=0$ is also possible if $\mathcal{M}(t)=\Omega(t)$.}. We consider the equation\footnote{Observe that selecting $p=2$ and $\alpha=0$ recovers the heat equation.}
\begin{align*}
\begin{split}
    \dot u - \Delta_g^p u + \alpha u|u|^{p-2} + u \sgrad\cdot\mathbf{w}  &= f, \\
    u(0) &= u_0,
\end{split}
\end{align*}
which, if $\mathcal{M}(t)=\Omega(t)$, we supplement with the boundary condition $u = 0$ on $\partial\Omega(t)$. The operator $A$ is defined by 
$A(t)(u) := \alpha u |u|^{p-2} - \Delta_{g(t)}^p u.$ 

We have the Gelfand triple structure 
\begin{align*}
    X(t) \subset L^2(\mathcal{M}(t)) \subset X^*(t).
\end{align*}
This is obvious if $p\geq 2$, and in the case $1\leq p < 2$, recalling that the dimension of the manifold is $2$, it follows from the Sobolev embedding 
\begin{align*}
    W^{1,p}(\mathcal{M}(t)) \hookrightarrow L^{2p/(2-p)}(\mathcal{M}(t)) \hookrightarrow L^2(\mathcal{M}(t)).
\end{align*}
Evidently, the pivot space is $H(t)=L^2(\mathcal{M}(t))$ and we are in the setting of \S \ref{sec:L2pivotspace} from where we identify the extra term in the definition of the time derivative to be $\Lambda(t) u = u \sgrad\cdot\mathbf{w}$ and we also have the evolving space equivalence property.
\end{eg}

\subsection{Proof of well-posedness}
We now proceed to establish existence, uniqueness and stability of weak solutions via the Faedo--Galerkin method. We start by choosing an orthogonal basis $\{w_j^0\}_{j\in \N}$ for $X_0$ and transport it along the flow to $\{w_j^t:= \phi_t w_j^0\}_{j\in \N}$, which forms a basis for $X(t)$ satisfying the following useful property
\begin{align*}
\dot w_j^t \equiv 0 \quad \forall j\in \N.
\end{align*}
Define the approximation spaces
\begin{align*}
V_n(t) = \mathrm{span}\{w_1^t, \dots, w_n^t\} \quad \text{ and } \quad L^p_{V_n} := \left\{\eta\in L^p_{X} \colon \eta(t)\in V_n(t), \,\, t\in [0,T]\right\}.
\end{align*}
It follows that $\cup_n L^p_{V_n}$ is dense in $L^p_{X}$. We also make use of the projection operator $P_n^t\colon H(t)\to V_n(t) \subset X(t)$ determined by the formula
\begin{align*}
(P_n^t h - h, \varphi)_{H(t)} = 0 \quad \text{ for all }  \varphi\in V_n(t).
\end{align*}
\begin{lem}\label{lem:existenceForGalerkin}
For each $n\in \N$, there exists a unique solution $u_n\in L^p_{V_n}$ to the Galerkin approximation
\begin{equation}\label{eq:approx_p_laplace}
    \begin{aligned}
    \langle \dot u_n(t), v(t) \rangle_{X^*(t), \, X(t)}  + \langle A(t)u_n(t), v(t) \rangle_{X^*(t), \, X(t)}  + \lambda(t;u_n(t),v(t))  &= \langle f(t), v(t) \rangle_{X^*(t), \, X(t)} \quad \forall v\in L^p_{V_n}, \\
    u_n(0)&=P_n^0 u_0,
\end{aligned}
\end{equation}
of the form 
    $u_n(t) = \sum_{j=1}^n u_j^n(t) w_j^t.$
\end{lem}
The proof of this lemma is standard and is relegated to the appendix.

\paragraph{A priori estimates}

Test $v=u_n$ in \eqref{eq:approx_p_laplace} to obtain, using Young's inequality with $\epsilon$ and coercivity of the operator $A$,
\begin{align*}
    \dfrac{1}{2} \dfrac{d}{dt} \|u_n(t)\|^2_{H(t)}+ C_c\|u_n(t)\|_{X(t)}^p
    &\leq C_1 \|f(t)\|_{X^*(t)}^{p'} + \epsilon\|u_n(t)\|_{X(t)}^p + C_2 \|u_n(t)\|^2_{H(t)} + c_c.
\end{align*}
Choosing $\epsilon = C_c/2$ we can manipulate the above and integrate it to get 
\begin{align*}
    \|u_n(t)\|^2_{H(t)} + C_c \int_0^t \|u_n\|^p_{X(t)} \leq \|u_0\|_{H_0}^2 + 2 C_1 \|f\|_{L^{p'}_{X^*}}^{p'} + 2C_2 \int_0^t \|u_n\|^2_{H(t)}+ 2T c_c,
\end{align*}
whence an application of Gronwall's inequality implies that 
\begin{align}
    (u_n)_n \,\, &\text{ is uniformly bounded in } L^\infty_{H} \cap L^p_{X},\label{eq:uniformBoundForU}\\
    \nonumber (u_n(T))_n \,\, &\text{ is uniformly bounded in }  H(T).
\end{align}
Observe also that, due to H\"older's inequality, we have for all $\eta\in L^p_X$
\begin{align*}
    \left| \int_0^T \left\langle A u_n, \eta\right\rangle_{X^*(t), X(t)} \right| &\leq  \int_0^T C_b\|u_n\|_X^{p-1} \|\eta\|_{X} +  c_b(t)\|\eta\|_{X(t)} \leq C_b \|u_n\|_{L^p_X} \|\eta\|_{L^p_X} + \norm{c_b}{L^{p'}(0,T)}\norm{\eta}{L^p_X},
\end{align*}
whence
\begin{align}\label{eq:uniformBoundForNonlinearTerm}
    \left(A u_n\right)_n \,\, &\text{ is uniformly bounded in } \,\, L^{p'}_{X^*}.
\end{align}
\begin{remark}
It is an open question whether $P_n^t\colon V_n(t) \to V_n(t)$ is bounded uniformly in $n$ when $t>0$. Without an affirmative answer, it becomes more challenging to obtain a bound on $\dot u_n$ in the dual space $L^p_{X^*}$ by the usual duality method but it is typically still possible by pulling back the equation onto the reference space and using the boundedness of $P_n^0$.
\end{remark}

\paragraph{Existence, uniqueness, and stability of weak solutions}

For clarity of the argument, we proceed with several separate results. We start by identifying the limits of the approximating sequences. The bounds in \eqref{eq:uniformBoundForU}--\eqref{eq:uniformBoundForNonlinearTerm} give the existence $u\in L^p_{X} \cap L^\infty_H$, $z\in H(T)$ and $\chi\in L^{p'}_{X^*}$ such that, up to a subsequence,
\begin{align*}
    u_n\overset{*}{\rightharpoonup} u & \text{ in } L^\infty_{H}, \quad u_n \rightharpoonup u \text{ in } L^p_{X}, \quad u_n(T) \rightharpoonup z \text{ in } H(T), \quad\text{and}\quad A u_n \rightharpoonup \chi \text{ in } L^{p'}_{X^*}.
\end{align*}
    %
We use these to pass to the limit in the approximating equations \eqref{eq:approx_p_laplace}.

\begin{prop}
The limit function $u\in \W^{p,p'}(X, X^*) \cap C^0_H$ satisfies 
\begin{align*}
    \dot u + \chi + \Lambda u &= f \quad \text{ in } L^{p'}_{X^*},\\
    u(0) &= u_0,\\
    u(T) &= z.
\end{align*}
\end{prop}

\begin{proof}
For any $v\in \W^{p,p}(V_n, V_n)$ we can integrate by parts in \eqref{eq:approx_p_laplace} and put the time derivative onto the test function:
\begin{align*}
    \dfrac{d}{dt} (u_n(t), v(t))_{H(t)} + \left\langle A u_n(t), v(t) \right\rangle_{X^*(t),\, X(t)} = \left\langle f(t), v(t) \right\rangle_{X^*(t),\, X(t)} + (u_n(t), \dot v(t))_{H(t)}.
    \end{align*}
For $j \leq n$, take $v(t) = \psi(t)w_j^t$ with $\psi \in C^1([0,T])$, which clearly satisfies $v \in \W^{p,p}(V_n,V_n).$ Integrating over time 
and then passing to the limit $n \to \infty$, we obtain 
\begin{align}
\begin{split}\label{eq:initial_cond_12}
    (z, \psi(T)w_j^T)_{H(T)} - (u_0, \psi(0)w_j)_{H_0}  + \int_0^T \left\langle \chi(t), \psi(t)w_j^t \right\rangle_{X^*(t),\, X(t)} &= \int_0^T \left\langle f(t), \psi(t)w_j^t \right\rangle_{X^*(t), \, X(t)}  \\
    &\quad + \int_0^T (u(t), \psi'(t)w_j^t)_{H(t)}.
\end{split}
\end{align}
Since $\{w_j^0\}$ is a basis for $X_0,$ given $v \in X_0$, there exist coefficients $a_j \in \mathbb{R}$ and a sequence $v_n = \sum_{j=1}^n a_jw_j^0$ such that $v_n \to v$ in $X_0.$ Hence $\phi_t v_n = \sum_{j=1}^n a_jw_j^t$ converges to $\phi_t v$ in $X(t).$ Multiplying the above displayed equality by $a_j$ and summing up $j=1, ..., n$ gives
\begin{align*}
    (z, \psi(T)\phi_T v_n)_{H(T)} - (u_0, \psi(0)v_n)_{H_0}  + \int_0^T \left\langle \chi(t), \psi(t)\phi_t v_n \right\rangle_{X^*(t),\, X(t)} &= \int_0^T \left\langle f(t), \psi(t)\phi_t v_n \right\rangle_{X^*(t), \, X(t)}\\
    &\quad+ \int_0^T (u(t), \psi'(t)\phi_t v_n)_{H(t)}.
\end{align*}
Take furthermore $\psi \in \mathcal{D}(0,T)$. Passing to the limit $n \to \infty$ by using the dominated convergence theorem, we obtain
\begin{align*}
    \int_0^T \left\langle \chi(t), \psi(t)\phi_t v \right\rangle_{X^*(t),\, X(t)} = \int_0^T \left\langle f(t), \psi(t)\phi_t v \right\rangle_{X^*(t), \, X(t)} + \int_0^T (u(t), \psi'(t)\phi_t v)_{H(t)}.
\end{align*}
This is exactly the statement
\[\frac{d}{dt}(u(t),\phi_t v)_{H(t)} = \langle f(t)-\chi(t), \phi_t v \rangle_{X^*(t),X(t)}\qquad \forall v \in X_0.\]
Hence, by the characterisation offered in Proposition \ref{prop:temamTypeGT}, 
it follows that $u\in\W^{p,p'}(X,X^*)$ with 
\begin{align*}
    \dot u + \Lambda u - \chi = f
\end{align*}
as desired. The fact that $u\in C^0_H$ follows from the continuous embedding $\W^{p,p'}(X,X^*)\hookrightarrow C^0_{H}$.


To check the initial condition, let $v\in\W^{p,p'}(X, X^*)$. Using  the transport formula in Theorem \ref{thm:transport} and the equation for $u$, we have
\begin{align*}
\begin{split}
    (u(T),v(T))_{H(T)} -  (u(0), v(0))_{H_0} &= \int_0^T \left\langle \dot u(t) + \Lambda(t)u(t), v(t) \right\rangle_{X^*(t), \, X(t)} + \left\langle \dot v(t), u(t) \right\rangle_{X^*(t), \, X(t)} \\
    &= \int_0^T \left\langle \dot v(t), u(t) \right\rangle_{X^*(t), \, X(t)} + \left\langle f(t) -\chi(t), v(t) \right\rangle_{X^*(t), \, X(t)}.
\end{split}
\end{align*}
Taking $v(t) = \psi(t)w_j^t$ for arbitrary $j \in \N$ and $\psi \in C^1([0,T])$, this becomes
\begin{align*}
    (u(T),\psi(T)w_j^T)_{H(T)} - (u(0), \psi(0)w_j)_{H_0} 
    &= \int_0^T \left\langle \psi'(t)w_j^t, u(t) \right\rangle_{X^*(t), X(t)} + \left\langle f(t) -\chi(t), \psi(t)w_j^t \right\rangle_{X^*(t),  X(t)}.
\end{align*}
Comparing this with \eqref{eq:initial_cond_12}, we see that
\[(u(T),\psi(T)w_j^T)_{H(T)} - (u(0), \psi(0)w_j)_{H_0}  = (z, \psi(T)w_j^T)_{H(T)} - (u_0, \psi(0)w_j)_{H_0}.\]
Picking $\psi$ such that $\psi(T) = 0$ removes the first term on both sides and then the density of $\{w_j^0\}$ in $H_0$ implies that $u(0) = u_0$. A similar argument gives the final condition.
\end{proof}

The final step for existence is now to identify the nonlinear term $\chi$ in the equation. In the classical setting, the proof of the statement relies on the monotonicity of the operator $A$, see for example \cite[Lemma 30.6]{MR1033498}. In our case, the presence of $\Lambda$ in the equation (and integration by parts formulae) means that in general, the elliptic operator $A+\Lambda$ is non-monotone. However, as $\Lambda$ is a lower order term, we are able to mitigate its effects by using an exponential scaling trick. 

\begin{prop}
We have $\chi = A u$ in $L^{p'}_{X^*}$.
\end{prop}
\begin{proof}

Let us define, for $\gamma>0$ to be chosen later, the functions 
\begin{align*}
    v(t) = e^{-\gamma t} u(t) \quad \text{ and } \quad v_n(t) = e^{-\gamma t} u_n(t).
\end{align*}
Since $e^{-\gamma t}$ belongs to $L^\infty(0,T)$,  $v_n\in L^p_{V_n}$ and we have 
\begin{align*}
    v_n\overset{*}{\rightharpoonup} v \quad \text{ in } L^\infty_{H} \quad\text{and}\quad v_n\rightharpoonup v \quad \text{ in } L^p_{X}.
\end{align*}
Define also $\chi_\gamma(t) = e^{-\gamma t } \chi(t)$ and $A_\gamma(t) \xi = e^{-\gamma t} A(t) e^{\gamma t} \xi$, which is still a monotone operator. Noting that 
\begin{align*}
    \dot v_n(t) = -\gamma v_n(t) + e^{-\gamma t} \dot u_n(t) \quad \text{ and } \quad \dot v(t) = -\gamma v(t) + e^{-\gamma t} \dot u(t),
\end{align*}
it follows that the new approximations $(v_n)_n$ and the function $v$ satisfy 
\begin{align}
    \nonumber \langle \dot v_n + A_\gamma v_n, \eta \rangle_{X^*(t), \, X(t)} + \langle (\gamma + \Lambda) v_n,  \eta\rangle_{X^*(t),X(t)}  &= \langle e^{-\gamma t} f, v \rangle_{X^*(t), \, X(t)} \qquad \forall \eta\in L^p_{V_n},\\
    \langle \dot v + \chi_\gamma, \eta \rangle_{X^*(t), \, X(t)} + \langle (\gamma + \Lambda) v,  \eta\rangle_{X^*(t),X(t)} &= \langle e^{-\gamma t} f, v \rangle_{X^*(t), \, X(t)}\qquad \forall \eta\in L^p_X.\label{eq:v_p_laplace}
\end{align}
Now define 
$$\mathcal{L}_\gamma(t) \colon X(t)\to X^*(t)\quad \text{ by } \quad \left\langle \mathcal{L}_\gamma(t) v, \eta \right\rangle_{X^*(t), \, X(t)} =  \dfrac{1}{2}\langle (2\gamma + \Lambda) v,  \eta\rangle_{X^*(t),X(t)}.$$ 
We choose the constant $\gamma$ in such a way that $\mathcal{L}_\gamma$ is monotone (in the case of the $p$-Laplace equation, any $\gamma$ satisfying $2\gamma \geq \|\sgrad\cdot\mathbf{w}\|_{L^\infty}$ works, and in general such a choice is possible due to Assumption \ref{ass:generalcaseChanges} (iii), see also the third condition in Remark \ref{rem:generalcase}). Now, on the one hand, testing \eqref{eq:v_p_laplace} with $\eta=v$ leads to 
\begin{align*}
    \dfrac{1}{2}\dfrac{d}{dt} \|v\|^2_{H(t)} + \left\langle \chi_\gamma + \mathcal{L}_\gamma v, v\right\rangle_{X^*(t), \, X(t)} = \left\langle e^{-\gamma t} f, v \right\rangle_{X^*(t), \, X(t)},
\end{align*}
which we integrate over $[0,T]$ to obtain 
\begin{align}\label{eq:chi_eq1}
    \int_0^T \left\langle \chi_\gamma+\mathcal{L}_\gamma v, v\right\rangle_{X^*(t), \, X(t)}  = \int_0^T \left\langle e^{-\gamma t}f, v \right\rangle_{X^*(t), \, X(t)} + \dfrac{\|u_0\|^2_{H_0}}{2} - \dfrac{e^{-\gamma T}\|u(T)\|^2_{H(T)}}{2}.
\end{align}
On the other hand, the same calculation for the approximation $v_n$ now gives
\begin{align*}
    \int_0^T \left\langle  (\mathcal L_\gamma+A_\gamma) v_n, v_n \right\rangle_{X^*(t), \, X(t)} = \int_0^T\left\langle e^{-\gamma t} f, v_n \right\rangle_{X^*(t), \, X(t)} + \dfrac{\|P_n u_0\|^2_{H_0}}{2} - \dfrac{e^{-\gamma T}\|u_n(T)\|^2_{H(T)}}{2},
\end{align*}
whence taking the limit superior and
using the weak lower-semicontinuity of norms,
we obtain
\begin{align}\label{eq:chi_eq2}
    \limsup_n \int_0^T \left\langle (A_\gamma + \mathcal L_\gamma) v_n), v_n \right\rangle_{X^*(t), \, X(t)} \leq \int_0^T\left\langle e^{-\gamma t} f, v \right\rangle_{X^*(t), \, X(t)} + \dfrac{\|u_0\|^2_{H_0}}{2} - \dfrac{e^{-\gamma T}\|u(T)\|^2_{H(T)}}{2}.
\end{align}
Combining \eqref{eq:chi_eq1} with \eqref{eq:chi_eq2} then gives
\begin{align}\label{eq:chi_eq3}
    \int_0^T \left\langle \chi_\gamma + \mathcal L_\gamma v, v\right\rangle_{X^*(t), \, X(t)} \geq \limsup_n \int_0^T \left\langle (A_\gamma + \mathcal L_\gamma) v_n, v_n \right\rangle_{X^*(t), \, X(t)}.
\end{align}
Now take an arbitrary $\eta\in L^p_{X}$. Monotonicity of $A_\gamma + \mathcal L_\gamma$ implies that 
    $\left\langle (A_\gamma+\mathcal L_\gamma)(v_n) - (A_\gamma+\mathcal L_\gamma)(\eta), v_n-\eta \right\rangle \geq 0,$ 
which we can expand to obtain 
\begin{align*}
    \int_0^T \left\langle (A_\gamma + \mathcal L_\gamma) v_n, v_n \right\rangle_{X^*(t), \, X(t)} \geq \left\langle (A_\gamma + \mathcal L_\gamma) v_n, \eta \right\rangle_{X^*(t), \, X(t)} + \left\langle (A_\gamma + \mathcal L_\gamma) \eta, v_n-\eta \right\rangle_{X^*(t), \, X(t)}.
\end{align*}
Taking the limit superior and using \eqref{eq:chi_eq3} on the left-hand side and the convergence results on the right-hand side, we get
\begin{align*}
    \int_0^T \left\langle \chi_\gamma + \mathcal L_\gamma v, v\right\rangle_{X^*(t), \, X(t)} \geq \left\langle \chi_\gamma + \mathcal L_\gamma v, \eta \right\rangle_{X^*(t), \, X(t)} + \left\langle (A_\gamma + \mathcal L_\gamma) \eta, v-\eta \right\rangle_{X^*(t), \, X(t)}.
\end{align*}
This reads
\begin{align*}
    \int_0^T \left\langle \chi_\gamma + \mathcal{L}_\gamma v -A_\gamma \eta - \mathcal{L}_\gamma \eta , v-\eta\right\rangle_{X(t), \, X^*(t)} \geq 0.
\end{align*}
To conclude the proof we apply the well-known Minty's monotonicity trick, which gives $\chi_\gamma = A_\gamma v$ and hence $\chi = Au$. 
\end{proof}

All in all, combining the previous results shows that the limit function $u$ is indeed a weak solution as per Definition \ref{def:weaksolution}. Finally, the result below establishes stability of solutions with respect to initial conditions and uniqueness follows as a consequence, concluding the proof of Theorem \ref{thm:application}.

\begin{prop}
If $u_1$ and $u_2$ are weak solutions of \eqref{eq:generalEquation} corresponding to initial data $u_{10}$ and $u_{20}$, then 
\begin{align*}
    \|u_1(t)-u_2(t)\|_{H(t)} \leq e^{C_\mathbf{w}t\slash 2}\|u_{10}-u_{20}\|_{H_0}.
\end{align*}
In particular, weak solutions are unique.
\end{prop}
\begin{proof}
By testing the equation for both $u_1$ and $u_2$ with $v=u_1-u_2$ and subtracting we obtain
\begin{align*}
    \dfrac{1}{2} \dfrac{d}{dt} \|u_1-u_2\|^2_{H(t)} + \left\langle Au_1-Au_2, u_1-u_2\right\rangle_{X^*(t), \, X(t)}
    &\leq \dfrac{C_\mathbf{w}}{2} \|u_1-u_2\|_{H(t)}^2.
\end{align*}
Monotonicity of $A$ implies that we can neglect the second term on the left-hand side and then  
Gronwall's inequality gives the result.
\end{proof}



\appendix
\section{Technical results}
\begin{lem}\label{lem:derivativeOfM}The derivative of $\mathbf{A}^0_t=J^0_t(\mathbf{D}\Phi^0_t)^{-1}(\mathbf{D}\Phi^0_t)^{-\T}$ satisfies
\begin{align*}
\partial_t \mathbf{A}^0_t &= \phi_{-t}(\grad \cdot \mathbf{w}(t))\mathbf{A}^0_t - J^0_t(\mathbf{D}\Phi^0_t)^{-1}(\phi_{-t}(\mathbf{D}\mathbf{w}(t)) + (\phi_{-t}(\mathbf{D}\mathbf{w}(t)))^{\T})(\mathbf{D}\Phi^0_t)^{-\T}.
\end{align*}
\end{lem}
\begin{proof}
To ease presentation, we define $\mathbf{D}_t = \mathbf{D}\Phi^0_t$. We begin with
\begin{align*}
    \partial_t \mathbf{A}^0_t 
    &= J^0_t\phi_{-t}(\grad \cdot \mathbf{w}(t))(\mathbf{D}_t)^{-1}(\mathbf{D}_t)^{-\T} + J^0_t\partial_t((\mathbf{D}_t)^{-1}(\mathbf{D}_t)^{-\T}).
\end{align*}
To simplify the second term, using the formula $(M^{-1})' = -M^{-1}M'M^{-1}$ for differentiating the inverse of a matrix $M$, and the identity
\[\partial_t \mathbf{D}_t = \phi_{-t}(\mathbf{D}\mathbf{w}(t))\mathbf{D}_t,\]
we get
\begin{align*}
    \partial_t((\mathbf{D}_t)^{-1}(\mathbf{D}_t)^{-\T}) 
     &= -(\mathbf{D}_t)^{-1}\phi_{-t}(\mathbf{D}\mathbf{w}(t))(\mathbf{D}_t)^{-\T} - (\mathbf{D}_t)^{-1}(\phi_{-t}(\mathbf{D}\mathbf{w}(t)))^{\T}(\mathbf{D}_t)^{-\T}.
\end{align*}
Hence
\begin{align*}
\partial_t \mathbf{A}^0_t &= J^0_t\phi_{-t}(\grad \cdot \mathbf{w}(t))(\mathbf{D}_t)^{-1}(\mathbf{D}_t)^{-\T} - J^0_t((\mathbf{D}_t)^{-1}\phi_{-t}(\mathbf{D}\mathbf{w}(t))(\mathbf{D}_t)^{-\T} + (\mathbf{D}_t)^{-1}(\phi_{-t}(\mathbf{D}\mathbf{w}(t)))^{\T}(\mathbf{D}_t)^{-\T})\\
&= \phi_{-t}(\grad \cdot \mathbf{w}(t))\mathbf{A}^0_t - J^0_t(\mathbf{D}_t)^{-1}(\phi_{-t}(\mathbf{D}\mathbf{w}(t)) + (\phi_{-t}(\mathbf{D}\mathbf{w}(t)))^{\T})(\mathbf{D}_t)^{-\T}.\qedhere
\end{align*}
\end{proof}
Let us now give an expression $\partial_t \mathbf{A}^0_t$ and see how it acts.
\begin{lem}\label{lem:actionOfMPrime}
For $v \in H^1_0(\Omega(t))$, we have the identity
\begin{align*}
    \int_{\Omega(t)}J^t_0\grad  v(t)^{\T}(\mathbf{D}\Phi^t_0)^{-1}\phi_t(\partial_t \mathbf{A}^0_t)(\mathbf{D}\Phi^t_0)^{-\T}\grad \psi &=  \int_{\Omega(t)}\grad v(t)^{\T}\grad \psi \grad \cdot \mathbf{w}(t) - \grad v(t)^{\T} (\mathbf{D}\mathbf{w}(t) + (\mathbf{D}\mathbf{w}(t))^{\T})\grad \psi
\end{align*}
for all 
$\psi \in H^1_0(\Gamma(t))$.
\end{lem}
\begin{proof}
The formula $\phi_t(\grad \phi_{-t} v) = \phi_t(\mathbf{D}\Phi^0_t)^{T}\grad v$ and Lemma \ref{lem:derivativeOfM} allows us to write
\begin{align*}
    \int_{\Omega_0}\grad \tilde v(t)^{\T}\partial_t \mathbf{A}^0_t\grad \varphi 
    &= \int_{\Omega(t)}\grad v(t)^{\T}\grad \psi \grad \cdot \mathbf{w}(t) - \grad v(t)^{\T} (\mathbf{D}\mathbf{w}(t) + (\mathbf{D}\mathbf{w}(t))^{\T})\grad \psi\qedhere.
\end{align*}
\end{proof}
\begin{proof}[Proof of Lemma \ref{lem:existenceForGalerkin}]
Denoting the solution vector $U_n(t) = (u_1^n(t), \dots, u_n^n(t)),$ 
the linear terms
\begin{align*}
    B(t)_{ij} = (w_i^t, w_j^t)_{H(t)}, \quad 
    G(t)_{ij} = \lambda(t;w_i^t,w_j^t), \quad 
    F(t)_j = \left\langle f(t), w_j^t\right\rangle_{X^*(t), \, X(t)},
\end{align*}
and the nonlinear term 
\begin{align*}
    A(t; U_n(t))_{j} = \left\langle A(t)u_n(t), w_j^t\right\rangle_{X^*(t), \, X(t)},
\end{align*}
the problem \eqref{eq:approx_p_laplace} is equivalent to the system of ODEs 
\begin{align*}
    B(t) \dot U_n(t) + A(t; U_n(t)) + G(t) U_n(t) = F(t)&, \\
    U_n(0) = (\alpha_1, \dots, \alpha_n)&, 
\end{align*}
where $\{\alpha_j\}$ are the coefficients of $P_n u_0 = \sum_{j=1}^n \alpha_j w_j^0.$ Since $B(t)$ is a Gram matrix (and hence invertible) and the operators defining the lower-order terms are measurable in time and continuous in `space', the conclusion follows from the classical Carath\'eodory existence theory.
\end{proof}

\section*{Acknowledgements}
AA was partially supported by the DFG through the DFG SPP 1962 Priority Programme \emph{Non-smooth and Complementarity-based Distributed Parameter Systems: Simulation and Hierarchical Optimization} within project 10. ADj was supported by the DFG under Germany’s Excellence Strategy – The Berlin Mathematics Research Center MATH+ and CRC 1114 “Scaling Cascades in Complex Systems''.
\bibliographystyle{abbrv}
\bibliography{biblio}

\end{document}